\documentclass[pdflatex,sn-mathphys-num]{sn-jnl}% Math and Physical Sciences Numbered Reference Style
\usepackage{graphicx}%
\usepackage{lmodern}
\usepackage{multirow}%
\usepackage{amsmath,amssymb,amsfonts}%
\usepackage{amsthm}%
\usepackage{mathrsfs}%
\usepackage[title]{appendix}%
\usepackage{xcolor}%
\usepackage{textcomp}%
\usepackage{manyfoot}%
\usepackage{booktabs}%
\usepackage{algorithm}%
\usepackage{algorithmicx}%
\usepackage{algpseudocode}%
\usepackage{listings}%
\usepackage{rotate}
\usepackage{multirow}
\usepackage{multicol}
\usepackage{comment}
\usepackage{bussproofs}
\usepackage{amssymb}
\usepackage{latexsym}
\usepackage{amssymb}
\usepackage{mathrsfs}
\usepackage{graphics}
\usepackage{stmaryrd} 
\usepackage{comment}
\usepackage{fancyhdr}
\usepackage{comment}
\usepackage{tikz-cd}
\usepackage{proof}
\usepackage{color}
\usepackage[title]{appendix}
  \makeatletter
\def\namedlabel#1#2{\begingroup
    #2%
    \def\@currentlabel{#2}%
    \phantomsection\label{#1}\endgroup
}
\makeatother

\theoremstyle{thmstyleone}%
\newtheorem{theorem}{Theorem}%  meant for continuous numbers
\newtheorem{lemma}{Lemma}
\newtheorem{corollary}{Corollary}
\newtheorem{proposition}[theorem]{Proposition}% 

\theoremstyle{thmstyletwo}%

\theoremstyle{thmstylethree}%
\newtheorem{definition}{Definition}%

\newcommand{\Bint}{\sf{2Int}}
\newcommand{\BintN}{\sf{N2Int}}
\newcommand{\BintNA}{\sf{N2Int^*}}
\newcommand{\SCintN}{\sf{SC2Int}}

\newcommand{\Bes}{\sf{BeS}}
\newcommand{\At}{\sf{At}}

\newcommand{\B}{\mathcal{B}}
\newcommand{\C}{\mathcal{C}}

\begin{document}

\title[Bilateral base-extension semantics]{Bilateral base-extension semantics}


\author*[1,2]{\fnm{Victor} \sur{Barroso-Nascimento}}\email{v.nascimento@ucl.ac.uk}

\author*[2,3]{\fnm{Maria} \sur{Osório}}\email{oosorio.mariac@gmail.com}

\affil[1]{\orgdiv{Computer Science Department}, \orgname{University College London}}
\affil[2]{\orgdiv{Mathematics  Department}, \orgname{University of Lisbon}}


\abstract{Bilateralism is the position according to which assertion and rejection are conceptually independent speech acts. Logical bilateralism demands that systems of logic provide conditions for assertion and rejection that are not reducible to each other, which often leads to independent definitions of proof rules (for assertion) and dual proof rules, also called refutation rules (for rejection). Since it provides a critical account of what it means for something to be a proof or a refutation, bilateralism is often studied in the context of proof-theoretic semantics, an approach that aims to elucidate both the meaning of proofs (and refutations) and what kinds of semantics can be given if proofs (and refutations) are considered as basic semantic notions. The recent literature on bilateral proof-theoretic semantics has only dealt with the semantics of proofs and refutations, whereas we deal with semantics in terms of proofs and refutations.  In this paper we present a bilateral version of base-extension semantics---one of the most widely studied proof-theoretic semantics---by allowing atomic bases to contain both atomic proof rules and atomic refutation rules. The semantics is shown to be sound and complete with respect to the bilateral dual intuitionistic logic $\Bint$. Structural similarities between atomic proofs and refutations also allow us to define duality notions for atomic rules, deductions and bases, which may then be used for the proof of bilateral semantic harmony results. Aside from enabling embeddings between different fragments of the language, bilateral semantic harmony is shown to be a restatement of the syntactic horizontal inversion principle, whose meaning-conferring character may now be interpreted as the requirement of preservation of harmony notions already present at the core of the semantics by inferences.}




\maketitle

\section{Introduction}

The traditional Fregean account of negation depicts denial of a proposition as the mere assertion of its opposite \cite{Johnston2024-JOHWAF}. As Frege himself puts it:

\begin{quote}
    To each thought there corresponds an opposite, so
that rejecting one of them coincides with accepting the other. To make a judgement is to make a choice between opposite thoughts. Accepting one of them and rejecting the other is one act. So there is no need of a special sign for rejecting a thought. We only need a special sign for negation as such. \cite[pg. 198]{Frege1979-FREPW-2}
\end{quote}

Logic has traditionally embraced the Fregean view, dealing with rejection only indirectly through definition of assertion conditions for negation.

This view has been challenged by contemporary logicians, leading to the creation of what is now known as logical bilateralism \cite{Rumfitt2000-RUMYAN-2,Kurbis_2019ProofFalsity}. According to bilateralists, assertion and rejection are distinct speech acts that cannot be reduced to each other, so a proper account of logic should distinguish between conditions for propositional assertion and for propositional rejection \cite{Rumfitt2000-RUMYAN-2}.  Bilateralism was originally proposed as a justification of classical logic, but it is arguably tenable even from a intuitionistic viewpoint \cite{KurbisInt}. Philosophical underpinnings aside, explicit accounts of rejection allow the definition of natural deduction systems for classical logic that, in virtue of their harmonic rules, enjoy proof-theoretic properties previously ascribed only to intuitionistic natural deduction \cite{Rumfitt2000-RUMYAN-2,dummett1991logical}. Later generalizations of the bilateral perspective were also shown to yield a multilateral account \cite{Wansing2023-WANLMR} capable of distinguishing between multiple independent speech acts and an abstract account of logical entailment \cite{BlasioMarcose287a972-17c4-3a9d-98e3-d60c93c28ef8} which sheds light on questions as fundamental as those concerning differences between proof-theoretic and model-theoretic entailment.

One way of portraying the main philosophical difference between the Fregean view and bilateralism is to point out that both acknowledge the existence of a fundamental duality which allows the rejection of propositions, but they disagree about the conceptual level at which the duality is located. Frege puts it at the level of \textit{propositions}, arguing that each proposition has a dual, opposite proposition whose truth conditions obtain if and only if the truth conditions for the original proposition does not obtain, whence rejecting the original proposition is tantamount to asserting its dual. On the other hand, bilateralists put it at the \textit{assertoric} level, arguing that a proposition can be correctly asserted if and only if its assertion conditions obtain, and correctly rejected if and only if its rejection conditions obtain. If assertion and rejection conditions are identified with truth conditions (something acceptable in truth-functional accounts but not universally \cite{dummett1991logical}), this boils down to making an assertion correct (and a rejection incorrect) when the truth conditions of the asserted proposition obtain, and an assertion incorrect (and a rejection correct) when its truth conditions do not obtain.

From a proof-theoretic perspective, the distinction between assertion and rejection is important inasmuch it leads to a stronger distinction between \textit{proofs} and \textit{refutations}. Proofs are traditionally conceived as structures capable of guaranteeing epistemic grounds for assertion, and refutations as structures that do the same but for rejection. If rejections are nothing more than negative assertions then refutations are nothing more than negative proofs, but as soon as rejection emancipates itself from assertion the possibility of formulating an independent concept of refutation presents itself. This insight is especially important in the intuitionistic setting, which often takes the concept of proof to be one of its central philosophical notions \cite{sep-intuitionistic-logic-development, Troelstra1988-TROCIM, vanDalen1973-VANLOI-3}. If refutations are independent from proofs and intuitionistic logic only deals with proofs, the natural conclusion would be that it is an incomplete (or at least partial) logic. This claim is reinforced by the observation that development of the intuitionistically acceptable co-implication operator, studied extensively in bi-intuitionistic logics \cite{Rauszer1974, Gor2020BiIntuitionisticLA}, is historically tied to the development of constructive notions of duality  \cite{Binder03042017, Wolter1998-WOLOLW}.

%From a proof-theoretic perspective, the distinction between assertion and rejection is important inasmuch it leads to a stronger distinction between \textit{proofs} and \textit{refutations}. Since assertion and rejection are independent assertoric acts, proofs are means of guaranteeing epistemic grounds for assertion and refutations are means of guaranteeing epistemic grounds for rejection, \Maria{talvez full stop aqui? a parte seguinte da frase parece meio lost} \cyan{Victor: acho que a segunda parte \'e importante mas voc\^e tem raz\~ao que eu preciso conectar melhor.} then proofs and refutations should also be treated as independent notions. This insight is especially important in the intuitionistic setting, which often takes the concept of proof to be one of its central philosophical notions \cite{sep-intuitionistic-logic-development, Troelstra1988-TROCIM, vanDalen1973-VANLOI-3}.

 In this paper we focus exclusively on $\Bint$, a bilateral logic that deals with intuitionistic connectives and their dual refutational versions \cite{Wan110.1093/logcom/ext035, ayhan2020cutfree}. This logic is chosen due to the simplicity of its proof-theoretic characterization, which in turn makes it particularly valuable to the investigation of fundamental proof-theoretic questions such as those of synonymy between proofs and criteria for the uniqueness of logical connectives \cite{Ayhan2023-AYHOSI, ayhan2022uniquenesslogicalconnectivesbilateralist, Wansing2024}. For readers interested in further delving into bilateralism, brief but thorough overviews are presented at the beginning of \cite{Wan110.1093/logcom/ext035, Wansing2023-WANLMR}.

The aim of this paper is to close an important gap in the literature on bilateral proof-theoretic semantics. As pointed out by Schroeder-Heister \cite{sep-intuitionistic-logic-development}, proof-theoretic semantics encompasses both the study of the semantics of proofs (\textit{i. e.} of the substantive contents of proofs) and of semantics in terms of proofs (\textit{i. e.} frameworks in which proofs are considered semantic values), and sometimes both aspects end up intertwined (\textit{i.e.} the semantics of proofs can itself be given in terms of proofs). Up until now, the bilateralist literature has focused only on the study of the semantics of proofs, providing discussions on the relevance of bilateralism for the engagement with traditional semantic notions such as sense and denotation \cite{Ayhan2020-AYHWIT, Tranchini10.1093/logcom/exu028}. This paper, on the other hand, develops bilateralism in the second direction, providing the first semantic framework in which proofs and refutations are taken to be core semantic concepts. The bilateral proof-theoretic structure of the semantics is also shown to yield interesting properties in the form of harmony results, which establishes that important syntactic dualities observed between proofs and refutations \cite{francez2014bilateralism} have clear semantic counterparts -- that is, bilateral base-extension semantics is shown to be one of the frameworks mentioned by Schroeder-Heister in which the semantics of proofs can be given in terms of proofs.

\subsection{Base-extension semantics}

Base-extension semantics ($\Bes$) is one of the most well-studied frameworks of proof-theoretic semantics. Its main presentation is sound and complete with respect to intuitionistic logic \cite{Sandqvist2015IL}, but some restrictions also lead to a classical semantics \cite{Sandqvist}. $\Bes$ became especially prominent after some traditional validity concepts were shown to be incomplete with respect to intuitionistic logic \cite{piecha2015failure}.

We fix a countably infinite set of propositional atoms and call it $\At$; lowercase Latin letters ($p, q$) denote particular atoms; capital Greek letters with the subscript `$\At$' ($\Gamma_{\At}, \Delta_{\At}$) denote sets of atoms; lowercase Greek letters ($\phi, \psi$) denote formulas; capital Greek letters without the subscript `$\At$' ($\Gamma, \Delta$) denote sets of formulas; commas between sets denote set union. The definitions below also assume that the reader is familiar with traditional concepts of natural deduction \cite{prawitz1965}.

\begin{definition}[Atomic systems]\label{def:as}
An {\em atomic system} (a.k.a. a {\em base}) $\mathcal{B}$ is a (possibly empty) set of atomic rules of the form
\begin{prooftree}
\AxiomC{$[\Gamma^{1}_{\At}]$}
\noLine
\UnaryInfC{.}
\noLine
\UnaryInfC{.}
\noLine
\UnaryInfC{.}
\noLine
\UnaryInfC{$p_1$}
\AxiomC{$\ldots$}
\AxiomC{$[\Gamma^{n}_{\At}]$}
\noLine
\UnaryInfC{.}
\noLine
\UnaryInfC{.}
\noLine
\UnaryInfC{.}
\noLine
\UnaryInfC{$p_n$}
\TrinaryInfC{$p$}
\end{prooftree}
 The sequence $\langle p^{1}, ... , p^{n}\rangle$ of immediate premises in a rule can be empty -- in this case the rule is called an {\em atomic axiom}. The sets $\Gamma^{i}_{\At}$ can also be empty, in which case they are omitted.
\end{definition}

\begin{definition}[Extensions]\label{def:exts1}
An atomic system $\mathcal{C}$ is an {\em extension} of an atomic system $\mathcal{B}$ (written $\mathcal{C} \supseteq \mathcal{B}$), if $\mathcal{C}$ results from adding a (possibly empty) set of atomic rules to $\mathcal{B}$.   
\end{definition}

%If $S$ is an atomic system and $K$ is a set of sentence letters, we write $\overline{S(K)}$ for the {\em closure} of $K$ under all the rules in $S$ in the usual set-theoretic sense. An atomic system $S$ is {\em dense} if $\overline{S(\emptyset)}=\At$.

\begin{definition}[Deducibility]\label{def:deduc}
  For any $\Gamma_{\At}$, the consequence $\Gamma_{\At} \vdash_\mathcal{B} p$ holds if and only if there is a natural deduction derivation with conclusion $p$ and open premises $\Delta_{\At}$ (for some $\Delta_{\At} \subseteq \Gamma_{\At}$) that only uses the rules of $\mathcal{B}$.
\end{definition}

Notice that if  $p \in \{p_{1}, ... , p_{n}\}$ then $p_1, ..., p_{n} \vdash_{\mathcal{B}} p$ regardless of the $\mathcal{B}$ due to how deducibility is defined.

The syntactic notion of base derivability is extended to a semantic notion of \textit{validity} (sometimes called \textit{support}) as follows:

\begin{definition}\label{def:wvalidity} Validity in $\Bes$ is defined as follows:
    
\begin{enumerate}
    \item $\Vdash_{\mathcal{B}} p$ iff $\vdash_{\mathcal{B}} p$, for  $p \in\At$;

\medskip

    \item $\Vdash_{\mathcal{B}} (\phi \land \psi)$ iff $\Vdash_{\mathcal{B}} \phi$ and $\Vdash_{\mathcal{B}} \psi$;

\medskip      

     \item $\Vdash_{\mathcal{B}} (\phi \to \psi)$ iff $\phi \Vdash_{\mathcal{B}}  \psi$; 

\medskip       

      \item $\Vdash_{\mathcal{B}} \phi \lor \psi$ iff  $ \forall \mathcal{C} (\mathcal{C} \supseteq \mathcal{B}) $ and all $p \in \At$, $\phi \Vdash_{\mathcal{C}} p$ and $\psi \Vdash_{\mathcal{C}} p$ implies $\Vdash_{\mathcal{C}}  p$;

\medskip

\item $\Vdash_{\mathcal{B}} \bot$ iff $\Vdash_{\mathcal{B}} p$ for all $p \in \At$

\medskip

\item For any non-empty set of formulas $\Gamma$, $\Gamma \Vdash_{\mathcal{B}} A$ iff for all $\mathcal{C}$ such that $\mathcal{C} \supseteq \mathcal{B}$ it holds that, if $\Vdash_{\mathcal{C}} \phi$ for all $\phi \in \Gamma$, then $\Vdash_{\mathcal{C}} \psi$;

\medskip

\item $\Gamma \Vdash_{\Bes} A$ iff $\Gamma \Vdash_{\mathcal{B}} A$ for all $\mathcal{B}$;

\end{enumerate}
\end{definition}

From the facts that validity is defined inductively and the base clause reduces validity to atomic derivability it follows that validity of a formula (or consequence) in $\mathcal{B}$ is ultimately reducible to atomic derivability in $\mathcal{B}$. Hence, since the logical entailment relation $\Vdash_{\Bes}$ is defined by quantifying over all bases $\mathcal{B}$, logical validity is ultimately reducible to derivability in all bases, making this a semantic in terms of proofs.

\subsection{Syntax and (model-theoretic) semantics for $\Bint$}

The bilateral logic $\Bint$ \cite{Wan110.1093/logcom/ext035} relies on a distinction between \textit{proofs} and \textit{dual proofs}, which may also be (more suggestively) called \textit{proofs} and \textit{refutations}. In syntactic presentations this manifests itself in the form of two kinds of rules; in semantic presentations, in the form of two kinds of semantic clauses.

\subsubsection{The natural deduction systems $\BintN$ and $\BintNA$}

 Syntactic presentations of $\Bint$ distinguish between proof rules and refutation rules, as well as between proof assumptions (or simply assumptions) and refutation assumptions (or contra-assumptions). The dependencies of a deduction are represented by a pair $\Gamma ; \Delta$, in which $\Gamma$ is a set of proof assumptions and $\Delta$ a set of refutation assumptions. Proof rules are denoted by a single line; refutations rules by a double line.

%The natural deduction system $\BintN$ is given by the rules of Figures \ref{fig:Bint1} and \ref{fig:Bint2}. The natural deduction system $\BintNA$, which is a simple adaptation of $\BintN$, is given by the rules of Figures \ref{fig:Bint1} and \ref{fig:Bint2} but with  

Whenever the premise of a rule is preceded by a single line it must be either a proof assumption or a deduction ending with a proof rule, and whenever it is preceded by a double line it must be either a refutation assumption or a deduction ending with a refutation rule. The dotted lines in the rules for $\BintNA$ are placeholders for single and double lines, but any particular application of them is considered valid only if all dotted lines are replaced by single lines or all dotted lines are replaced by double lines.  When defining deductions inductively we stipulate that a single occurrences of a proof assumption $\phi$ already counts as a proof of $\phi$ depending on $\phi; \emptyset$, and a single occurrence of a refutation assumption $\phi$ already counts as a refutation of $\phi$ depending on $\emptyset ; \phi$. Proof and refutation assumptions are often numbered for the purpose of graphically representing the rule application that discharges them; an assumption with superscript $n$ is discharged by an application with $n$ on its right label. For a precise definition of all such notions, the reader is once again referred to \cite{Wan110.1093/logcom/ext035}. It is also worthy of notice that $\BintNA$ was already presented as a variant of the original $\BintN$ in \cite{AyhanRules10.1093/logcom/exae014, oddsson2025strongnegationdefinable2int}.

%Those proofs are certainly odd when we consider that they use operators entirely unrelated to the ones involved in the original derivation, but this is necessary due to the lack of other rules capable of converting proofs into refutations and vice-versa. For instance, $E_{2} \to (+)$ is the only proof rule with a proof as its premise and a refutation as its conclusion, so it is not surprising that it appears in proof concerning such transitions. As such, for some proof-theoretic endeavors it seems preferable to define those derivable rules as part of the original system, as done in \cite{AyhanRules10.1093/logcom/exae014, oddsson2025strongnegationdefinable2int}.

%The shape of the rules makes it clear that there is a kind of symmetry between the proof and refutation conditions for some connectives (\textit{e.g.} between proof conditions for $\phi \wedge \psi$ and refutation conditions for $\phi \vee \psi$), a fact that plays a significant role in the harmony results that will be shown later.

\bigskip
\bigskip

 Proof rules of the natural deduction system $\BintN$:

 \bigskip

\begin{prooftree}
\AxiomC{$\Gamma, [\phi]; \Delta$}
\noLine
\UnaryInfC{$\Pi$}
\UnaryInfC{$\psi$}
\RightLabel{\tiny{$I \to (+)$}}
\UnaryInfC{$\phi \to \psi$}
\DisplayProof
\qquad
\AxiomC{$\Gamma_1 ; \Delta_1$}
\noLine
\UnaryInfC{$\Pi_1$}
\UnaryInfC{$\phi \to \psi$}
\AxiomC{$\Gamma_2; \Delta_{2}$}
\noLine
\UnaryInfC{$\Pi_2$}
\UnaryInfC{$\phi$}
\RightLabel{\tiny{$E \to (+)$}}
\BinaryInfC{$\psi$}
\end{prooftree}

\begin{prooftree}
\AxiomC{$\Gamma; \Delta$}
\noLine
\UnaryInfC{$\Pi$}
\UnaryInfC{$\phi$}
\RightLabel{\tiny{$I_{1} \vee (+)$}}
\UnaryInfC{$\phi \vee \psi$}
\DisplayProof
\quad
\AxiomC{$\Gamma; \Delta$}
\noLine
\UnaryInfC{$\Pi$}
\UnaryInfC{$\psi$}
\RightLabel{\tiny{$I_{2} \vee (+)$}}
\UnaryInfC{$\phi \vee \psi$}
\DisplayProof
\quad
\AxiomC{$\Gamma_1; \Delta_{1}$}
\noLine
\UnaryInfC{$\Pi_1$}
\UnaryInfC{$\phi \vee \psi$}
\AxiomC{$\Gamma_2, [\phi]; \Delta_{2}$}
\noLine
\UnaryInfC{$\Pi_2$}
\UnaryInfC{$\chi$}
\AxiomC{$\Gamma_3, [\psi]; \Delta_{3}$}
\noLine
\UnaryInfC{$\Pi_3$}
\UnaryInfC{$\chi$}
\RightLabel{\tiny{$E\vee (+)$}}
\TrinaryInfC{$\chi$}
\end{prooftree}

\begin{prooftree}
\AxiomC{$\Gamma_1; \Delta_{1}$}
\noLine
\UnaryInfC{$\Pi_1$}
\UnaryInfC{$\phi$}
\AxiomC{$\Gamma_2; \Delta_{2}$}
\noLine
\UnaryInfC{$\Pi_2$}
\UnaryInfC{$\psi$}
\RightLabel{\tiny{$I \wedge (+)$}}
\BinaryInfC{$\phi \wedge \psi$}
\DisplayProof
\quad
\AxiomC{$\Gamma; \Delta$}
\noLine
\UnaryInfC{$\Pi$}
\UnaryInfC{$\phi \wedge \psi$}
\RightLabel{\tiny{$ E_{1} \wedge (+)$}}
\UnaryInfC{$\phi$}
\DisplayProof
\qquad
\AxiomC{$\Gamma; \Delta$}
\noLine
\UnaryInfC{$\Pi$}
\UnaryInfC{$\phi \wedge \psi$}
\RightLabel{\tiny{$E_{2} \wedge (+)$}}
\UnaryInfC{$\psi$}
\end{prooftree}

\begin{prooftree}
\AxiomC{$\Gamma_{1}; \Delta_{1}$}
\noLine
\UnaryInfC{$\Pi_{1}$}
\UnaryInfC{$\phi$}
\AxiomC{$\Gamma_{1}; \Delta_{1}$}
\noLine
\UnaryInfC{$\Pi_{1}$}
\doubleLine
\UnaryInfC{$\psi$}
\RightLabel{\tiny{$I \mapsfrom (+)$}}
\BinaryInfC{$\phi \mapsfrom \psi$}
\DisplayProof
\quad
\AxiomC{$\Gamma; \Delta$}
\noLine
\UnaryInfC{$\Pi$}
\UnaryInfC{$\phi \mapsfrom \psi$}
\RightLabel{\tiny{$ E_{1} \mapsfrom (+)$}}
\UnaryInfC{$\phi$}
\DisplayProof
\quad
\AxiomC{$\Gamma; \Delta$}
\noLine
\UnaryInfC{$\Pi$}
\UnaryInfC{$\phi \mapsfrom \psi$}
\doubleLine
\RightLabel{\tiny{$ E_{2} \mapsfrom (+)$}}
\UnaryInfC{$\psi$}
\end{prooftree}

\begin{prooftree}
    \AxiomC{$\Gamma; \Delta$}
\noLine
\UnaryInfC{$\Pi$}
\UnaryInfC{$\bot$}
\RightLabel{\tiny{$ \bot (+)$}}
\UnaryInfC{$\phi$}
\DisplayProof
\qquad
\AxiomC{}
\RightLabel{\tiny{$ \top (+)$}}
\UnaryInfC{$\top$}
\end{prooftree}

\bigskip
\bigskip

Refutation rules of the natural deduction system $\BintN$:

\bigskip

\begin{prooftree}
\AxiomC{$\Gamma; \Delta, \llbracket \psi \rrbracket$}
\noLine
\UnaryInfC{$\Pi$}
\doubleLine
\UnaryInfC{$\phi$}
\doubleLine
\RightLabel{\tiny{$I \mapsfrom (-)$}}
\UnaryInfC{$\phi \mapsfrom \psi$}
\DisplayProof
\qquad
\AxiomC{$\Gamma_1 ; \Delta_1$}
\noLine
\UnaryInfC{$\Pi_1$}
\doubleLine
\UnaryInfC{$\phi \mapsfrom \psi$}
\AxiomC{$\Gamma_2; \Delta_{2}$}
\noLine
\UnaryInfC{$\Pi_2$}
\doubleLine
\UnaryInfC{$\psi$}
\RightLabel{\tiny{$E \mapsfrom (-)$}}
\doubleLine
\BinaryInfC{$\phi$}
\end{prooftree}

\begin{prooftree}
\AxiomC{$\Gamma; \Delta$}
\noLine
\UnaryInfC{$\Pi$}
\doubleLine
\UnaryInfC{$\phi$}
\RightLabel{\tiny{$ I_{1} \wedge (-)$}}
\doubleLine
\UnaryInfC{$\phi \wedge \psi$}
\DisplayProof
\quad
\AxiomC{$\Gamma; \Delta$}
\noLine
\UnaryInfC{$\Pi$}
\doubleLine
\UnaryInfC{$\psi$}
\RightLabel{\tiny{$I_{2} \wedge (-)$}}
\doubleLine
\UnaryInfC{$\phi \wedge \psi$}
\DisplayProof
\quad
\AxiomC{$\Gamma_1; \Delta_{1}$}
\noLine
\UnaryInfC{$\Pi_1$}
\doubleLine
\UnaryInfC{$\phi \wedge \psi$}
\AxiomC{$\Gamma_2; \Delta_{2},  \llbracket \phi \rrbracket$}
\noLine
\UnaryInfC{$\Pi_2$}
\doubleLine
\UnaryInfC{$\chi$}
\AxiomC{$\Gamma_3\; \Delta_{3}, \llbracket \psi \rrbracket$}
\noLine
\UnaryInfC{$\Pi_3$}
\doubleLine
\UnaryInfC{$\chi$}
\RightLabel{\tiny{$ E \wedge (-)$}}
\doubleLine
\TrinaryInfC{$\chi$}
\end{prooftree}

\begin{prooftree}
\AxiomC{$\Gamma_1; \Delta_{1}$}
\noLine
\UnaryInfC{$\Pi_1$}
\doubleLine
\UnaryInfC{$\phi$}
\AxiomC{$\Gamma_2; \Delta_{2}$}
\noLine
\UnaryInfC{$\Pi_2$}
\doubleLine
\UnaryInfC{$\psi$}
\RightLabel{\tiny{$I \vee (-)$}}
\doubleLine
\BinaryInfC{$\phi \vee \psi$}
\DisplayProof
\quad
\AxiomC{$\Gamma; \Delta$}
\noLine
\UnaryInfC{$\Pi$}
\doubleLine
\UnaryInfC{$\phi \vee \psi$}
\RightLabel{\tiny{$E_{1} \wedge (-)$}}
\doubleLine
\UnaryInfC{$\phi$}
\DisplayProof
\qquad
\AxiomC{$\Gamma; \Delta$}
\noLine
\UnaryInfC{$\Pi$}
\doubleLine
\UnaryInfC{$\phi \vee \psi$}
\RightLabel{\tiny{$E_{2} \wedge (-)$}}
\doubleLine
\UnaryInfC{$\psi$}
\end{prooftree}

\begin{prooftree}
\AxiomC{$\Gamma_{1}; \Delta_{1}$}
\noLine
\UnaryInfC{$\Pi_{1}$}
\UnaryInfC{$\phi$}
\AxiomC{$\Gamma_{1}; \Delta_{1}$}
\noLine
\UnaryInfC{$\Pi_{1}$}
\doubleLine
\UnaryInfC{$\psi$}
\RightLabel{\tiny{$I\to (-)$}}
\doubleLine
\BinaryInfC{$\phi \to \psi$}
\DisplayProof
\quad
\AxiomC{$\Gamma; \Delta$}
\noLine
\UnaryInfC{$\Pi$}
\doubleLine
\UnaryInfC{$\phi \to \psi$}
\RightLabel{\tiny{$E_{1} \to (-)$}}
\UnaryInfC{$\phi$}
\DisplayProof
\quad
\AxiomC{$\Gamma; \Delta$}
\noLine
\UnaryInfC{$\Pi$}
\doubleLine
\UnaryInfC{$\phi \to \psi$}
\doubleLine
\RightLabel{\tiny{$E_{2} \to (-)$}}
\UnaryInfC{$\psi$}
\end{prooftree}

\begin{prooftree}
    \AxiomC{$\Gamma; \Delta$}
\noLine
\UnaryInfC{$\Pi$}
\doubleLine
\UnaryInfC{$\top$}
\RightLabel{\tiny{$ \top (-)$}}
\doubleLine
\UnaryInfC{$\phi$}
\DisplayProof
\qquad
\AxiomC{}
\doubleLine
\RightLabel{\tiny{$ \bot (-)$}}
\UnaryInfC{$\bot$}
\end{prooftree}

\bigskip
\bigskip

Rules of $\BintNA$ replacing their corresponding versions in $\BintN$:

\bigskip

\begin{prooftree}
\AxiomC{$\Gamma_1; \Delta_{1}$}
\noLine
\UnaryInfC{$\Pi_1$}
\UnaryInfC{$\phi \vee \psi$}
\AxiomC{$\Gamma_2, [\phi]; \Delta_{2}$}
\noLine
\UnaryInfC{$\Pi_2$}
\dottedLine
\UnaryInfC{$\chi$}
\AxiomC{$\Gamma_3, [\psi]\; \Delta_{3}$}
\noLine
\UnaryInfC{$\Pi_3$}
\dottedLine
\UnaryInfC{$\chi$}
\RightLabel{\tiny{$E\vee (+)$}}
\dottedLine
\TrinaryInfC{$\chi$}
\DisplayProof
\qquad \qquad
  \AxiomC{$\Gamma; \Delta$}
\noLine
\UnaryInfC{$\Pi$}
\UnaryInfC{$\bot$}
\RightLabel{\tiny{$ \bot (+)$}}
\dottedLine
\UnaryInfC{$\phi$}
\end{prooftree}

\begin{prooftree}
\AxiomC{$\Gamma_1; \Delta_{1}$}
\noLine
\UnaryInfC{$\Pi_1$}
\doubleLine
\UnaryInfC{$\phi \wedge \psi$}
\AxiomC{$\Gamma_2; \Delta_{2},  \llbracket \phi \rrbracket$}
\noLine
\UnaryInfC{$\Pi_2$}
\dottedLine
\UnaryInfC{$\chi$}
\AxiomC{$\Gamma_3\; \Delta_{3}, \llbracket \psi \rrbracket$}
\noLine
\UnaryInfC{$\Pi_3$}
\dottedLine
\UnaryInfC{$\chi$}
\RightLabel{\tiny{$ E \wedge (-)$}}
\dottedLine
\TrinaryInfC{$\chi$}
\DisplayProof
\qquad \qquad
    \AxiomC{$\Gamma; \Delta$}
\noLine
\UnaryInfC{$\Pi$}
\doubleLine
\UnaryInfC{$\top$}
\RightLabel{\tiny{$ \top (-)$}}
\dottedLine
\UnaryInfC{$\phi$}
\end{prooftree}

\bigskip

Some properties of the natural deduction system $\BintN$ makes it significantly harder to use in the soundness and completeness proofs for our semantics. As such, we prove the results for $\BintNA$, whose syntactic deducibility relations are defined as follows:

\begin{definition}\label{def:deductionsoriginalproof}
    
   $\Gamma ; \Delta \vdash^{+}_{\BintNA} \phi$ holds iff there is a deduction of $\phi$ from open proof assumptions $\Theta$ and open refutation assumptions $\Sigma$ (for $\Theta \subseteq \Gamma$ and $\Sigma \subseteq \Delta$) using the rules of $\BintNA$ that ends with an application of a proof rule.

    %$\Gamma ; \Delta \vdash^{+}_{\BintN} \phi$ [$\Gamma ; \Delta \vdash^{-}_{\BintN} \phi$] holds iff there is a deduction of $\phi$ from open proof assumptions $\Theta$ and open refutation assumptions $\Sigma$ (for $\Theta \subseteq \Gamma$ and $\Sigma \subseteq \Delta$) using the rules of $\BintN$ that ends with an application of a proof rule [refutation rule].
    
    %$\Gamma ; \Delta \vdash^{+}_{\BintN} \phi$ [$\Gamma ; \Delta \vdash^{-}_{\BintN} \phi$] holds iff there is a deduction of $\phi$ from open proof assumptions $\Theta$ and open refutation assumptions $\Sigma$ (for $\Theta \subseteq \Gamma$ and $\Sigma \subseteq \Delta$) using the rules of $\BintN$ that ends with an application of a proof rule, and $\Gamma ; \Delta \vdash^{-}_{\BintN} \phi$ holds iff there is such a deduction that ends with an application of a refutation rule.
    
    %$\Gamma ; \Delta \vdash^{+}_{\BintNA} \phi$ holds iff such deductions use the rules of $\BintNA$ instead.
\end{definition}

\begin{definition}\label{def:deductionsoriginalrefutation}

  $\Gamma ; \Delta \vdash^{-}_{\BintNA} \phi$ holds iff there is a deduction of $\phi$ from open proof assumptions $\Theta$ and open refutation assumptions $\Sigma$ (for $\Theta \subseteq \Gamma$ and $\Sigma \subseteq \Delta$) using the rules of $\BintNA$ that ends with an application of a refutation rule.
  
   %  $\Gamma ; \Delta \vdash^{+}_{\BintNA} \phi$ holds iff there is a deduction of $\phi$ from open proof assumptions $\Theta$ and open refutation assumptions $\Sigma$ (for $\Theta \subseteq \Gamma$ and $\Sigma \subseteq \Delta$) using the rules of $\BintNA$ that ends with an application of a proof rule, and $\Gamma ; \Delta \vdash^{-}_{\BintNA} \phi$ holds iff there is such a deduction that ends with an application of a refutation rule
    
    %$\Gamma ; \Delta \vdash^{+}_{\BintNA} \phi$ holds iff such deductions use the rules of $\BintNA$ instead.
\end{definition}

It is straightforward to show that, since the new inferences allowed in $\BintNA$ are already derivable in $\BintN$, the replacement does not affect deductive validity:

\begin{proposition} \label{propositionderivability} \normalfont
    The following rules are derivable in $\BintN$:
    \begin{prooftree}
     \AxiomC{$\Gamma_1; \Delta_{1}$}
\noLine
\UnaryInfC{$\Pi_1$}
\UnaryInfC{$\phi \vee \psi$}
\AxiomC{$\Gamma_2, [\phi]; \Delta_{2}$}
\noLine
\UnaryInfC{$\Pi_2$}
\doubleLine
\UnaryInfC{$\chi$}
\AxiomC{$\Gamma_3, [\psi] ; \Delta_{3}$}
\noLine
\UnaryInfC{$\Pi_3$}
\doubleLine
\UnaryInfC{$\chi$}
\doubleLine
\TrinaryInfC{$\chi$}
\DisplayProof
\qquad \qquad
  \AxiomC{$\Gamma; \Delta$}
\noLine
\UnaryInfC{$\Pi$}
\UnaryInfC{$\bot$}
\doubleLine
\UnaryInfC{$\phi$}
    \end{prooftree}

\begin{prooftree}
\AxiomC{$\Gamma_1; \Delta_{1}$}
\noLine
\UnaryInfC{$\Pi_1$}
\doubleLine
\UnaryInfC{$\phi \wedge \psi$}
\AxiomC{$\Gamma_2; \Delta_{2},  \llbracket \phi \rrbracket$}
\noLine
\UnaryInfC{$\Pi_2$}
\UnaryInfC{$\chi$}
\AxiomC{$\Gamma_3 ; \Delta_{3}, \llbracket \psi \rrbracket$}
\noLine
\UnaryInfC{$\Pi_3$}
\UnaryInfC{$\chi$}
\TrinaryInfC{$\chi$}
\DisplayProof
\qquad \qquad
  \AxiomC{$\Gamma; \Delta$}
\noLine
\UnaryInfC{$\Pi$}
\doubleLine
\UnaryInfC{$\top$}
\UnaryInfC{$\phi$}
\end{prooftree}
    
\end{proposition}

\begin{proof}

We show the result by relying on properties of $\mapsfrom$, $\to$, $\top$ and $\bot$:

\begin{prooftree}
\AxiomC{$\Gamma_{1}, \Delta_{1}$}
\noLine
        \UnaryInfC{$\Pi_{1}$}
    \UnaryInfC{$\phi \vee \psi$}
    \AxiomC{}
    \RightLabel{\tiny{$\top (+)$}}
    \UnaryInfC{$\top$}
    \AxiomC{$\Gamma_{2},[\phi]^1; \Delta$}
    \noLine
    \UnaryInfC{$\Pi_{2}$}
    \doubleLine
    \UnaryInfC{$\chi$}
      \RightLabel{\tiny{$I \mapsfrom (+)$}}
    \BinaryInfC{$\top \mapsfrom \chi$}
        \AxiomC{}
          \RightLabel{\tiny{$\top (+)$}}
    \UnaryInfC{$\top$}
        \AxiomC{$\Gamma_{3},[\psi]^2; \Delta$}
    \noLine
    \UnaryInfC{$\Pi_{3}$}
    \doubleLine
    \UnaryInfC{$\chi$}
          \RightLabel{\tiny{$I \mapsfrom (+)$}}
    \BinaryInfC{$\top \mapsfrom \chi$}
             \RightLabel{\tiny{$E \vee (+), 1, 2$}}
    \TrinaryInfC{$\top \mapsfrom \chi$}
    \doubleLine
                 \RightLabel{\tiny{$E_{2} \mapsfrom (+)$}}
    \UnaryInfC{$\chi$}
\end{prooftree}

\begin{prooftree}
\AxiomC{$\Gamma_{1}, \Delta_{1}$}
\noLine
        \UnaryInfC{$\Pi_{1}$}
        \doubleLine
    \UnaryInfC{$\phi \wedge\psi$}
    \AxiomC{$\Gamma_{2}; \Delta, \llbracket \phi \rrbracket^1$}
    \noLine
    \UnaryInfC{$\Pi_{2}$}
    \UnaryInfC{$\chi$}
        \AxiomC{}
    \doubleLine
    \RightLabel{\tiny{$\bot (-)$}}
    \UnaryInfC{$\bot$}
      \RightLabel{\tiny{$I \to (-)$}}
      \doubleLine
    \BinaryInfC{$\chi \to \bot$}
        \AxiomC{$\Gamma_{3}; \Delta, \llbracket \psi \rrbracket^2$}
    \noLine
    \UnaryInfC{$\Pi_{3}$}
    \UnaryInfC{$\chi$}
         \AxiomC{}
    \doubleLine
    \RightLabel{\tiny{$\bot (-)$}}
    \UnaryInfC{$\bot$}
          \RightLabel{\tiny{$I \to (-)$}}
          \doubleLine
    \BinaryInfC{$\chi \to \bot$}
             \RightLabel{\tiny{$E \wedge (-), 1, 2$}}
             \doubleLine
    \TrinaryInfC{$\chi \to \bot$}
                 \RightLabel{\tiny{$E_{1} \to (-)$}}
    \UnaryInfC{$\chi$}
\end{prooftree}

\begin{prooftree}
    \AxiomC{$\Pi$}
    \UnaryInfC{$\bot$}
    \RightLabel{\tiny{$\bot (+)$}}
    \UnaryInfC{$\psi \mapsfrom \phi$}
    \RightLabel{\tiny{$E_{2} \mapsfrom (+)$}}
    \doubleLine
    \UnaryInfC{$\phi$}
    \DisplayProof
    \qquad
     \AxiomC{$\Pi$}
     \doubleLine
    \UnaryInfC{$\top$}
    \doubleLine
    \RightLabel{\tiny{$\top (-)$}}
    \UnaryInfC{$\phi \to \psi$}
    \RightLabel{\tiny{$E_{1} \to (-)$}}
    \UnaryInfC{$\phi$}
\end{prooftree}

\end{proof}

Those proofs are certainly odd due to their use of operators unrelated to the ones involved in the original derivation, but such detours seem necessary due to the lack of other rules capable of converting proofs into refutations (and vice-versa). For instance, $E_{2} \to (+)$ is the only proof rule with a proof as its premise and a refutation as its conclusion, so it is not surprising that it appears in proof concerning such transitions.  

Although a normalization proof for $\BintN$ is presented in \cite{WANSING201723}, instead of showing the result directly, it proceeds by embedding $\BintN$ deductions in the traditional natural deduction system for intuitionistic logic and relying on normalizations proofs for it \cite{prawitz1965}.  The normal intuitionistic deductions in which the deductions of $\BintN$ are embedded are guaranteed to satisfy the subformula property, but the corresponding deductions of $\BintN$ are not. In fact, it seems that deductions in $\BintN$ cannot be guaranteed to always reduce to other deductions in $\BintN$ satisfying the subformula property. To see why, consider the following simple deduction of $\top$ from $\bot ; \emptyset$:

\begin{prooftree}
    \AxiomC{$\bot ; \emptyset$}
    \RightLabel{\tiny{$\bot(+)$}}
    \UnaryInfC{$\top \mapsfrom \top$}
    \doubleLine
    \RightLabel{\tiny{$E_{2} \mapsfrom (+)$}}
    \UnaryInfC{$\top$}
\end{prooftree}

Since $\top \mapsfrom \top$ is neither a subformula of an open premise nor of the conclusion, the deduction above does not satisfy the subformula property. If the rule $\bot(+)$ is allowed to produce refutations (as in $\BintNA$), we can eliminate the application of $E_{2} \mapsfrom (+)$ and the occurrence of $\top \mapsfrom \top$ by directly obtaining a refutation of $\top$ from the proof assumption $\bot ; \emptyset$, but if it is only allowed to produce proofs (as in $\BintN$) it seems that the application of $E_{2} \mapsfrom (+)$ cannot be eliminated. For the same reason, it does not seem to be the case that applications of $\bot (+)$ (and of $\top (-)$) in $\BintN$ can always be transformed into applications containing atomic conclusions, a result often used in proofs of the subformula property. It is also worth noticing that $\SCintN$, a sequent calculus presentation of $\Bint$ shown to satisfy cut elimination and the subformula property in \cite{Ayhan2023-AYHOSI}, is also defined with adaptations equivalent to those of $\BintNA$.

\subsubsection{Model-theoretic semantics for $\Bint$}

Model-theoretic semantics for $\Bint$ are defined in \cite{Wan110.1093/logcom/ext035}. The semantics is essentially a dual interpretation of the traditional notion of intuitionistic Kripke model \cite{kripke1965semantical}. Particular states $w$ in models $\mathcal{M}$ are allowed to support either the truth or the falsity of a statement (denoted respectively by $\mathcal{M}, w \vDash^{+} \phi$ and $\mathcal{M}, w \vDash^{-} \phi$). Validity in a model is defined by quantifying over all of its states, and logical validity by quantifying over all its models. It is also important to point out that the states of a model are allowed to simultaneously support the truth and falsity of the same formula.

\begin{definition} \label{def:Kripkemodelbilateral}
    A model $\mathcal{M}$ for $\Bint$ is a sequence $\langle I, \geq, v^{+}, v^{-} \rangle$, where $I$ is a set of objects $w$, $\geq$ is a relation on $I$ and $v^{+}$, $v^{-}$ are functions $v^{+}: I \to \At$ and $v^{-}: I \to \At$, satisfying the condition that $w' \geq w$ implies both $v^{+}(w') \supseteq v^{+}(w)$ and  $v^{-}(w') \supseteq v^{-}(w)$.
\end{definition}

\begin{definition} \label{def:modelvalidity}
 \normalfont   The relations $\vDash^{+}$ and $\vDash^{-}$ are defined as follows, with $\nvDash^{+}$ and $\nvDash^{-}$ being defined as holding if and only if, respectively,  $\vDash^{+}$ and $\vDash^{-}$ do not hold:

\begin{enumerate}
    \item[\namedlabel{model:1}{1.}] $\mathcal{M}, w \vDash^{+} p$ iff $p \in v^{+}(w)$, for $p \in \At$;

\medskip
    
      \item[\namedlabel{model:2}{2.}] $\mathcal{M}, w \vDash^{-} p$ iff $p \in v^{-}(w)$, for $p \in \At$;

\medskip 

    \item[\namedlabel{model:3}{3.}] $\mathcal{M}, w  \vDash^{+} \bot$ never holds;

\medskip

    \item[\namedlabel{model:4}{4.}] $\mathcal{M}, w \vDash^{-} \bot$ always holds;

\medskip

    \item[\namedlabel{model:5}{5.}] $\mathcal{M}, w \vDash^{+} \top$ always holds;

\medskip

   \item[\namedlabel{model:6}{6.}] $\mathcal{M}, w \vDash^{-} \top$ never holds;

\medskip

   \item[\namedlabel{model:7}{7.}] $\mathcal{M}, w \vDash^{+} \phi \land \psi$ iff $\mathcal{M}, w \vDash^{+} \phi$ and $\mathcal{M}, w \vDash^{+} \psi$;

\medskip

   \item[\namedlabel{model:8}{8.}] $\mathcal{M}, w \vDash^{-} \phi \land \psi$ iff  $\mathcal{M}, w \vDash^{-} \phi$ or $\mathcal{M}, w \vDash^{-} \psi$;

\medskip

   \item[\namedlabel{model:9}{9}] $\mathcal{M}, w \vDash^{+} \phi \lor \psi$ iff  $\mathcal{M}, w \vDash^{+} \phi$ or $\mathcal{M}, w \vDash^{+} \psi$;

\medskip

   \item[\namedlabel{model:10}{10.}]$\mathcal{M}, w\vDash^{-} \phi \lor \psi$ iff $\mathcal{M}, w\vDash^{-} \phi$ and $\mathcal{M}, w\vDash^{-} \psi$;

\medskip

   \item[\namedlabel{model:11}{11.}] $\mathcal{M}, w\vDash^{+} \phi \to \psi$ iff, for all $w' \geq w$, either $\mathcal{M}, w\nvDash^{+} \phi$ or $\mathcal{M}, w\vDash^{+} \psi$;

\medskip

   \item[\namedlabel{model:12}{12.}] $\mathcal{M}, w\vDash^{-} \phi \to \psi$ iff $\mathcal{M}, w\vDash^{+} \phi$ and $\mathcal{M}, w\vDash^{-} \psi$;

\medskip

   \item[\namedlabel{model:13}{13.}] $\mathcal{M}, w\vDash^{+} \phi \mapsfrom \psi$ iff $\mathcal{M}, w \vDash^{+} \phi$ and $ \mathcal{M}, w \vDash^{-} \psi$;

\medskip

   \item[\namedlabel{model:14}{14.}] $\mathcal{M}, w \vDash^{-} \phi \mapsfrom \psi$ iff, for all $w' \geq w$, either $\mathcal{M}, w\nvDash^{-} \psi$ or $\mathcal{M}, w\vDash^{-} \phi$;

\medskip

   \item[\namedlabel{model:15}{15.}] $\mathcal{M} \vDash^{+} \phi$ iff $\mathcal{M}, w \vDash^{+} \phi$ for all $w$ in the set $I$ of $\mathcal{M}$;

\medskip

   \item[\namedlabel{model:16}{16.}] $\mathcal{M} \vDash^{-} \phi$ iff $\mathcal{M}, w \vDash^{-} \phi$ for all $w$ in the set $I$ of $\mathcal{M}$

\medskip

   \item[\namedlabel{model:17}{17.}] $\vDash^{+} \phi$ iff $\mathcal{M} \vDash^{+} \phi$ for all $\mathcal{M}$;

\medskip

   \item[\namedlabel{model:18}{18.}]$\vDash^{-} \phi$ iff $\mathcal{M} \vDash^{-} \phi$ for all $\mathcal{M}$;

\end{enumerate}

\end{definition}

Clauses for negation $(\neg \phi)$ and co-negation ($-\phi$) are also defined in \cite{Wan110.1093/logcom/ext035}, but they are not strictly necessary because $\neg \phi$ is equivalent to $\phi \to \bot$ and $- \phi$ to $\top \mapsfrom \phi$. It is shown there that $\vDash^{+} \phi$ holds if and only if there is a deduction of $\phi$ in $\BintN$ from the empty set of premises that ends with a proof rule, and $\vDash^{-} \phi$ holds if and only if there is a deduction of $\phi$ in $\BintN$ from the empty set of premises that ends with a refutation rule. Moreover, for $\phi$ not containing co-implications in its subformulas, $\vDash^{+} \phi$ holds if and only if $\phi$ is derivable from empty premises in intuitionistic logic; for $\phi$ not containing implications in its subformulas, $\vDash^{-} \phi$ holds if and only if $\phi$ is derivable from empty premises in bi-intuitionistic logic. The restriction on subformulas is necessary because intuitionistic logic does not contain $\mapsfrom$ and dual intuitionistic logic does not contain $\to$ in their respective languages (because proof conditions for $\mapsfrom$ and refutation condutions for $\to$ both rely on definitions of duality absent in these logics).

Although entailment relations are not defined in \cite{Wan110.1093/logcom/ext035}, they could have been defined as follows \cite{Agudelo-Agudelo02012024}:

\begin{definition}
    Semantic entailment for $\Bint$ is defined as follows:

    \begin{enumerate}
        \item $\Gamma; \Delta \vDash^{+} \chi$ iff, for every $\mathcal{M}$ and $w$, if $\mathcal{M}, w \vDash^{+} \phi$ for all $\phi \in \Gamma$ and $\mathcal{M}, w \vDash^{-} \psi$ for all $\psi \in \Delta$  then $\mathcal{M}, w \vDash^{+} \phi$.
        \item $\Gamma; \Delta \vDash^{-} \chi$ iff, for every $\mathcal{M}$ and $w$, if $\mathcal{M}, w \vDash^{+} \phi$ for all $\phi \in \Gamma$ and $\mathcal{M}, w \vDash^{-} \psi$ for all $\psi \in \Delta$  then $\mathcal{M}, w \vDash^{-} \phi$.
    \end{enumerate}
\end{definition}

Such clauses are defined in \cite{Agudelo-Agudelo02012024} by considering signed formulas $+\phi$ and $-\phi$ and a single notion of support instead of two, but the definitions are easily seen to be equivalent. In fact, bilateral systems defined using signed formulas are just notational variants of systems defined with multiples rules or derivability relations  \cite{FrancezEquivalentlogics3020002}.

\section{Bilateral base-extension semantics}

\subsection{Basic definitions}

Since bilateralism diverges from traditional accounts by considering proofs and refutations as equally primitive, the natural first step towards a bilateral base-extension semantics is to define bases in which atomic proofs and refutations are primitive:

\begin{definition}[Bilateral atomic systems]\label{def:bilateralbase}
A {\em bilateral atomic system} (a.k.a. a {\em bilateral base}, or just {\em base}) $\mathcal{B}$ is a (possibly empty) set of atomic rules of the form
\begin{prooftree}
\AxiomC{$[\Gamma^{1}_{\At}]; \llbracket \Delta^{1}_{\At} \rrbracket$}
\noLine
\UnaryInfC{.}
\noLine
\UnaryInfC{.}
\noLine
\UnaryInfC{.}
\UnaryInfC{$p_1$}
\AxiomC{$\ldots$}
\AxiomC{$[\Gamma^{n}_{\At}]; \llbracket \Delta^{n}_{\At} \rrbracket$}
\noLine
\UnaryInfC{.}
\noLine
\UnaryInfC{.}
\noLine
\UnaryInfC{.}
\doubleLine
\UnaryInfC{$p_n$}
\TrinaryInfC{$p$}
\DisplayProof \qquad
\AxiomC{$[\Gamma^{1}_{\At}]; \llbracket \Delta^{1}_{\At} \rrbracket$}
\noLine
\UnaryInfC{.}
\noLine
\UnaryInfC{.}
\noLine
\UnaryInfC{.}
\UnaryInfC{$p_1$}
\AxiomC{$\ldots$}
\AxiomC{$[\Gamma^{n}_{\At}]; \llbracket \Delta^{n}_{\At} \rrbracket$}
\noLine
\UnaryInfC{.}
\noLine
\UnaryInfC{.}
\noLine
\UnaryInfC{.}
\doubleLine
\UnaryInfC{$p_n$}
\doubleLine
\TrinaryInfC{$p$}
\end{prooftree}
where $p_i,p\in\At$, $\Gamma^{i}_{\At}$ is a (possibly empty) set of atomic proof assumptions and  $\Delta^{i}_{\At}$ is a (possibly empty) set of atomic refutation assumptions. The sequence $\langle p^{1}, ... , p^{n}\rangle$ of premises in a rule can be empty -- in this case the rule is called an {\em atomic axiom}. The sets $\Gamma^{i}_{\At}, \Delta^{i}_{\At}$ can once again be omitted when empty. The atom $p$ is the rule's \textit{conclusion}.
\end{definition}

\begin{definition} \label{def:proofrulesrefutrulesandaxioms}
    An atomic rule is called a \textit{proof rule} if a single line appears above its conclusion, and a \textit{refutation rule} if a double line appear above it instead. Atomic axioms may also be called \textit{proof axioms} if they are proof rules and \textit{refutation axioms} otherwise.
\end{definition}

As in usual bilateralist definitions of deduction, whenever a single line appears above a rule premise $p^i$ the rule requires either a proof assumption with shape $p^i$ or a deduction whose last rule application is a proof rule and whose conclusion is $p^i$ to be applied, and whenever a double line appears above a rule premise $p^i$ the rule requires either a refutation assumption with shape $p^i$ or a deduction whose last rule application is a refutation rule and whose conclusion is $p^i$ to be applied.

Since from now on we only deal with bilateral bases, we will simply refer to them as \textit{bases} or \textit{atomic systems} for brevity. We also speak plainly of \textit{atomic rules} when referring to proof and refutation rules indistinctly. The same holds for the remaining definitions of bilateral base-extension semantics.

\begin{definition}[Extensions]\label{def:exts2}
An atomic system $\mathcal{C}$ is an {\em extension} of an atomic system $\mathcal{B}$ (written $\mathcal{C} \supseteq \mathcal{B}$), if $\mathcal{C}$ results from adding a (possibly empty) set of atomic rules to $\mathcal{B}$.   
\end{definition}

%If $S$ is an atomic system and $K$ is a set of sentence letters, we write $\overline{S(K)}$ for the {\em closure} of $K$ under all the rules in $S$ in the usual set-theoretic sense. An atomic system $S$ is {\em dense} if $\overline{S(\emptyset)}=\At$.

\begin{definition}[Proof deductions]\label{def:deducatomicproofs} The consequence $\Gamma_{\At}; \Delta_{\At} \vdash^{+}_\mathcal{B} p$ holds if and only if there is a natural deduction derivation with conclusion $p$ depending on the open atomic proof assumptions $\Theta_{\At}$ and open atomic refutation assumptions $\Sigma_{\At}$  (for some $\Theta_{\At} \subseteq \Gamma_{\At}$ and $\Sigma_{\At} \subseteq \Delta_{\At}$) that only uses the rules of $\mathcal{B}$, which either consists in a single occurrence of a proof assumption or whose last rule application is a proof rule.
\end{definition}

\begin{definition}[Refutation deductions]\label{def:deducatomicrefs} The consequence $\Gamma_{\At}; \Delta_{\At} \vdash^{-}_\mathcal{B} p$ holds if and only if there is a natural deduction derivation with conclusion $p$ depending on the open atomic proof assumptions $\Theta_{\At}$ and open atomic refutation assumptions $\Sigma_{\At}$  (for some $\Theta_{\At} \subseteq \Gamma_{\At}$ and $\Sigma_{\At} \subseteq \Delta_{\At}$) that only uses the rules of $\mathcal{B}$, which either consists in a single occurrence of a refutation assumption or whose last rule application is a refutation rule.
\end{definition}

\begin{definition}[Atomic proofs] An \textit{atomic proof} is a proof deduction depending on no proof or refutation assumptions. We denote by $\vdash^{+}_\mathcal{B} p$ the existence of an atomic proof of $p$ in $\mathcal{B}$.
    
\end{definition}

\begin{definition}[Atomic refutations] An \textit{atomic refutation} is a refutation deduction depending on no proof or refutation assumptions. We denote by $\vdash^{-}_\mathcal{B} p$ the existence of an atomic refutation of $p$ in $\mathcal{B}$.
    
\end{definition}

Once again, notice that that if  $p \in \{p_{1}, ... , p_{n}\}$ then $p_1, ..., p_{n}; \Gamma_{\At} \vdash^{+}_{\mathcal{B}} p$ and $\Gamma_{\At} ;  \{p_1, ..., p_{n} \}  \vdash^{-}_{\mathcal{B}} p$ regardless of the $\mathcal{B}$ and $\Gamma_{\At}$ due to how deducibility is defined.

To exemplify how the new definition works and fix some notations for deductions, consider a base $\mathcal{B}$ containing the following rules:

\begin{prooftree}
    \AxiomC{}
    \noLine
    \UnaryInfC{}
     \noLine
    \UnaryInfC{}
     \noLine
    \UnaryInfC{}
     \noLine
    \UnaryInfC{}
        \noLine
    \UnaryInfC{}
    \UnaryInfC{$p$}
    \AxiomC{}
     \noLine
    \UnaryInfC{}
     \noLine
    \UnaryInfC{}
     \noLine
    \UnaryInfC{}
     \noLine
    \UnaryInfC{}
        \noLine
    \UnaryInfC{}
    \doubleLine
    \UnaryInfC{$q$}
       \RightLabel{\tiny{$R1$}}
    \BinaryInfC{$r$}
    \DisplayProof
    \qquad
       \AxiomC{}
        \noLine
    \UnaryInfC{}
     \noLine
    \UnaryInfC{}
     \noLine
    \UnaryInfC{}
     \noLine
    \UnaryInfC{}
        \noLine
    \UnaryInfC{}
    \UnaryInfC{$p$}
    \AxiomC{}
     \noLine
    \UnaryInfC{}
     \noLine
    \UnaryInfC{}
     \noLine
    \UnaryInfC{}
     \noLine
    \UnaryInfC{}
        \noLine
    \UnaryInfC{}
    \doubleLine
    \UnaryInfC{$q$}
    \doubleLine
       \RightLabel{\tiny{$R2$}}
    \BinaryInfC{$r$}
    \DisplayProof
    \qquad
    \AxiomC{}
     \noLine
    \UnaryInfC{}
     \noLine
    \UnaryInfC{}
     \noLine
    \UnaryInfC{}
     \noLine
    \UnaryInfC{}
    \noLine
    \UnaryInfC{}
        \noLine
    \UnaryInfC{}
        \noLine
    \UnaryInfC{}
       \RightLabel{\tiny{$R3$}}
    \UnaryInfC{$p$}
    \DisplayProof
    \qquad
    \AxiomC{$\llbracket q \rrbracket$}
    \noLine
    \UnaryInfC{.}
    \noLine
    \UnaryInfC{.}
    \noLine
    \UnaryInfC{.}
    \UnaryInfC{$r$}
    \RightLabel{\tiny{$R4$}}
    \UnaryInfC{$s$}
\end{prooftree}

The following are deductions showing $\emptyset; q \vdash^{-}_{\mathcal{B}} r$ and $\emptyset; \emptyset \vdash^{+}_{\mathcal{B}} s$ (or simply $\vdash^{+}_{\mathcal{B}} s$), respectively:

\begin{prooftree}
    \AxiomC{}
    \RightLabel{\tiny{R3}}
    \UnaryInfC{$p$}
    \AxiomC{$\emptyset ; q$}
    \doubleLine
    \RightLabel{\tiny{$R2$}}
    \BinaryInfC{$r$}
    \DisplayProof
    \qquad \quad
      \AxiomC{}
    \RightLabel{\tiny{R3}}
    \UnaryInfC{$p$}
    \AxiomC{$\llbracket q \rrbracket^{1}$}
    \RightLabel{\tiny{$R1$}}
    \BinaryInfC{$r$}
    \RightLabel{\tiny{$R4$, 1}}
    \UnaryInfC{$s$}
\end{prooftree}

So open proof assumptions with shape $q$ are denoted by $q ; \emptyset$ and open refutation assumptions with shape $q$ are denoted by $\emptyset; q$, because if we simply wrote down the atom $q$ there would be no way of saying if it was a proof assumption or a refutation assumption. This is not necessary when we consider discharged assumptions because the brackets we use already specify whether it is a proof or a refutation assumption. 

Notice that, although we could represent proof assumptions and refutation assumptions by writing either a line or a double line above their occurrences, since bases might contain atomic axioms (such as the $R3$ depicted above) the notation for assumptions and atomic axioms would coincide, possibly leading to ambiguity (unless either rules or assumptions are always labeled to clarify their nature). We opt for this notation instead. If our definitions also took into account higher-order atomic rules capable of discharging other atomic rules \cite{Piecha2016, Peter} this would not be an issue because there would be no meaningful distinction between atomic axioms and atomic assumptions, but this would also lead to a significantly more complex notion of atomic deduction.

Once the basic definitions are established, support definitions for proofs and refutations in bases can be defined for all formulas through semantic clauses. Just like the model-theoretic relations $\mathcal{M}, w\vDash^{+} \phi$ and $\mathcal{M}, w\vDash^{-} \phi$ are respectively used to define support for the truth and falsity of $\phi$ in the state $w$ of the model $\mathcal{M}$, the relation $\Vdash^{+}_{\mathcal{B}}$ is meant to represent proof support in $\mathcal{B}$ ($\Vdash^{+}_{\mathcal{B}} \phi$ iff there is a proof of $\phi$ in $\mathcal{B}$), and $\Vdash^{-}_{\mathcal{B}}$ to represent refutation support in $\mathcal{B}$ ($\Vdash^{-}_{\mathcal{B}} \phi$ iff there is a refutation of $\phi$ in $\mathcal{B}$). Material entailment for a base $\mathcal{B}$ is defined by taking into account extensions of $\mathcal{B}$ capable of proving and refuting the formulas in the antecedent and whether they always yield a proof/refutation of the consequent or not.  Logical entailment relations $\Vdash^{+}$ and $\Vdash^{-}$ are defined by considering what can be proved/refuted in all bases (when the antecedent is empty) and which material consequences hold in all bases (when it is not).

%After establishing that $* \in \{+, -\}$ and that whenever $*$ appears in a clause all occurrences of it must be substituted by the same sign, we define validity as follows:

\begin{definition}[Bilateral validity]\label{def:bilateralvalidity}
 \normalfont   The relations $\Vdash_{\mathcal{B}}^{+}, \Vdash_{\mathcal{B}}^{-}, \Vdash^{+} and \Vdash^{-}$ are defined as follows:

\begin{enumerate}
    \item[\namedlabel{c:at+}{($\At +$)}] $\Vdash_{\mathcal{B}}^{+} p$ iff $\vdash^{+}_{\mathcal{B}} p$, for $p \in \At$;

\medskip
    
        \item[\namedlabel{c:at-}{$(\At -)$}] $\Vdash_{\mathcal{B}}^{-} p$ iff $\vdash^{-}_{\mathcal{B}} p$, for $p \in \At$;

\medskip 

\item[\namedlabel{c:bot+}{($\bot +$)}] $\Vdash_\mathcal{B}^{+} \bot$ iff $\Vdash_\mathcal{B}^{+} p$ and $\Vdash_\mathcal{B}^{-} p$ holds for every $p \in \At$;

\medskip

\item[\namedlabel{c:bot-}{($\bot -$)}] $\Vdash_\mathcal{B}^{-} \bot$ holds for all $\mathcal{B}$;

\medskip

\item[\namedlabel{c:top+}{($\top +$)}] $\Vdash_\mathcal{B}^{+} \top$ holds for all $\mathcal{B}$;

\medskip

\item[\namedlabel{c:top-}{($\top -$)}] $\Vdash_\mathcal{B}^{-} \top$ iff $\Vdash_\mathcal{B}^{+} p$ and $\Vdash_\mathcal{B}^{-} p$ holds for every $p \in \At$;

\medskip

\item[\namedlabel{c:and+}{($\land +$)}] $\Vdash_{\mathcal{B}}^{+} \phi \land \psi$ iff $\Vdash_{\mathcal{B}}^{+} \phi$ and $\Vdash_{\mathcal{B}}^{+} \psi$

\medskip

\item[\namedlabel{c:and-}{($\land -$)}] $\Vdash_{\mathcal{B}}^{-} \phi \land \psi$ iff  $ \forall \mathcal{C} (\mathcal{C} \supseteq \mathcal{B})$ and $\forall p (p \in \At)$ : ($ \emptyset ; \phi \Vdash_{\mathcal{C}}^{+} p$ and $\emptyset ;\psi \Vdash_{\mathcal{C}}^{+} p$ implies $\Vdash_{\mathcal{C}}^{+} p$) and ($ \emptyset ; \phi \Vdash_{\mathcal{C}}^{-} p$ and $\emptyset ;\psi \Vdash_{\mathcal{C}}^{-} p$ implies $\Vdash_{\mathcal{C}}^{-} p$);

\medskip

\item[\namedlabel{c:or+}{($\lor +$)}] $\Vdash_{\mathcal{B}}^{+} \phi \lor \psi$ iff  $ \forall \mathcal{C} (\mathcal{C} \supseteq \mathcal{B})$ and $\forall p (p \in \At)$ $:  (\phi ;\emptyset \Vdash_{\mathcal{C}}^{+} p$ and $\psi ; \emptyset \Vdash_{\mathcal{C}}^{+} p$ implies $\Vdash_{\mathcal{C}}^{+} p)$ and $(\phi ;\emptyset \Vdash_{\mathcal{C}}^{-} p$ and $\psi ; \emptyset \Vdash_{\mathcal{C}}^{-} p$ implies $\Vdash_{\mathcal{C}}^{-} p)$;
\medskip

\item[\namedlabel{c:or-}{($\lor -$)}] $\Vdash_{\mathcal{B}}^{-} \phi \lor \psi$ iff $\Vdash_{\mathcal{B}}^{-} \phi$ and $\Vdash_{\mathcal{B}}^{-} \psi$;

\medskip

\item[\namedlabel{c:to+}{($\to +$)}] $\Vdash_\mathcal{B}^{+} \phi \to \psi$ iff $\phi ; \emptyset \Vdash_\mathcal{B}^{+} \psi$;

\medskip

\item[\namedlabel{c:to-}{($\to -$)}] $\Vdash_\mathcal{B}^{-} \phi \to \psi$ iff $\Vdash_\mathcal{B}^{+} \phi$ and $\Vdash_\mathcal{B}^{-} \psi$;

\medskip

\item[\namedlabel{c:mapsfrom+}{($\mapsfrom +$)}] $\Vdash_\mathcal{B}^{+} \phi \mapsfrom \psi$ iff $\Vdash_\mathcal{B}^{+} \phi$ and $\Vdash_\mathcal{B}^{-} \psi$;

\medskip

\item[\namedlabel{c:mapsfrom-}{($\mapsfrom -$)}] $\Vdash_\mathcal{B}^{-} \phi \mapsfrom \psi$ iff $\emptyset ; \psi \Vdash_\mathcal{B}^{-} \phi$;

%\medskip

%\item For non-empty $\Gamma$, $\Gamma \Vdash_\mathcal{B}^{+} \phi$ iff $\forall (\mathcal{C} \supseteq \mathcal{B}) : \   \Vdash_{\mathcal{C}}^{+} \psi$ for all $\psi \in \Gamma$ implies $\Vdash_{\mathcal{C}}^{+} \phi$.

%\medskip

%\item For non-empty $\Gamma$, $\Gamma \Vdash_\mathcal{B}^{-} \phi$ iff $\forall \mathcal{C} (\mathcal{B} \supseteq \mathcal{C}) : \   \Vdash_{\mathcal{C}}^{-} \psi$ for all $\psi \in \Gamma$ implies $\Vdash_{\mathcal{C}}^{-} \phi$.

\medskip

\item[\namedlabel{c:inf+}{(Inf $+$)}] For non-empty $\Gamma \cup \Delta$, $\Gamma ; \Delta \Vdash^{+}_{\mathcal{B}} \chi$ iff $ \forall \mathcal{C} (\mathcal{C} \supseteq \mathcal{B}) : \ $ if $ \ \Vdash_{\mathcal{C}}^{+} \phi$ for all $\phi \in \Gamma$ and $\Vdash_{\mathcal{C}}^{-} \psi$ for all $\psi \in \Delta$ then $\Vdash_{\mathcal{C}}^{+} \chi$.

\medskip

\item[\namedlabel{c:inf-}{(Inf $-$)}] For non-empty $\Gamma \cup \Delta$, $\Gamma ; \Delta \Vdash^{-}_{\mathcal{B}} \chi$ iff $ \forall\mathcal{C} (\mathcal{C} \supseteq \mathcal{B}) : \ $if$ \ \Vdash_{\mathcal{C}}^{+} \phi$ for all $\phi \in \Gamma$ and $\Vdash_{\mathcal{C}}^{-} \psi$ for all $\psi \in \Delta$ then $\Vdash_{\mathcal{C}}^{-} \chi$.

\medskip

\item[\namedlabel{c:2int+}{($\Bint +$)}] $\Gamma ; \Delta \Vdash^{+} \phi$ iff $\Gamma ; \Delta \Vdash^{+}_{\mathcal{B}} \phi$ for all $\mathcal{B}$;

\medskip

\item[\namedlabel{c:2int-}{($\Bint -$)}] $\Gamma ; \Delta \Vdash^{-} \phi$ iff $\Gamma ; \Delta \Vdash^{-}_{\mathcal{B}} \phi$ for all $\mathcal{B}$;

\end{enumerate}

\end{definition}

From now on we write $\nvdash^{*}_{\mathcal{B}}$, $\nVdash^{*}_{\mathcal{B}}$ and $\nVdash^{*}$ for $* \in \{+, -\}$ to respectively denote that $\vdash^{*}_{\mathcal{B}}$, $\Vdash^{*}_{\mathcal{B}}$ and $\Vdash^{*}$ does not hold. We are, of course, also implicitly equating the relations $\emptyset ; \emptyset \Vdash^{+}_{\mathcal{B}} \phi$ and $\emptyset ; \emptyset \Vdash^{-}_{\mathcal{B}} \phi$ with the relations $\Vdash^{+}_{\mathcal{B}} \phi$ and $ \Vdash^{-}_{\mathcal{B}} \phi$ (respectively), as well as $\emptyset ; \emptyset \Vdash^{+} \phi$ and $\emptyset ; \emptyset \Vdash^{-} \phi$ with $ \Vdash^{+} \phi$ and $ \Vdash^{-} \phi$ (respectively).

Clearly, atomic proofs and refutations now play the same role of the functions $v^{+}$ and $v^{-}$ of Definition \ref{def:Kripkemodelbilateral}. A few differences between the clauses of Definitions \ref{def:bilateralvalidity} and \ref{def:modelvalidity} are worth noticing. Clauses \ref{c:bot+}, \ref{c:top-}, \ref{c:or+} and \ref{c:and-} have to be stated as they are to avoid well-known issues with base-extension definitions using metalinguistic disjunction and semantic unsatisfiability \cite{piecha2015failure}. We could also equivalently define $\bot$ and $\top$ in terms of syntactic constraints on bases by considering them atoms, forbidding bases from containing proofs of $\bot$ and refutations of $\top$  (similar to what is done in \cite{nascimento2025ecumenicalviewprooftheoreticsemantics}) and requiring them to always contain a proof of $\top$ and a refutation of $\bot$, but that would make the completeness proof reliant on proof-theoretic consistency results.

\subsection{Basic lemmata}

Before proceeding to the soundness, completeness and semantic harmony proofs, we prove a few useful lemmata. An asterisk $*$ will sometimes be used as a placeholder for $+$ and $-$ when the result are shown for both signs, but whenever an asterisk is used in a statement or proof all occurrences of it must be uniformly replaced by the same sign.

\begin{lemma} \label{lemma:atomicsupportiffderivability}
 $\Gamma_{\mathcal{\At}}; \Delta_{\At} \Vdash^{*}_{\mathcal{B}} p$ iff $\Gamma_{\mathcal{\At}}; \Delta_{\At} \vdash^{*}_{\mathcal{B}} p$
\end{lemma}

\begin{proof} \

\begin{itemize}
    \item[($\Rightarrow$)] Assume $\Gamma_{\mathcal{\At}} \Vdash^{*}_{\mathcal{B}} p$. Let $\mathcal{C}$ be the base obtained by adding to $\mathcal{B}$ one proof axiom with conclusion $q$ for every $q \in \Gamma_{\At}$ and one refutation axiom with conclusion $r$ for every $r \in \Delta_{\At}$. By applying each of these rules once we get $\vdash^{+}_{\mathcal{C}} q$ and thus $\Vdash^{+}_{\mathcal{B}} q$ for all $q \in \Gamma_{\At}$ by \ref{c:at+}, as well as $\vdash^{-}_{\mathcal{C}} r$ and thus $\Vdash^{-}_{\mathcal{B}} r$ for all $r \in \Delta_{\At}$ by \ref{c:at-}. Since $\Gamma_{\mathcal{\At}} \Vdash^{*}_{\mathcal{B}} p$ holds by assumption, we get $\Vdash^{*}_{\mathcal{C}} q$ by either \ref{c:inf+} or \ref{c:inf-}, depending on what the original $*$ was, so by either \ref{c:at+} or \ref{c:at-} also $\vdash^{*}_{\mathcal{C}}q$. Now consider the shape of the deduction $\Pi$ showing $\vdash^{*}_{\mathcal{C}}q$. If the new rules added to $\mathcal{B}$ in order to obtain $\mathcal{C}$ were not used at all, $\Pi$ is already a deduction in $\mathcal{B}$, so $\vdash^{*}_{\mathcal{B}}q$ and then $\Gamma_{\mathcal{\At}}; \Delta_{\At} \vdash^{*}_{\mathcal{B}} p$ by either Definition \ref{def:deducatomicproofs} or \ref{def:deducatomicrefs}. If $\Pi$ does use one of the new rules, let $\Pi'$ be the deduction obtained by replacing every application of one of the new proof axioms with conclusion $q$ by a proof assumption with shape $q$, as well as every application of one of the new refutation axioms with conclusion $r$ by a refutation assumption with shape $r$.  Since all of the new rules were replaced, every rule of $\Pi'$ is on $\mathcal{B}$. It is straightforward to check that, due to the shape of the new rules of $\mathcal{C}$, all new proof assumptions of $\Pi'$ are in $\Gamma_{\At}$ and all new refutation assumptions are in $\Delta_{\At}$, so $\Pi'$ is a deduction showing $\Theta_{\At}; \Sigma_{\At} \vdash^{*}_{\mathcal{B}} p$ for some $\Theta_{\At} \subseteq \Gamma_{\At}$ and $\Sigma_{\At} \subseteq \Delta_{\At}$, hence once again by either Definition \ref{def:deducatomicproofs} or \ref{def:deducatomicrefs} we get $\Gamma_{\mathcal{\At}}; \Delta_{\At} \vdash^{*}_{\mathcal{B}} p$.

    \medskip

    \item[$(\Leftarrow)$] Assume $\Gamma_{\mathcal{\At}}; \Delta_{\At} \vdash^{*}_{\mathcal{B}} p$. Pick any $\mathcal{C} \supseteq \mathcal{B}$ with $\Vdash^{+}_{\mathcal{C}} q$ for every $q \in \Gamma_{\At}$ and  $\Vdash^{-}_{\mathcal{C}} r$ for every $q \in \Delta_{\At}$. Then by \ref{c:at+} we have $\vdash^{+}_{\mathcal{C}} q$ for every $q \in \Gamma_{\At}$ and by \ref{c:at-} we have $\vdash^{-}_{\mathcal{C}} r$ for every $r \in \Delta_{\At}$. Since $\Gamma_{\mathcal{\At}}; \Delta_{\At} \vdash^{*}_{\mathcal{B}} p$ holds by assumption and $\mathcal{C} \supseteq \mathcal{B}$ then the rules used in the deduction showing this are all in $\mathcal{C}$, so $\Gamma_{\mathcal{\At}}; \Delta_{\At} \vdash^{*}_{\mathcal{C}} p$. By composing the proofs of $q$ for $q \in \Gamma_{\At}$ and the refutations of $r$ for $r \in \Delta_{\At}$ with the open premises of the deduction showing $\Gamma_{\mathcal{\At}}; \Delta_{\At} \vdash^{*}_{\mathcal{C}} p$ we get a deduction showing $\vdash^{*}_{\mathcal{C}} p$, hence $\Vdash^{*}_{\mathcal{C}} p$ by either \ref{c:at+} or \ref{c:at-}. Since $\mathcal{C}$ is an arbitrary extension of $\mathcal{B}$ with $\Vdash^{+}_{\mathcal{C}} q$ for every $q \in \Gamma_{\At}$ and  $\Vdash^{-}_{\mathcal{C}} r$ for every $q \in \Delta_{\At}$ we conclude $\Gamma_{\mathcal{\At}}; \Delta_{\At} \Vdash^{*}_{\mathcal{B}} p$ by either \ref{c:inf+} or \ref{c:inf-}.

    \end{itemize}

\end{proof}

\begin{lemma}(Monotonicity)
If $\Gamma ; \Delta \Vdash^{*}_{\mathcal{B}} \chi$ and $\mathcal{C} \supseteq \mathcal{B}$ then $\Gamma ; \Delta \Vdash^{*}_{\mathcal{C}} \chi$.
\end{lemma}

\begin{proof}
   We start by proving the result for $\Gamma \cup \Delta = \emptyset$ through induction on number of logical operators on the formula $\phi$, then proceeding to show it for $\Gamma \cup \Delta \neq \emptyset$. Notice that use of the induction hypothesis is necessary only in a handful of inductive steps.

   \bigskip

\begin{enumerate}
    \item[] \textbf{Case 1.} $\Gamma \cup \Delta = \emptyset$:
\end{enumerate}

\begin{enumerate}

        \item ($ \chi = p$ for some $p \in \At$).

        \medskip
        
          \noindent Assume $\Vdash_{\mathcal{B}}^{*}p$. By either \ref{c:at+} or \ref{c:at-} we have $\vdash_{\mathcal{B}}^{*}p$. Now pick any $\mathcal{C} \supseteq \mathcal{B}$. Since by the definition of extension all rules of $\mathcal{B}$ are also rules of $\mathcal{C}$ we conclude that the deduction showing $\vdash_{\mathcal{B}}^{*}p$ also shows $\vdash_{\mathcal{C}}^{*}p$, hence $\Vdash_{\mathcal{C}}^{*}p$ by either \ref{c:at+} or \ref{c:at-}.

        \medskip

        \item $(\chi =\bot)$.
        \medskip
        \begin{enumerate}
            \item[$(+)$] Assume $\Vdash_{\mathcal{B}}^{+}\bot$. By \ref{c:bot+} we have $\Vdash_{\mathcal{B}}^{+}p$ and $\Vdash_{\mathcal{B}}^{-}p$ for every $p\in\mathsf{At}$, so also $\vdash_{\mathcal{B}}^{+}p$ and $\vdash_{\mathcal{B}}^{-}p$ for every $p\in\mathsf{At}$ by \ref{c:at+} and \ref{c:at-}. Pick any $\mathcal{C} \supseteq \mathcal{B}$. The deductions showing $\vdash_{\mathcal{B}}^{+}p$ and $\vdash_{\mathcal{B}}^{-}p$ for every $p\in\mathsf{At}$ are also deductions in $\mathcal{C}$, hence $\Vdash_{\mathcal{C}}^{+}p$ and $\Vdash_{\mathcal{C}}^{-}p$ for every $p\in\mathsf{At}$ by \ref{c:at+} and \ref{c:at-} and $\Vdash^{+}_{\mathcal{C}} \bot$ by \ref{c:bot+}.
            
            \medskip
              \item[$(-)$]  $\Vdash_{\mathcal{C}}^{-}\bot$ already holds for any base $\mathcal{C}$ by definition.
            \medskip
        \end{enumerate}
        \medskip
        \item $(\chi = \phi\wedge \psi)$.
        \medskip
        \begin{enumerate}
            \item[$(+)$] Assume $\Vdash_{\mathcal{B}}^{+} \phi\wedge \psi$. By \ref{c:and+} we have $\Vdash_{\mathcal{B}}^{+} \phi$ and $\Vdash_{\mathcal{B}}^{+} \psi$. Pick any $\mathcal{C}\supseteq\mathcal{B}$. The induction hypothesis yields $\Vdash_{\mathcal{C}}^{+} \phi$ and $\Vdash_{\mathcal{C}}^{+} \psi$, whence $\Vdash_{\mathcal{C}}^{+} \phi\wedge \psi$.
            \medskip
            \item[$(-)$] Assume $\Vdash_{\mathcal{B}}^{-} \phi \wedge \psi$. Pick any $\mathcal{C}\supseteq \mathcal{B}$. Now pick any $\mathcal{D} \supseteq \mathcal{C}$ such that $\emptyset; \phi\Vdash_{\mathcal{D}}^{+}p$ and $\emptyset; \psi\Vdash_{\mathcal{D}}^{+}p$ for arbitrary $p \in \At$. By transitivity of extensions we have $\mathcal{D} \supseteq \mathcal{B}$, hence since $\Vdash_{\mathcal{B}}^{-} \phi \wedge \psi$ we conclude $\Vdash^{+}_{\mathcal{D}} p$ by \ref{c:and-}. A similar argument shows that if $\emptyset; \phi\Vdash_{\mathcal{D}}^{-}p$ and $\emptyset; \psi\Vdash_{\mathcal{D}}^{-}p$ then $\Vdash^{-}_{\mathcal{D}} p$ for arbitrary $p \in \At$. This means that, for arbitrary $\mathcal{C} \supseteq \mathcal{D}$, ($\emptyset; \phi\Vdash_{\mathcal{D}}^{+}p$ and $\emptyset; \psi\Vdash_{\mathcal{D}}^{+}p$ implies $\Vdash^{+}_{\mathcal{D}} p$) and (if $\emptyset; \phi\Vdash_{\mathcal{D}}^{-}p$ and $\emptyset; \psi\Vdash_{\mathcal{D}}^{-}p$ implies $\Vdash^{-}_{\mathcal{D}} p$), so by \ref{c:and-} we conclude $\Vdash_{\mathcal{C}}^{-} \phi \wedge \psi$

        \end{enumerate}
        \medskip
        \item $(\chi= \phi\rightarrow \psi)$.
        \begin{enumerate}
        \medskip
            \item[$(+)$] Assume $\Vdash_{\mathcal{B}}^{+} \phi \rightarrow \psi$. By \ref{c:to+} we have $\phi;\emptyset\Vdash_{\mathcal{B}}^{+} \psi$. Pick any $\mathcal{C} \supseteq \mathcal{B}$ and any $\mathcal{D} \supseteq \mathcal{C}$ with $\Vdash^{+}_{\mathcal{D}} \phi$. By transitivity of extensions we have $\mathcal{D} \supseteq \mathcal{B}$, so since $\phi;\emptyset\Vdash_{\mathcal{B}}^{+} \psi$ and $\Vdash^{+}_{\mathcal{D}} \phi$ we conclude $\Vdash^{+}_{\mathcal{D}} \psi$ by \ref{c:inf+}. But $\mathcal{D}$ is an arbitrary extension of $\mathcal{C}$ with $\Vdash^{+}_{\mathcal{D}} \phi$, so we conclude $\phi;\emptyset\Vdash_{\mathcal{C}}^{+} \psi$ by \ref{c:inf+} and then $\Vdash^{+}_{\mathcal{C}} \phi \to \psi$ by \ref{c:to+}.
            \medskip
            \item[$(-)$] Assume $\Vdash_{\mathcal{B}}^{-} \phi\rightarrow \psi$. By \ref{c:to-} we have $\Vdash_{\mathcal{B}}^{+}
            \phi$ and $\Vdash_{\mathcal{B}}^{-} \psi$. Pick any $\mathcal{C} \supseteq \mathcal{B}$. The induction hypothesis yields $\Vdash_{\mathcal{C}}^{+} \psi$ and $\Vdash_{\mathcal{C}}^{-} \psi$, whence $\Vdash_{\mathcal{B}}^{-} \phi\rightarrow \psi$ by \ref{c:to-}.  
        \end{enumerate}
        \medskip
       The $(+)$ cases for $\bot$, $\lor$ and $\mapsfrom$ are symmetrical with the $(-)$ cases for $\top$, $\land$ and $\rightarrow$, respectively, as are the $(-)$ cases for $\bot$, $\lor$ and $\mapsfrom$ and the $(+)$ cases for $\top$, $\land$ and $\rightarrow$, so we omit the remainder of the proof.

    \end{enumerate}

\begin{enumerate}
    \item[]  \textbf{Case 2.} $\Gamma \cup \Delta \neq \emptyset$.

    \medskip
    
   \noindent Pick any $\mathcal{C} \supseteq \mathcal{B}$. Now pick any $\mathcal{D} \supseteq \mathcal{C}$ with $\Vdash^{+}_{\mathcal{D}} \phi$ for all $\phi \in \Gamma$ and $\Vdash^{-}_{\mathcal{D}} \psi$ for all $\psi \in \Delta$. Since $\mathcal{D} \supseteq \mathcal{B}$ holds by transitivity of extension and $\Gamma ; \Delta \Vdash^{*}_{\mathcal{B}} \chi$ holds by assumption, by \ref{c:inf+} or  \ref{c:inf-} we conclude $\Vdash^{*}_{\mathcal{D}} \chi$. Since $\mathcal{D}$ was an arbitrary extension of $\mathcal{C}$ with $\Vdash^{+}_{\mathcal{D}} \phi$ for all $\phi \in \Gamma$ and $\Vdash^{-}_{\mathcal{D}} \psi$ for all $\psi \in \Delta$, by either \ref{c:inf+} or  \ref{c:inf-} we conclude $\Gamma ; \Delta \Vdash^{*}_{\mathcal{C}} \chi$.
\end{enumerate}

\end{proof}

\begin{lemma}\label{lemma:disjunctiongeneral}

 If $\Vdash^{+}_{\mathcal{B}} \phi \vee \psi$, $\phi;\emptyset\Vdash^{*}_{\mathcal{B}} \chi$ and $
 \psi;\emptyset\Vdash^{*}_{\mathcal{B}} \chi$, then $\Vdash^{*}_{\mathcal{B}} \chi$.

\end{lemma}
\begin{proof} The result is proved by induction in the number of logical operator on the formula $\chi$.
  %  We prove the clauses $(+)$ and $(-)$ simultaneously, due to the interaction between proofs and refutations in the semantic clauses for proofs of the co-implication, and refutations of the implication. The proof is by induction on the structure of the formula $C$. 
    
    \medskip 

    \begin{enumerate}
    \item ($\chi = p$ for some $p \in \At$).
    \medskip

    The result follows immediately by \ref{c:or+}.
    
    \bigskip
    
    \item $(\chi=\bot)$.
    \medskip
    \begin{enumerate}

      \item[$(*= +)$] Pick any $\mathcal{C} \supseteq \mathcal{B}$ such that $\Vdash^{+}_{\mathcal{C}} \phi$. Since $\phi ; \emptyset \Vdash_{\mathcal{B}} \bot$ and $\Vdash^{+}_{\mathcal{C}} \phi$ by \ref{c:inf+} we conclude $\Vdash^{+}_{\mathcal{C}} \bot$, so by \ref{c:bot+} also $\Vdash^{+}_{\mathcal{C}} p$ and $\Vdash^{-}_{\mathcal{C}} p$ for every $p \in \At$. Since $\mathcal{C}$ was an arbitrary extension of $\mathcal{B}$ with $\Vdash^{+}_{\mathcal{C}} \phi$, by \ref{c:inf+} and \ref{c:inf-} we conclude $\phi ; \emptyset \Vdash^{+}_{\mathcal{B}} p$ and $\phi ; \emptyset \Vdash^{-}_{\mathcal{B}} p$ for every $p \in \At$. A similar argument establishes  $\psi ; \emptyset \Vdash^{+}_{\mathcal{B}} p$ and $\psi ; \emptyset \Vdash^{-}_{\mathcal{B}} p$ for every $p \in \At$. The induction hypothesis then yields $\Vdash^{+}_{\mathcal{B}} p$ and $\Vdash^{-}_{\mathcal{B}} p$ for all $p \in \At$, hence by \ref{c:bot+} we conclude $\Vdash^{+}_{\mathcal{B}} \bot$

      \medskip

        \item[$(* = -)$]  Immediate by \ref{c:bot-}.
    
    \end{enumerate}
   
    \bigskip
    
    \item $(\chi = \sigma \wedge \tau)$.
    \medskip
    \begin{enumerate}
        \item[$(* =+)$] Pick any $\mathcal{C} \supseteq \mathcal{B}$ with  $\Vdash_{\mathcal{C}}^+ \phi$. Since $\phi;\emptyset\Vdash_{\mathcal{B}}^+ \sigma \wedge \tau$ we conclude $\Vdash_{\mathcal{C}}^+ \sigma \wedge \tau$ by \ref{c:inf+}, so by \ref{c:and+} we have $\Vdash_{\mathcal{C}}^+ \sigma$ and $\Vdash_{\mathcal{C}}^+ \tau$. Since $\mathcal{C}$ was an arbitrary extension of $\mathcal{B}$ with $\Vdash^{+}_{\mathcal{C}} \phi$ we conclude  $\phi; \emptyset \Vdash^{+}_{\mathcal{B}} \sigma$ and $\phi ; \emptyset \Vdash^{+}_{\mathcal{B}} \tau$ by \ref{c:inf+}. A similar argument yields $\psi; \emptyset \Vdash^{+}_{\mathcal{B}} \sigma$ and $\psi ; \emptyset \Vdash^{+}_{\mathcal{B}} \tau$. The induction hypothesis then yields $\Vdash_{\mathcal{B}}^+ \sigma$ and $\Vdash_{\mathcal{B}}^+ \tau$, whence by \ref{c:and+} we conclude $\Vdash^{+}_{\mathcal{B}} \sigma \land \tau$.
        
        \medskip
        
        \item[$(* = -)$] Pick any $\mathcal{C} \supseteq \mathcal{B}$ with $\emptyset;\sigma \Vdash^{+}_{\mathcal{C}}p$ and $\emptyset ; \tau\Vdash^{+}_{\mathcal{C}}p$ for some arbitrary $p \in \At$. Now pick any $\mathcal{D} \supseteq \mathcal{C}$ with $\Vdash^{+}_{\mathcal{D}} \phi$. Since $\phi ; \emptyset \Vdash^{-}_{\mathcal{B}} \sigma \land \tau$ holds by assumption and $\mathcal{D} \supseteq \mathcal{B}$ holds by transitivity of extension, we conclude $\Vdash^{-}_{\mathcal{D}} \sigma \land \tau$ by \ref{c:inf-}. Since $\emptyset;\sigma \Vdash^{+}_{\mathcal{C}}p$ and $\emptyset ; \tau\Vdash^{+}_{\mathcal{C}}p$ both hold, by monotonicity we get $\emptyset;\sigma \Vdash^{+}_{\mathcal{D}}p$ and $\emptyset ; \tau \Vdash^{+}_{\mathcal{D}}p$, hence by \ref{c:and-} we conclude $\Vdash^{+}_{\mathcal{D}} p$. Since $\mathcal{D}$ was an arbitrary extension of $\mathcal{C}$ with $\Vdash^{+}_{\mathcal{D}} \phi$ we conclude $\phi ; \emptyset \Vdash^{+}_{\mathcal{C}} p$. A similar argument yields $\psi ; \emptyset \Vdash^{+}_{\mathcal{C}} p$. Since $\Vdash^{+}_{\mathcal{B}} \phi \lor \psi$ holds by assumption, by monotonicity we have $\Vdash^{+}_{\mathcal{C}} \phi \lor \psi$. Since $\Vdash^{+}_{\mathcal{C}} \phi \lor \psi$, $\phi ; \emptyset \Vdash^{+}_{\mathcal{C}} p$ and $\psi ; \emptyset \Vdash^{+}_{\mathcal{C}} p$ the induction hypothesis yields $ \Vdash^{+}_{\mathcal{C}} p$, so for every extension $\mathcal{C}$ of $\mathcal{B}$ and for every $p \in \At$ we have ($\emptyset;\sigma \Vdash^{+}_{\mathcal{C}}p$ and $\emptyset ; \tau\Vdash^{+}_{\mathcal{C}}p$ implies $\Vdash^{+}_{\mathcal{C}} p$). A similar argument shows that ($\emptyset;\sigma \Vdash^{-}_{\mathcal{C}}p$ and $\emptyset ; \tau\Vdash^{-}_{\mathcal{C}}p$ implies $\Vdash^{-}_{\mathcal{C}} p$) for every $\mathcal{C} \supseteq \mathcal{B}$ and every $p \in \At$, so we conclude $\Vdash^{-}_{\mathcal{B}} \sigma \land \tau$ by \ref{c:and-}.

    \end{enumerate}
    
    \bigskip

    \item $(\chi =\sigma\rightarrow \tau)$.
    \medskip
    \begin{enumerate}
        \item[$(* = +)$] Pick any $\mathcal{C} \supseteq \mathcal{B}$ with $\Vdash^{+}_{\mathcal{C}}\sigma$. Now pick any $\mathcal{D} \supseteq \mathcal{C}$ with $\Vdash^{+}_{\mathcal{D}} \phi$. Since $\phi; \emptyset \Vdash^{+}_{\mathcal{B}} \sigma \to \tau$ and $\mathcal{D} \supseteq \mathcal{B}$ by \ref{c:inf+} we have $ \Vdash^{+}_{\mathcal{D}} \sigma \to \tau$, so by \ref{c:to+} we conclude $\sigma ; \emptyset \Vdash^{+}_{\mathcal{D}} \tau$. Since  $\Vdash^{+}_{\mathcal{C}}\sigma$ we obtain $\Vdash^{+}_{\mathcal{D}}\sigma$ by monotonicity, hence  $\Vdash^{+}_{\mathcal{D}}\tau$ by \ref{c:inf+}. Since $\mathcal{D}$ was an arbitrary extension of $\mathcal{C}$ with $\Vdash^{+}_{\mathcal{D}} \phi$ we conclude $\phi ; \emptyset \Vdash^{+}_{\mathcal{C}} \tau$ by \ref{c:inf+}. A similar argument yields $\psi ; \emptyset \Vdash^{+}_{\mathcal{C}} \tau$. Since $\Vdash^{+}_{\mathcal{B}} \phi \lor \psi$ holds by assumption we get $\Vdash^{+}_{\mathcal{C}} \phi \lor \psi$ by monotonicity, and since $\phi ; \emptyset \Vdash^{+}_{\mathcal{C}} \tau$ and $\psi ; \emptyset \Vdash^{+}_{\mathcal{C}} \tau$ the induction hypothesis yields $\Vdash^{+}_{\mathcal{C}} \tau$. Since $\mathcal{C}$ was an arbitrary extension of $\mathcal{B}$ with $\Vdash^{+}_{\mathcal{C}}\sigma$ we conclude $\sigma ; \emptyset \Vdash^{+}_{\mathcal{B}} \tau$ by \ref{c:inf+}, hence $\Vdash^{+}_{\mathcal{B}}\sigma \to \tau$ by \ref{c:to+}.
        
      \medskip
        
        \item[$(* = -)$] Pick any $\mathcal{C} \supseteq \mathcal{B}$ with $\Vdash_{\mathcal{C}}^{+} \phi$.  Since $\phi;\emptyset\Vdash_{\mathcal{B}}^- \sigma\rightarrow \tau$, we have $\Vdash_{\mathcal{C}}^- \sigma\rightarrow \tau$ by \ref{c:inf-}, so $\Vdash_{\mathcal{C}}^+ \sigma$ and $\Vdash_{\mathcal{C}}^- \tau$ by \ref{c:to-}. Since $\mathcal{C}$ was an arbitrary extension of $\mathcal{C}$ with $\Vdash_{\mathcal{C}}^{+} \phi$ we conclude $\phi ; \emptyset \Vdash_{\mathcal{B}}^{+} \sigma$ and $\phi ; \emptyset \Vdash_{\mathcal{B}}^{-} \tau$. A similar argument yields $\psi ; \emptyset \Vdash_{\mathcal{B}}^{+} \sigma$ and $\psi ; \emptyset \Vdash_{\mathcal{B}}^{-} \tau$. The induction hypothesis then yields $\Vdash_{\mathcal{B}}^{+} \sigma$ and $\Vdash_{\mathcal{B}}^{-} \tau$, whence $\Vdash^{-}_{\mathcal{B}} \sigma \to \tau$ by \ref{c:to-}.
     
    \end{enumerate}
    
    \bigskip 
    
    \end{enumerate}
    The induction steps for $\top$, $\wedge$ and $\mapsfrom$ are symmetrical with the ones above, so we once again omit their proofs.
\end{proof}

\begin{lemma}\label{lemma:conjunctionongeneral}

 If $\Vdash^{-}_{\mathcal{B}} \phi \land \psi$, $\emptyset ;\phi \Vdash^{*}_{\mathcal{B}} \chi$ and $
 \emptyset;\psi \Vdash^{*}_{\mathcal{B}} \chi$, then $\Vdash^{*}_{\mathcal{B}} \chi$.

\end{lemma}

\begin{proof}
    Straightforward adaption of the proof of Lemma \ref{lemma:disjunctiongeneral}.
\end{proof}

\begin{lemma} \label{lemma:exfalsobot}
   If $\Vdash^{+}_{\mathcal{B}} \bot$ then $\Vdash^{+}_{\mathcal{B}} \chi$ and $\Vdash^{-}_{\mathcal{B}} \chi$ for every $\chi$.
\end{lemma}

\begin{proof} The result is proved by induction on the number of logical operators on $\chi$.

  %  We prove the clauses $(+)$ and $(-)$ simultaneously, due to the interaction between proofs and refutations in the semantic clauses for proofs of the co-implication, and refutations of the implication. The proof is by induction on the structure of the formula $C$. 
    
    \medskip 

    \begin{enumerate}
    \item ($\chi = p$ for some $p \in \At$).   The result follows immediately by \ref{c:bot+}.

    \medskip

    \item $(\chi=\bot)$.    $\Vdash^{+}_{\mathcal{B}} \bot$ follows immediately from the assumption and $\Vdash^{-}_{\mathcal{B}} \bot$ holds by \ref{c:bot-}.

    \medskip

      \item $(\chi=\top)$. From $\Vdash^{+}_{\mathcal{B}} \bot$ follows $\Vdash^{+}_{\mathcal{B}} p$ and  $\Vdash^{-}_{\mathcal{B}} p$ for all $p \in \At$ by \ref{c:bot+},  whence $\Vdash^{-}_{\mathcal{B}} \top$ by \ref{c:top-}, and by \ref{c:top+} we get $\Vdash^{+}_{\mathcal{B}} \top$.

      \medskip

\item $(\chi = \phi \land \psi)$. Assume $\Vdash^{+}_{\mathcal{B}} \bot$. The induction hypothesis yields $\Vdash^{+}_{\mathcal{B}} \phi$ and $\Vdash^{+}_{\mathcal{B}} \psi$, hence $\Vdash^{+}_{\mathcal{B}} \phi \land \psi$ by \ref{c:and+}. From $\Vdash^{+}_{\mathcal{B}} \bot$ it also follows that $\Vdash^{+}_{\mathcal{B}} p$ and  $\Vdash^{-}_{\mathcal{B}} p$ hold for all $p \in \At$, hence by monotonicity $\Vdash^{+}_{\mathcal{C}} p$ and  $\Vdash^{-}_{\mathcal{C}} p$ for every $\mathcal{C} \supseteq \mathcal{B}$, so trivially for every $\mathcal{C} \supseteq \mathcal{B}$ we have that if $ \emptyset ; \phi \Vdash^{*}_{\mathcal{C}} p$ and $ \emptyset ; \psi\Vdash^{*}_{\mathcal{C}} p$ then $\Vdash^{*}_{\mathcal{C}} p$, whence $\Vdash^{-}_{\mathcal{C}} \phi \land \psi$ by \ref{c:and-}.

      \medskip

\item $(\chi = \phi \to \psi)$. Assume $\Vdash^{+}_{\mathcal{B}} \bot$. The induction hypothesis yields $\Vdash^{+}_{\mathcal{B}} \psi$, so by monotonicity for all $\mathcal{C} \supseteq \mathcal{B}$ we have $\Vdash^{+}_{\mathcal{C}} \psi$, whence $\phi ; \emptyset \Vdash^{+}_{\mathcal{B}} \psi$ holds trivially by \ref{c:inf+} and so $\Vdash^{+}_{\mathcal{B}} \phi \to \psi$ holds by \ref{c:to+}. The induction hypothesis also yields $\Vdash^{+}_{\mathcal{B}} \phi$ and $\Vdash^{-}_{\mathcal{B}} \psi$, so $\Vdash^{-}_{\mathcal{B}} \phi \to \psi$ by \ref{c:to-}.

    \end{enumerate}

 The induction steps for $\wedge$ and $\mapsfrom$ are symmetrical with the ones above.
  
\end{proof}

\begin{lemma} \label{lemma:exfalsotop}
   If $\Vdash^{-}_{\mathcal{B}} \top$ then $\Vdash^{+}_{\mathcal{B}} \chi$ and $\Vdash^{-}_{\mathcal{B}} \chi$ for every $\chi$.
\end{lemma}

\begin{proof} Straightforward adaptation of the proof of Lemma \ref{lemma:exfalsobot}.

\end{proof}

\subsection{On the adequate representation of semantic conditions by syntactic rules}

A distinguishing feature of bilateral base-extension semantics is that the clauses defined through recourse to atomic validity (\ref{c:bot+}, \ref{c:top-}, \ref{c:or+} and \ref{c:and-}) have to take into account both atomic proofs and atomic refutations (\textit{e.g.} a proof of $\bot$ or of $\top$ requires both a proof and a refutation of every atom). This essentially means that the rules depicted in Proposition \ref{propositionderivability} are explicitly used when defining the semantics. If we define naïve semantic clauses that only take into account atomic proofs in the proof support clauses and atomic refutations in the refutation support clauses instead, it becomes possible to show that the inferences in Proposition \ref{propositionderivability} are not sound. 

\begin{definition} \label{def:naive}
    The following are naïve semantic clauses for bilateral base-extension validity:

\begin{enumerate}

    \item[\namedlabel{c:naibot+}{($N\ \bot +$)}] $\Vdash^{+}_{\mathcal{B}} \bot$ iff $\Vdash^{+}_{\mathcal{B}} p$, for all $p \in \At$;

    \medskip
    
       \item[\namedlabel{c:naitop-}{($N \ \top -$)}] $\Vdash^{-}_{\mathcal{B}} \top$ iff $\Vdash^{-}_{\mathcal{B}} p$, for all $p \in \At$;

\medskip
       \item[\namedlabel{c:naior+}{($N\lor +$)}] $\Vdash_{\mathcal{B}}^{+} \phi \lor \psi$ iff  $ \forall \mathcal{C} (\mathcal{C} \supseteq \mathcal{B})$ and $\forall p (p \in \At)$ $:  \phi ;\emptyset \Vdash_{\mathcal{C}}^{+} p$ and $\psi ; \emptyset \Vdash_{\mathcal{C}}^{+} p$ implies $\Vdash_{\mathcal{C}}^{+} p$;

       \medskip
       
       \item[\namedlabel{c:naiand-}{($N\land -$)}] $\Vdash_{\mathcal{B}}^{-} \phi \land \psi$ iff  $ \forall \mathcal{C} (\mathcal{C} \supseteq \mathcal{B})$ and $\forall p (p \in \At)$ : $ \emptyset ; \phi \Vdash_{\mathcal{C}}^{-} p$ and $\emptyset ;\psi \Vdash_{\mathcal{C}}^{-} p$ implies $\Vdash_{\mathcal{C}}^{-} p$
\end{enumerate}

\end{definition}

\begin{proposition}

\normalfont If clauses \ref{c:bot+}, \ref{c:top-},\ref{c:or+} and \ref{c:and-} are replaced by the clauses \ref{c:naibot+}, \ref{c:naitop-},\ref{c:naior+} and \ref{c:naiand-}, the following can be shown:
    \begin{enumerate}
    \item $\bot; \emptyset \Vdash^{-} \chi$ does not hold in general;

    \medskip
    
    \item $\emptyset; \top \Vdash^{+} \chi$ does not hold in general;

    \medskip

    \item $\Vdash^{+}_{\mathcal{B}} \phi \lor \psi$, $\phi ; \emptyset \Vdash^{-}_{\mathcal{B}} \chi$ and $\psi ; \emptyset \Vdash^{-}_{\mathcal{B}} \chi$ does not imply $\Vdash^{-}_{\mathcal{B}} \chi$ in general;

    \medskip

    \item $\Vdash^{-}_{\mathcal{B}} \phi \land \psi$, $ \emptyset; \phi \Vdash^{+}_{\mathcal{B}} \chi$ and $\emptyset; \psi \Vdash^{+}_{\mathcal{B}} \chi$ does not imply $\Vdash^{+}_{\mathcal{B}} \chi$ in general;
\end{enumerate}
\end{proposition}

\begin{proof} The result is shown by considering specific bases $\mathcal{B}$:

\begin{enumerate} 

    \item   Let $\mathcal{B}$ be a base containing proof axioms (cf. Definition \ref{def:proofrulesrefutrulesandaxioms}) with conclusion $p$ for every $p \in \At$ and no refutation rules. Let $\chi = q$ for $q \in \At$.  Since $\vdash^{+}_{\mathcal{B}} p$ can be shown for all $p \in \At$ by simply applying the proof axiom concluding $p$, we conclude $\Vdash^{+}_{\mathcal{B}} p$ for all $p$ by \ref{c:at+} and thus $\Vdash^{+}_{\mathcal{B}} \bot$ by \ref{c:naibot+}. Since $\mathcal{B}$ does not contain refutation rules we have $\nvdash^{-}_{\mathcal{B}} q$, so $\nVdash^{-}_{\mathcal{B}} q$ by \ref{c:at-}. Since by the definition of extension $\mathcal{B} \supseteq \mathcal{B}$ holds for every $\mathcal{B}$, by \ref{c:inf-} we have $\bot ; \emptyset \nVdash^{-}_{\mathcal{B}} q$, whence $\bot ; \emptyset \nVdash^{-} q$ by \ref{c:2int-}.

    \medskip

     \item  Straightforward adaptation of the previous case in which we pick a base $\mathcal{B}$ containing atomic axioms refuting $p$ for every $p \in \At$ but with no proof rules.

     \medskip

     \item Put $\phi = p$, $\psi = q$ and $\chi = r$ for $p, q, r \in \At$, assuming also that those atoms are distinct from each other. Let $\mathcal{B}$ be a base containing rules of the following shape for all $s \in \At$, but no other rules:

     \begin{prooftree}
         \AxiomC{$[p] $}
         \noLine
         \UnaryInfC{.}
         \noLine
         \UnaryInfC{.}
         \noLine
         \UnaryInfC{.}
         \UnaryInfC{$s$}
          \AxiomC{$[q] $}
         \noLine
         \UnaryInfC{.}
         \noLine
         \UnaryInfC{.}
         \noLine
         \UnaryInfC{.}
         \UnaryInfC{$s$}
         \RightLabel{\tiny{$R1 - s$}}
         \BinaryInfC{$s$}
         \DisplayProof
           \qquad \qquad
         \AxiomC{}
         \noLine
         \UnaryInfC{}
         \noLine
         \UnaryInfC{}
         \noLine
         \UnaryInfC{}
         \noLine
         \UnaryInfC{}
         \noLine
         \UnaryInfC{}
         \UnaryInfC{$p$}
         \doubleLine
           \RightLabel{\tiny{$R2$}}
         \UnaryInfC{$r$}
         \DisplayProof
           \qquad \qquad
         \AxiomC{}
         \noLine
         \UnaryInfC{}
         \noLine
         \UnaryInfC{}
         \noLine
         \UnaryInfC{}
         \noLine
         \UnaryInfC{}
         \noLine
         \UnaryInfC{}
         \UnaryInfC{$q$}
         \doubleLine
           \RightLabel{\tiny{$R3$}}
         \UnaryInfC{$r$}
     \end{prooftree}

\noindent It can be shown (see Lemma \ref{lemma:atomicsupportiffderivability}) that, for all $\mathcal{C} \supseteq \mathcal{B}$ and $s \in \At$, if $p ; \emptyset \Vdash^{+}_{\mathcal{C}} s$ and $q ; \emptyset \Vdash^{+}_{\mathcal{C}} s$ then $p ; \emptyset \vdash^{+}_{\mathcal{C}} s$ and $q ; \emptyset \vdash^{+}_{\mathcal{C}} s$, hence these deductions can be used together with $R1 - s$ (guaranteed to be in $\mathcal{C}$ because $\mathcal{C} \supseteq \mathcal{B}$) to obtain a deduction showing $\vdash^{+}_{\mathcal{C}} s$, whence $\Vdash^{+}_{\mathcal{C}} s$ by \ref{c:at+}. Since $\mathcal{B}$ contains a rule $R1 - s$ for every $s \in \At$ and $\mathcal{C}$ was arbitrary we conclude that, for every $\mathcal{C} \supseteq \mathcal{B}$ and every $s \in \At$, if $p ; \emptyset \Vdash^{+}_{\mathcal{C}} s$ and $q ; \emptyset \Vdash^{+}_{\mathcal{C}} s$ then $ \Vdash^{+}_{\mathcal{C}} s$, hence by \ref{c:naior+} we conclude $\Vdash^{+}_{\mathcal{B}} p \lor q$. Now pick any $\mathcal{C} \supseteq \mathcal{B}$ with $\Vdash^{+}_{\mathcal{C}} p$. Then  $\vdash^{+}_{\mathcal{C}} p$ by \ref{c:at+}, so  we can compose the proof of $p$ with $R2$ to obtain a refutation of $r$, hence $\vdash^{-}_{\mathcal{C}} r$ and then $\Vdash^{-}_{\mathcal{C}} r$ by \ref{c:at-}. Since $\mathcal{C}$ was arbitrary extension of $\mathcal{B}$ with $\Vdash^{+}_{\mathcal{C}} p$ we conclude $p ; \emptyset\Vdash^{-}_{\mathcal{B}} r$ by \ref{c:inf-}. A similar argument using the rule $R3$ establishes $q ; \emptyset\Vdash^{-}_{\mathcal{B}} r$. So we have $\Vdash^{+}_{\mathcal{B}} p \lor q$, $p ; \emptyset\Vdash^{-}_{\mathcal{B}} r$ and $q ; \emptyset\Vdash^{-}_{\mathcal{B}} r$. Now we proceed to show $\nVdash^{-}_{\mathcal{\mathcal{B}}} r$. Since they are the only refutation rules of $\mathcal{B}$, any refutation of $r$ in $\mathcal{B}$ has to end with either $R2$ or $R3$, so such a refutation would require either an atomic proof of $p$ or an atomic proof of $q$ for the premise of $R2$ or $R3$. Since they are the only proof rules in $\mathcal{B}$ with conclusion $p$ or $q$, such atomic proofs have to end either with an application of $R1 - p$ or of $R1 - q$. Assume that the final rule is an application of $R1 - p$. Then either the right premise with shape $p$ of this final application $R1 - p$ is a proof assumption, in which case the deduction cannot be an atomic proof, or it is a deduction showing $q ; \emptyset \vdash^{+}_{\mathcal{B}} p$ (which might already be a proof of $p$, as $q$ does not need to effectively occur in the deduction showing this). If it is a deduction showing $q ; \emptyset \vdash^{+}_{\mathcal{B}} p$ then this deduction must also end with an application of $R1 - p$, which once again requires a deduction showing $q ; \emptyset \vdash^{+}_{\mathcal{B}} p$ for its right premise. The same reasoning applies if instead we consider the penultimate rule of the refutation to be an application of $R1-q$. As such, for any natural number $n$ we can show that there is no deduction in $\mathcal{B}$ with $n$ rule applications that is an atomic refutation of $r$ (since after reaching a deduction with $n$ rule applications we would require yet another deduction for one of the premises), hence since atomic proofs are always finite conclude that there is no refutation of $r$ in $\mathcal{B}$. This allows us to conclude $\nvdash^{-}_{\mathcal{B}} r$ and thus $\nVdash^{-}_{\mathcal{B}} r$ by \ref{c:at-}, as desired.

\medskip

\item Straightforward adaptation of the previous case in which we pick a base $\mathcal{B}$ containing only the followign atomic rules, for every $s \in \At$:

\begin{prooftree}
         \AxiomC{$\llbracket p \rrbracket $}
         \noLine
         \UnaryInfC{.}
         \noLine
         \UnaryInfC{.}
         \noLine
         \UnaryInfC{.}
         \doubleLine
         \UnaryInfC{$r$}
          \AxiomC{$\llbracket q \rrbracket $}
         \noLine
         \UnaryInfC{.}
         \noLine
         \UnaryInfC{.}
         \noLine
         \UnaryInfC{.}
         \doubleLine
         \UnaryInfC{$s$}
         \doubleLine
         \RightLabel{\tiny{$R1 - s$}}
         \BinaryInfC{$s$}
         \DisplayProof
           \qquad \qquad
         \AxiomC{}
         \noLine
         \UnaryInfC{}
         \noLine
         \UnaryInfC{}
         \noLine
         \UnaryInfC{}
         \noLine
         \UnaryInfC{}
         \noLine
         \UnaryInfC{}
         \doubleLine
         \UnaryInfC{$p$}
           \RightLabel{\tiny{$R2$}}
         \UnaryInfC{$r$}
         \DisplayProof
           \qquad \qquad
         \AxiomC{}
         \noLine
         \UnaryInfC{}
         \noLine
         \UnaryInfC{}
         \noLine
         \UnaryInfC{}
         \noLine
         \UnaryInfC{}
         \noLine
         \UnaryInfC{}
         \doubleLine
         \UnaryInfC{$q$}
           \RightLabel{\tiny{$R3$}}
         \UnaryInfC{$r$}
     \end{prooftree}

\end{enumerate}

\end{proof}

Our observations concerning the problems presented by $\BintN$ (which are no longer present in $\BintNA$) with respect to normalization, the subformula property and the (im)possibility of letting the rules $\bot(+)$ and $\top(-)$ have only atomic conclusions can be taken as an explanation of this phenomenon. In Definition \ref{def:bilateralvalidity} we define validity through recursion on the logical connectives of formulas, meaning that validity for all formulas ultimately reduces to atomic validity. On the other hand, deducibility definitions for $\BintN$ and $\BintNA$ are defined through recursion on the length of derivations (a derivation of minimal length being either a single proof assumption or a single refutation assumption), so the complexity of formulas plays no role in the recursion. This is why we can bypass the limitations of the rules of $\BintN$ by picking a deduction with length $n$ of a formula $\phi$ and using it to obtain a deduction with length $m$ ($m > n$) of a formula $\psi$ that may be less complex than $\phi$, as done in the proof of Proposition \ref{propositionderivability}. The same cannot be done with validity clauses because proof and refutation support for a formula must always be drawn from proofs and refutations of its subformulas (possibly combined with proofs and refutations of arbitrary atomic formulas).

    An interesting argument that can be extracted from these observations is that the rules of a system of natural deduction are not necessarily representative of the logic's underlying semantic concept of proof (or refutation). This is exemplified by our bilateral base-extension semantics, since the rules of $\BintNA$ can be taken as effectively representing the semantics of $\Bint$ but the rules of the deductively equivalent $\BintN$ cannot. $\BintN$ essentially relies on how deducibility is defined to obtain the semantic principles expressed proof-theoretically in Proposition \ref{propositionderivability} -- which are now seen to be characteristic semantic principles of $\Bint$ instead of arbitrary derivable rules.

\subsection{Soundness}

Soudness and completeness will be proved specifically by taking $\BintNA$ into account.

\begin{theorem}[Soudness]
    If $\Gamma; \Delta \vdash^{*}_{\BintNA} A$ then $\Gamma; \Delta \Vdash^{*} A$.
\end{theorem}

\begin{proof}
    We prove the result by induction on the length of derivations, understood as the number of rule applications it contains.

    \medskip

\begin{enumerate}
    \item[]  \textbf{Case 1}. The derivation has length $0$. Then it either consists in a single occurrence of proof assumption $\phi$ and is a deduction showing $\phi; \emptyset \vdash^{+}_{\BintNA} \phi$ or in a single occurrence of a refutation assumption $\phi$ and is a deduction showing $\emptyset; \phi \vdash^{-}_{\BintNA} \phi$. It follows immediately from \ref{c:inf+} that  $\phi; \emptyset \Vdash^{+}_{\mathcal{B}} \phi$ holds for every $\mathcal{B}$, whence $\phi; \emptyset \Vdash^{+} \phi$ by \ref{c:2int+}. It also  follows immediately from \ref{c:inf-} that  $\emptyset; \phi \Vdash^{-}_{\mathcal{B}} \phi$ holds for every $\mathcal{B}$, whence $\emptyset; \phi \Vdash^{-} \phi$ by \ref{c:2int-}.

    \medskip

     \item[] \textbf{Case 2}. The derivation has lenght greater than $0$. Then the result is proved by considering the last rule applied in it, so we simply have to show that each rule is semantically valid. Notice, however, that even though a deduction ending with an application of $(\top +)$ or  $(\bot -)$ has no open assumptions it can still be considered a deduction showing $\Gamma ; \Delta^{+} \vdash_{\BintNA} \top$ or $\Gamma ; \Delta \vdash^{-}_{\BintNA} \bot$ (respectively) due to how dependencies are handled in Definitions \ref{def:deductionsoriginalproof} and \ref{def:deductionsoriginalrefutation}.

\end{enumerate}

        \begin{enumerate}
\item $(\top +)$ $ \Gamma; \Delta \Vdash^+\top$.

\medskip 

\noindent For any $\mathcal{B}$ we have $\Vdash_{\mathcal{B}}^+\top$, so by monotonicity $\Vdash_{\mathcal{C}}^+\top$ holds for every $\mathcal{C} \supseteq \mathcal{B}$, hence for any $\Gamma$ and $\Delta$ it trivially follows that $ \Gamma; \Delta \Vdash^+_{\mathcal{B}}\top$ holds by \ref{c:inf+}, whence $ \Gamma; \Delta \Vdash^+\top$ by \ref{c:2int+} and arbitrariness of $\mathcal{B}$.

\medskip

\begin{comment}

\item $( \top -)$ If $ \Gamma; \Delta \Vdash^-\top$ then $ \Gamma; \Delta \Vdash^*\chi$.

\medskip 

\noindent Assume $ \Gamma; \Delta \Vdash^-\top$. Then $ \Gamma; \Delta \Vdash^{-}_{\mathcal{B}}\top$ for all $\mathcal{B}$ by \ref{c:2int-}. Pick any $\mathcal{B}$ and any $\mathcal{C} \supseteq \mathcal{B}$ with $\Vdash^{+}_{\mathcal{C}}\sigma$ and $\Vdash^{-}_{\mathcal{C}}\tau$ for every $\sigma\in\Gamma$ and $\tau\in\Delta$. Then $\Vdash^{-}_{\mathcal{C}} \top$ by \ref{c:inf-}. From Lemma \ref{lemma:exfalsotop} we get $\Vdash^{*}_{\mathcal{C}} \chi$, hence by arbitrariness of $\mathcal{C}$ we conclude $\Gamma ; \Delta\Vdash^{*}_{\mathcal{B}} \chi$ by either \ref{c:inf+} or \ref{c:inf-}, thence we have $ \Gamma; \Delta \Vdash^*\chi$ by either \ref{c:2int+} or \ref{c:2int-}.

\medskip

\item $( \bot -)$ $ \Gamma; \Delta \Vdash^- \bot$.

\medskip 

\noindent For any $\mathcal{B}$ we have $\Vdash_{\mathcal{B}}^-\bot$, so by monotonicity $\Vdash_{\mathcal{C}}^-\bot$ holds for every $\mathcal{C} \supseteq \mathcal{B}$, hence for any $\Gamma$ and $\Delta$ it trivially follows that $ \Gamma; \Delta \Vdash^-_{\mathcal{B}}\bot$ holds by \ref{c:inf-}, whence $ \Gamma; \Delta \Vdash^-\bot$ by \ref{c:2int-} and arbitrariness of $\mathcal{B}$.

\end{comment}

\medskip

\item $(\bot +)$ If $ \Gamma; \Delta \Vdash^+ \bot$ then $ \Gamma; \Delta \Vdash^*\chi$.

\medskip 

\noindent Assume $ \Gamma; \Delta \Vdash^+\bot$. Then $ \Gamma; \Delta \Vdash^{+}_{\mathcal{B}}\bot$ for all $\mathcal{B}$ by \ref{c:2int+}. Pick any $\mathcal{B}$ and any $\mathcal{C} \supseteq \mathcal{B}$ with $\Vdash^{+}_{\mathcal{C}}\sigma$ and $\Vdash^{-}_{\mathcal{C}}\tau$ for every $\sigma\in\Gamma$ and $\tau\in\Delta$. Then $\Vdash^{+}_{\mathcal{C}} \bot$ by \ref{c:inf+}. From Lemma \ref{lemma:exfalsobot} we get $\Vdash^{*}_{\mathcal{C}} \chi$, hence by arbitrariness of $\mathcal{C}$ we conclude $\Gamma ; \Delta\Vdash^{*}_{\mathcal{B}} \chi$ by either \ref{c:inf+} or \ref{c:inf-}, thence we have $ \Gamma; \Delta \Vdash^*\chi$ by either \ref{c:2int+} or \ref{c:2int-}.

\medskip

\item $(I\wedge +)$ If $\Gamma_1;\Delta_1\Vdash^+ \phi$ and $\Gamma_2;\Delta_2\Vdash^+ \psi$ then $\Gamma_1,\Gamma_2;\Delta_1,\Delta_2\Vdash^+ \phi \wedge \psi$.

\medskip 
\noindent Assume $\Gamma_1; \Delta_1 \Vdash^{+} \phi$ and $\Gamma_2; \Delta_2 \Vdash^{+} \phi$. Then (i) $\Gamma_1; \Delta_1 \Vdash^{+}_{\mathcal{B}} \phi $ and (ii) $\Gamma_2; \Delta_2 \Vdash^{+}_{\mathcal{B}} \psi$ hold for arbitrary $\mathcal{B}$ by \ref{c:2int+}. Pick any $\mathcal{B}$ and any $\mathcal{C} \supseteq \mathcal{B}$ with $\Vdash^{+}_{\mathcal{C}}\sigma$ and $\Vdash^{-}_{\mathcal{C}}\tau$ for every $\sigma\in\Gamma_1\cup\Gamma_2$ and $\tau\in\Delta_1\cup\Delta_2$. Then by (i) and \ref{c:inf+} we get $\Vdash^{+}_{\mathcal{C}} \phi$, and by (ii) and \ref{c:inf+} we get $\Vdash^{+}_{\mathcal{C}}\psi$. Hence  $\Vdash^{+}_{\mathcal{C}} \phi \wedge \psi$, and since $\mathcal{C}$ was arbitrary we get $\Gamma_1,\Gamma_2;\Delta_1,\Delta_2\Vdash_{\mathcal{B}}^+ \phi \wedge \psi$ by \ref{c:inf+}, thence from the arbitrariness of $\mathcal{B}$ we get $\Gamma_1,\Gamma_2;\Delta_1,\Delta_2\Vdash^+ \phi\wedge \psi$.
\bigskip

\item $(E\wedge_1 +)$. If $\Gamma;\Delta\Vdash^+ \phi\wedge \psi$ then $\Gamma;\Delta\Vdash^+\phi$.

\medskip 

\noindent Assume $\Gamma;\Delta\Vdash^+ \phi\wedge \psi$. Then $\Gamma;\Delta\Vdash_{\mathcal{B}}^+ \phi\wedge \psi$ holds in every $\mathcal{B}$ by \ref{c:2int+}. Pick an arbitrary $\mathcal{B}$, and let $\mathcal{C}$ be a base extending $\mathcal{B}$ such that $\Vdash^{+}_{\mathcal{C}}\gamma$ and $\Vdash^{-}_{\mathcal{C}}\delta$ for every $\gamma\in\Gamma$ and $\delta\in\Delta$. Then since $\Gamma;\Delta\Vdash_{\mathcal{B}}^+ \phi\wedge \psi$ we have $\Vdash_{\mathcal{C}} \phi\wedge \psi$ by \ref{c:inf+}, so also $\Vdash^{+}_{\mathcal{C}} \phi$ by \ref{c:and+}. By arbitrariness of $\mathcal{C}$ we conclude $\Gamma;\Delta\Vdash_{\mathcal{B}}^+ \phi$, thence by arbitrariness of $\mathcal{B}$ we conclude $\Gamma;\Delta\Vdash^+ \phi$ by \ref{c:2int+}.
\bigskip 

\item $(E\wedge_2 +)$. If $\Gamma;\Delta\Vdash^+ \phi\wedge 
\psi$ then $\Gamma;\Delta\Vdash^+ \psi$.

\medskip 
\noindent The proof is analogous to the one for $(E \wedge_{1} +)$.
\bigskip 

\item $(I \lor_{1} +)$ If $\Gamma ; \Delta \Vdash^{+} \phi$ then $\Gamma ; \Delta \Vdash^{+} \phi\lor \psi$.

\medskip

\noindent Assume $\Gamma; \Delta \Vdash^{+} \phi$. Then $\Gamma; \Delta \Vdash^{+}_{\mathcal{B}} \phi$ holds for every $\mathcal{B}$ by \ref{c:2int+}. For some arbitrary $\mathcal{B}$, pick any $\mathcal{C} \supseteq \mathcal{B}$ such that $\Vdash^{+}_{\mathcal{C}} \sigma$ for every $\sigma \in \Gamma$ and $\Vdash_{\mathcal{C}}^{-} \tau$ for every $\tau \in \Delta$. Then since $\Gamma; \Delta \Vdash^{+}_{\mathcal{B}} \phi$ we conclude $\Vdash^{+}_{\mathcal{C}} \phi$ by \ref{c:inf+}. Now pick any $\mathcal{D} \supseteq \mathcal{C}$ such that $\phi ; \emptyset \Vdash^{+}_{\mathcal{D}} p$ and $\psi ; \emptyset \Vdash^{+}_{\mathcal{D}} p$ for some arbitrary $p \in \At$. Since $\Vdash^{+}_{\mathcal{C}} \phi$, by monotonicity we conclude $\Vdash^{+}_{\mathcal{D}} \phi$, and since $\phi ; \emptyset \Vdash^{+}_{\mathcal{D}} p$ we conclude $\Vdash^{+}_{\mathcal{D}} p$ by \ref{c:inf+}. A similar argument shows that, for any $\mathcal{D} \supseteq \mathcal{C}$, if $\phi ; \emptyset \Vdash^{-}_{\mathcal{D}} p$ and $\psi ; \emptyset \Vdash^{-}_{\mathcal{D}} p$ then $\Vdash^{-}_\mathcal{D} p$. From this we conclude $\Vdash^{+}_{\mathcal{C}} \phi \lor \psi$ by \ref{c:or+}, hence by arbitrariness of $\mathcal{C}$ also $\Gamma ; \Delta \Vdash^{+}_{\mathcal{C}} \phi\lor \psi$ by \ref{c:inf+}, thence $\Gamma ; \Delta \Vdash^{+} \phi\lor \psi$ by \ref{c:2int+}.

\medskip

\item $(I \lor_{2} +)$ If $\Gamma;\Delta \Vdash^{+} \psi$ then $\Gamma;\Delta \Vdash^{+} \phi \lor \psi$.

\medskip

\noindent The proof is analogous to the one for $(I \lor_{1} +)$.

\bigskip
        
\item ($E \lor +$) If $\Gamma_{1}; \Delta_{1} \Vdash^+ \phi \lor \psi$ and $\Gamma_{2}, \phi; \Delta_{2} \Vdash^{*} \chi$ and $\Gamma_{3}, \psi; \Delta_{3} \Vdash^{*} \chi$ then $\Gamma_1, \Gamma_2, \Gamma_3; \Delta_1, \Delta_2, \Delta_3 \Vdash^* \chi$. 

\medskip

\noindent Assume $\Gamma_{1}; \Delta_{1} \Vdash^+ \phi \lor \psi$ and $\Gamma_{2}, A; \Delta_{2} \Vdash^{+} \chi$ and $\Gamma_{3}, \psi; \Delta_{3} \Vdash^{+} \chi$. Then (i) $\Gamma_1; \Delta_1 \Vdash^{+}_{\mathcal{B}} \phi \lor \psi$, (ii) $\Gamma_2, \phi; \Delta_2 \Vdash^{+}_{\mathcal{B}} \chi$ and (iii) $\Gamma_3, \psi; \Delta_3 \Vdash^{+}_{\mathcal{B}} \chi$ hold for arbitrary $\mathcal{B}$. Now let $\mathcal{C}$ be some base extending $\mathcal{B}$ such that $\Vdash^{+}_{\mathcal{C}}\sigma$ and $\Vdash^{-}_{\mathcal{C}}\tau$ for every $\sigma\in\Gamma_1\cup\Gamma_2\cup\Gamma_3$ and $\tau\in\Delta_1\cup\Delta_2\cup\Delta_3$.  Note that, in particular, $\Vdash^{+}_{\mathcal{C}}\gamma$ and $\Vdash^{-}_{\mathcal{C}}\delta$ for every $\gamma\in\Gamma_1$ and $\delta\in\Delta_1$, so from (i) we get $\Vdash^{+}_{\mathcal{C}} \phi \lor \psi$ by \ref{c:inf+}. Now pick any $\mathcal{D} \supseteq \mathcal{C}$ such that $\Vdash^{+}_{\mathcal{D}} \phi$. Since $\Vdash^{+}_{\mathcal{C}} \theta$ and $\Vdash^{-}_{\mathcal{C}} \omega$ for every $\theta \in \Gamma_{2}$ and $\omega \in \Delta_{2}$, by monotonicity we have $\Vdash^{+}_{\mathcal{D}} \theta$ and $\Vdash^{-}_{\mathcal{D}} \omega$, so from (ii) we obtain $\Vdash^{+}_{\mathcal{D}} \chi$ by \ref{c:inf+}, and by arbitrariness of $\mathcal{D}$ and \ref{c:inf+} we conclude $\phi ; \emptyset \Vdash^{+}_{\mathcal{C}} \chi$. A similar argument shows $\psi ; \emptyset \Vdash^{+}_{\mathcal{C}} \chi$. Since $\Vdash^{+}_{\mathcal{C}} \phi \lor \psi$, $\phi ; \emptyset \Vdash^{+}_{\mathcal{C}} \chi$ and $\psi ; \emptyset \Vdash^{+}_{\mathcal{C}} \chi$ we get $\Vdash^{+}_{\mathcal{C}} \chi$ by Lemma \ref{lemma:disjunctiongeneral}, hence by arbitrariness of $\mathcal{C}$ and \ref{c:inf+} we conclude  $\Gamma_1, \Gamma_2, \Gamma_3; \Delta_1, \Delta_2, \Delta_3 \Vdash^+_{\mathcal{B}} \chi$, thence $\Gamma_1, \Gamma_2, \Gamma_3; \Delta_1, \Delta_2, \Delta_3 \Vdash^+ \chi$ by arbitrariness of $\mathcal{B}$ and \ref{c:2int+}. A similar argument shows that by assuming  $\Gamma_{1}; \Delta_{1} \Vdash^+ \phi \lor \psi$ and $\Gamma_{2}, \phi; \Delta_{2} \Vdash^{-} \chi$ and $\Gamma_{3}, \psi; \Delta_{3} \Vdash^{-} \chi$ we obtain $\Gamma_1, \Gamma_2, \Gamma_3; \Delta_1, \Delta_2, \Delta_3 \Vdash^- \chi$

\bigskip

\item $(I\rightarrow +)$. If $\Gamma,\phi;\Delta\Vdash^+ \psi$ then $\Gamma;\Delta\Vdash \phi \rightarrow \psi$.

\medskip 

\noindent Assume $\Gamma, \phi;\Delta\Vdash^+ \psi$. Then $\Gamma, \phi;\Delta\Vdash^+_{\mathcal{B}} \psi$ holds for all $\mathcal{B}$ by \ref{c:2int+}. Let $\mathcal{B}$ be some base and $\mathcal{C}$ a base extending $\mathcal{B}$ such that $\Vdash_{\mathcal{C}}^+\gamma$ and $\Vdash_{\mathcal{C}}^-\delta$, for all $\gamma\in\Gamma$ and all $\delta\in\Delta
$. Furthermore, let $\mathcal{D}$ be a base such that $\mathcal{D}\supseteq\mathcal{C}$ and $\Vdash_{\mathcal{D}}^+ \phi$. By monotonicity we have $\Vdash_{\mathcal{D}}^+\gamma$ and $\Vdash_{\mathcal{D}}^-\delta$ for all $\gamma\in\Gamma$ and $\delta\in\Delta
$, so since $\Gamma,\phi;\Delta\Vdash^+_{\mathcal{B}}\psi$, we conclude $\Vdash_{\mathcal{D}}^+ \psi$, by \ref{c:inf+}. By arbitrariness of $\mathcal{D}$ we conclude $\phi ; \emptyset \Vdash^{+}_{\mathcal{C}} \psi$ by \ref{c:inf+}, hence by \ref{c:to+} we have $\Vdash_{\mathcal{C}}^+\phi\rightarrow \psi$, so $\Gamma;\Delta\Vdash_{\mathcal{B}}\phi\rightarrow \psi$ by \ref{c:inf+}. By arbitrariness of $\mathcal{B}$ and \ref{c:2int+}, we conclude $\Gamma;\Delta\Vdash \phi\rightarrow \psi$.

\bigskip 

\item $(E\rightarrow +)$. If $\Gamma_1;\Delta_1\Vdash^+ \phi\rightarrow \psi$ and $\Gamma_2;\Delta_2\Vdash^+ \phi$ then $\Gamma_1,\Gamma_2;\Delta_1,\Delta_2\Vdash^+ \psi$.
\medskip 

\noindent Assume $\Gamma_1; \Delta_1 \Vdash^{+} \phi \rightarrow \psi$ and $\Gamma_2; \Delta_2 \Vdash^{+} \phi$. Then (i) $\Gamma_1; \Delta_1 \Vdash^{+}_{\mathcal{B}} \phi \rightarrow \psi$ and (ii) $\Gamma_2; \Delta_2 \Vdash^{+}_{\mathcal{B}} \phi$ hold for arbitrary $\mathcal{B}$ by \ref{c:2int+}. Let $\B$ be any base and $\mathcal{C}$ an extension of it such that $\Vdash_{\mathcal{C}}^+\gamma$ and $\Vdash_{\mathcal{C}}^-\delta$ for all $\gamma\in\Gamma_1\cup\Gamma_2$ and all $\delta\in\Delta_1\cup\Gamma_2
$. By (i) and \ref{c:inf+}, we get $\Vdash_{\mathcal{C}}^+\phi\rightarrow \psi$, hence $\phi;\emptyset\Vdash_{\mathcal{C}}^+\psi$ by \ref{c:to+}, and by (ii) and \ref{c:inf+} we have $\Vdash_{\mathcal{C}}^+\phi$. We conclude that $\Vdash_{\mathcal{C}}^+\psi$ by \ref{c:inf+}, whence, by arbitrariness of $\mathcal{C}$ and \ref{c:inf+} again, $\Gamma_1,\Gamma_2;\Delta_1,\Delta_2\Vdash_{\mathcal{B}}^+ \psi$. By arbitrariness of $\mathcal{B}$ and \ref{c:2int+},we conclude $\Gamma_1,\Gamma_2;\Delta_1,\Delta_2\Vdash^+ \psi$.

\bigskip
\item $(I\mapsfrom +)$. If $\Gamma_1;\Delta_1\Vdash^+ \phi$ and $\Gamma_2;\Delta_2\Vdash^- \psi$ then $\Gamma_1,\Gamma_2;\Delta_1,\Delta_2\Vdash^+ \phi\mapsfrom \psi$.
\medskip 

    \noindent Assume $\Gamma_1;\Delta_1\Vdash^+ \phi$ and $\Gamma_2;\Delta_2\Vdash^- \psi$. Then, for any $\mathcal{B}$, we have (i) $\Gamma_1;\Delta_1\Vdash_{\mathcal{B}}^+ \phi$ and (ii) $\Gamma_2;\Delta_2\Vdash_{\mathcal{B}}^- \psi$, by \ref{c:2int+}. Let $\B$ be some base, and take $\mathcal{C}$ extending $\mathcal{B}$ such that $\Vdash_{\mathcal{C}}^+\gamma$ and $\Vdash_{\mathcal{C}}^-\delta$, for all $\gamma\in\Gamma_1\cup\Gamma_2$ and all $\delta\in\Delta_1\cup\Delta_2$. By monotonicity and (i), we have $\Gamma_1;\Delta_1\Vdash_{\mathcal{C}}^+ \phi$, and by (ii) we have $\Gamma_1;\Delta_1\Vdash_{\mathcal{C}}^- \psi$, whence $\Vdash_{\C}^+\phi$ and $\Vdash_{\C}^-\psi$, by \ref{c:inf+} and \ref{c:inf-}, respectively. Thus, by \ref{c:mapsfrom+} we have  $\Vdash_{\mathcal{C}}^+\phi\mapsfrom \psi$, so $\Gamma_1,\Gamma_2;\Delta_1,\Delta_2\Vdash_{\mathcal{B}}^+ \phi\mapsfrom \psi$ by \ref{c:inf+} and arbitrariness of $\mathcal{C}$. By arbitrariness of $\mathcal{B}$ and \ref{c:2int+}, we conclude $\Gamma_1,\Gamma_2;\Delta_1,\Delta_2\Vdash^+ \phi\mapsfrom \psi$.

\bigskip

\item $(E\mapsfrom_1 +)$. If $\Gamma;\Delta\Vdash^+ \phi\mapsfrom \psi$ then $\Gamma;\Delta\Vdash^+ \phi$. 

\medskip 

\noindent Assume $\Gamma;\Delta\Vdash^+ \phi\mapsfrom \psi$. Then, by \ref{c:2int+},  $\Gamma;\Delta\Vdash_{\mathcal{B}}^+ \phi\mapsfrom \psi$ holds for any $\mathcal{B}$. Pick any $\mathcal{B}$ and let $\mathcal{C}$ be a base extending $\mathcal{B}$ such that $\Vdash_{\mathcal{C}}^+\gamma$ and $\Vdash_{\mathcal{C}}^-\delta$, for all $\gamma\in\Gamma$ and all $\delta\in\Delta
$. Then, by monotonicity, $\Gamma;\Delta\Vdash_{\mathcal{C}}^+\phi\mapsfrom \psi$, whence $\Vdash_{\mathcal{C}}^+\phi\mapsfrom \psi$ by \ref{c:inf+},  and so $\Vdash_{\mathcal{C}}^+\phi$ by \ref{c:mapsfrom+}. We conclude by arbitrainess of $\mathcal{C}$ and \ref{c:inf+} that $\Gamma;\Delta\Vdash_{\mathcal{B}}^+ \phi$ holds for any base $\mathcal{B}$, whence by \ref{c:2int+} $\Gamma;\Delta\Vdash^+ \phi$.
\bigskip

\item $(E\mapsfrom_2 +)$. If $\Gamma;\Delta\Vdash^+ \phi\mapsfrom \psi$ then $\Gamma;\Delta\Vdash^- \psi$. 

\medskip 
The proof is analogous to the one for $(E \mapsfrom_{1} +)$.
\bigskip

\item The induction steps for $(\bot -)$, $(\top -)$, $(I \lor -)$, $(E_{1} \lor -)$, $(E_{2} \lor -)$, $(I_{1} \land -)$, $(I_{2} \land -)$, $(E \land -)$, $(I \mapsfrom -)$, $(E \mapsfrom -)$, $(I \to-)$, $(E \to_{1} -)$, $(E \to_{2} -)$ are symmetrical with the induction steps for $(\top +)$, $(\bot +)$, $(I \land +)$, $(E_{1} \land +)$, $(E_{2} \land +)$, $(I_{1} \lor +)$, $(I_{2} \lor +)$, $(E \lor +)$, $(I \to +)$, $(E \to +)$, $(I \mapsfrom +)$, $(E \mapsfrom_{1} +)$ and $(E \mapsfrom_{2} +)$, respectively, so their proofs are committed. The steps for $(\top-)$ and $(\land -)$ also use Lemmas \ref{lemma:exfalsotop} and \ref{lemma:conjunctionongeneral} instead of \ref{lemma:exfalsobot} and \ref{lemma:disjunctiongeneral}.

\begin{comment}

\item $(E \bot +)$. If $\Gamma;\Delta\Vdash^+\bot$ then $\Gamma;\Delta\Vdash^+\phi$.

\medskip 

\noindent Assume that $\Gamma;\Delta\Vdash^+\bot$. Then, by \ref{c:2int+}, $\Gamma;\Delta\Vdash_{\mathcal{B}}^+\bot$ holds for any base $\mathcal{B}$. Take $\mathcal{C}$ extending $\mathcal{B}$ such that $\Vdash_{\mathcal{C}}^+\gamma$ and $\Vdash_{\mathcal{C}}^-\delta$, for all $\gamma\in\Gamma$ and all $\delta\in\Delta
$. Then, by monotonicity $\Gamma;\Delta\Vdash_{\mathcal{C}}^+\bot$, whence by \ref{c:inf+} $\Vdash_{\mathcal{C}}^+\bot$ holds. Therefore, by \ref{c:bot+}, $\Vdash_{\mathcal{C}}^+p$ and $\Vdash_{\mathcal{C}}^-p$ hold for every atom $p\in\mathsf{At}$, and so, $\Vdash_{\mathcal{C}}^+\phi$  follows from induction on the structure of $\phi$. Therefore, we have that $\Gamma;\Delta\Vdash_{\mathcal{B}}^+\phi$ by \ref{c:inf+}. Since $\mathcal{B}$ was arbitrary, we conclude by \ref{c:2int+} that $\Gamma;\Delta\Vdash^+\phi$.
\medskip

\end{comment}

        \end{enumerate}
  
\end{proof}

\subsection{Completeness}

We now proceed to prove completeness of bilateral base-extension semantics with respect to the natural deduction system $\BintNA$ for $\Bint$. The proof follows the traditional strategy due to Sandqvist \cite{Sandqvist2015IL} of creating a atomic simulation base capable of mirroring the logical rules of the system.

\begin{definition}
   If $\Gamma$ is a set of formulas, then $\Gamma^{Sub} = \{ \phi \ | \ \phi $ is a subformula of $\psi$ and $\psi \in \Gamma\}$
\end{definition}

\begin{definition}
  An \textit{atomic mapping} $\alpha$ for a set $\Gamma^{Sub}$ is any injective function $\alpha: \Gamma^{Sub} \mapsto \At$ such that, for all $p \in \Gamma^{Sub}$ with $p \in \At$, $\alpha(p) = p$.

\end{definition}

\noindent For every formula $\phi$, we may denote by $p^{\phi}$ the unique atom to which $\phi$ is assigned.

\begin{definition}
    Given an atomic mapping $\alpha$ for $\Gamma$, a \textit{simulation base} $\mathcal{U}$ is a base containing the following atomic rules for every  $q \in \At$, and no other:
\end{definition}

\begin{prooftree}
\AxiomC{$\Gamma, [p^{\phi}]; \Delta$}
\noLine
\UnaryInfC{$\Pi$}
\UnaryInfC{$p^{\psi}$}
\RightLabel{\tiny{$I(+), p^{\phi\to \psi}$}}
\UnaryInfC{$p^{\phi \to \psi}$}
\DisplayProof
\qquad
\AxiomC{$\Gamma_1 ; \Delta_1$}
\noLine
\UnaryInfC{$\Pi_1$}
\UnaryInfC{$p^{\phi\to \psi}$}
\AxiomC{$\Gamma_2; \Delta_{2}$}
\noLine
\UnaryInfC{$\Pi_2$}
\UnaryInfC{$p^{\phi}$}
\RightLabel{\tiny{$ E (+),p^{\phi\to \psi}$}}
\BinaryInfC{$p^{\psi}$}
\end{prooftree}

\begin{prooftree}
\AxiomC{$\Gamma; \Delta$}
\noLine
\UnaryInfC{$\Pi$}
\UnaryInfC{$p^{\phi}$}
\RightLabel{\tiny{$I_{1} (+),p^{\phi\vee \psi}$}}
\UnaryInfC{$p^{\phi\vee \psi}$}
\DisplayProof
\quad
\AxiomC{$\Gamma; \Delta$}
\noLine
\UnaryInfC{$\Pi$}
\UnaryInfC{$p^{\psi}$}
\RightLabel{\tiny{$I_{2} (+),p^{\phi\vee \psi}$}}
\UnaryInfC{$p^{\phi\vee \psi}$}
\DisplayProof

\bigskip

\AxiomC{$\Gamma_1; \Delta_{1}$}
\noLine
\UnaryInfC{$\Pi_1$}
\UnaryInfC{$p^{\phi\vee \psi}$}
\AxiomC{$\Gamma_2, [p^{\phi}]; \Delta_{2}$}
\noLine
\UnaryInfC{$\Pi_2$}
\UnaryInfC{$q$}
\AxiomC{$\Gamma_3, [p^{\psi}]; \Delta_{3}$}
\noLine
\UnaryInfC{$\Pi_3$}
\UnaryInfC{$q$}
\RightLabel{\tiny{$ E_{1} (+), p^{\phi\vee \psi}, q$}}
\TrinaryInfC{$q$}
\DisplayProof

\bigskip

\AxiomC{$\Gamma_1; \Delta_{1}$}
\noLine
\UnaryInfC{$\Pi_1$}
\UnaryInfC{$p^{\phi\vee \psi}$}
\AxiomC{$\Gamma_2, [p^{\phi}]; \Delta_{2}$}
\noLine
\UnaryInfC{$\Pi_2$}
\doubleLine
\UnaryInfC{$q$}
\AxiomC{$\Gamma_3, [p^{\psi}]; \Delta_{3}$}
\noLine
\UnaryInfC{$\Pi_3$}
\doubleLine
\UnaryInfC{$q$}
\doubleLine
\RightLabel{\tiny{$ E_{2} (+), p^{\phi\vee \psi}, q$}}
\TrinaryInfC{$q$}
\end{prooftree}

\begin{prooftree}
\AxiomC{$\Gamma_1; \Delta_{1}$}
\noLine
\UnaryInfC{$\Pi_1$}
\UnaryInfC{$p^{\phi}$}
\AxiomC{$\Gamma_2; \Delta_{2}$}
\noLine
\UnaryInfC{$\Pi_2$}
\UnaryInfC{$p^{\psi}$}
\RightLabel{\tiny{$ I (+),p^{\phi\wedge \psi}$}}
\BinaryInfC{$p^{\phi\wedge \psi}$}
\DisplayProof
\qquad
\AxiomC{$\Gamma; \Delta$}
\noLine
\UnaryInfC{$\Pi$}
\UnaryInfC{$p^{\phi\wedge \psi}$}
\RightLabel{\tiny{$ E_{1} (+),p^{\phi\wedge \psi}$}}
\UnaryInfC{$p^{\phi}$}
\DisplayProof
\qquad
\AxiomC{$\Gamma; \Delta$}
\noLine
\UnaryInfC{$\Pi$}
\UnaryInfC{$p^{\phi\wedge \psi}$}
\RightLabel{\tiny{$ E_{2} (+), p^{\phi\wedge \psi}$}}
\UnaryInfC{$p^{\psi}$}
\end{prooftree}

\begin{prooftree}
\AxiomC{$\Gamma_{1}; \Delta_{1}$}
\noLine
\UnaryInfC{$\Pi_{1}$}
\UnaryInfC{$p^{\phi}$}
\AxiomC{$\Gamma_{1}; \Delta_{1}$}
\noLine
\UnaryInfC{$\Pi_{1}$}
\doubleLine
\UnaryInfC{$p^{ \psi}$}
\RightLabel{\tiny{$I (+),p^{\phi\mapsfrom \psi}$}}
\BinaryInfC{$p^{\phi\mapsfrom \psi}$}
\DisplayProof
\quad
\AxiomC{$\Gamma; \Delta$}
\noLine
\UnaryInfC{$\Pi$}
\UnaryInfC{$p^{\phi\mapsfrom \psi}$}
\RightLabel{\tiny{$ E_{1} (+),p^{\phi\mapsfrom \psi}$}}
\UnaryInfC{$p^{\phi}$}
\DisplayProof
\qquad
\AxiomC{$\Gamma; \Delta$}
\noLine
\UnaryInfC{$\Pi$}
\UnaryInfC{$p^{\phi\mapsfrom \psi}$}
\RightLabel{\tiny{$ E_{2} (+), p^{\phi\mapsfrom \psi}$}}
\doubleLine
\UnaryInfC{$p^{\psi}$}
\end{prooftree}

\begin{prooftree}
\AxiomC{}
\RightLabel{\tiny{$ p^{\top} (+)$}}
\UnaryInfC{$p^{\top}$}
\DisplayProof
\quad
\AxiomC{}
\RightLabel{\tiny{$ p^{\bot} (-)$}}
\doubleLine
\UnaryInfC{$p^{\bot}$}
\end{prooftree}

\begin{prooftree}
\AxiomC{$\Gamma; \Delta$}
\noLine
\UnaryInfC{$\Pi$}
\UnaryInfC{$p^{\bot}$}
\RightLabel{\tiny{$p^{\bot} (+)_{1}, q$}}
\UnaryInfC{$q$}
\DisplayProof
\quad
\AxiomC{$\Gamma; \Delta$}
\noLine
\UnaryInfC{$\Pi$}
\doubleLine
\UnaryInfC{$p^{\top}$}
\doubleLine
\RightLabel{\tiny{$p^{\top} (-), q$}}
\UnaryInfC{$q$}
\DisplayProof
\quad
\AxiomC{$\Gamma; \Delta$}
\noLine
\UnaryInfC{$\Pi$}
\UnaryInfC{$p^{\bot}$}
\doubleLine
\RightLabel{\tiny{$p^{\bot} (+)_{2}, q$}}
\UnaryInfC{$q$}
\DisplayProof
\quad
\AxiomC{$\Gamma; \Delta$}
\noLine
\UnaryInfC{$\Pi$}
\doubleLine
\UnaryInfC{$p^{\top}$}
\RightLabel{\tiny{$p^{\top} (-)_{2}, q$}}
\UnaryInfC{$q$}
\end{prooftree}

\begin{prooftree}
\AxiomC{$\Gamma; \Delta,\llbracket p^{\psi}\rrbracket$}
\noLine
\UnaryInfC{$\Pi$}
\doubleLine
\UnaryInfC{$p^{\phi}$}
\RightLabel{\tiny{$I(-), p^{\phi\mapsfrom \psi}$}}
\doubleLine
\UnaryInfC{$p^{\phi\mapsfrom \psi}$}
\DisplayProof
\qquad
\AxiomC{$\Gamma_1 ; \Delta_1$}
\noLine
\UnaryInfC{$\Pi_1$}
\doubleLine
\UnaryInfC{$p^{\phi\mapsfrom \psi}$}
\AxiomC{$\Gamma_2; \Delta_{2}$}
\noLine
\UnaryInfC{$\Pi_2$}
\doubleLine
\UnaryInfC{$p^{\psi}$}
\RightLabel{\tiny{$ E (-),p^{\phi\mapsfrom \psi}$}}
\doubleLine
\BinaryInfC{$p^{\phi}$}
\end{prooftree}

\begin{prooftree}
\AxiomC{$\Gamma; \Delta$}
\noLine
\UnaryInfC{$\Pi$}
\doubleLine
\UnaryInfC{$p^{\phi}$}
\RightLabel{\tiny{$I_{1} (-),p^{\phi\wedge \psi}$}}
\doubleLine
\UnaryInfC{$p^{\phi\wedge \psi}$}
\DisplayProof
\quad
\AxiomC{$\Gamma; \Delta$}
\noLine
\UnaryInfC{$\Pi$}
\doubleLine
\UnaryInfC{$p^{\psi}$}
\RightLabel{\tiny{$I_{2} (-),p^{\phi\wedge \psi}$}}
\doubleLine
\UnaryInfC{$p^{\phi\wedge \psi}$}
\DisplayProof

\bigskip

\AxiomC{$\Gamma_1; \Delta_{1}$}
\noLine
\UnaryInfC{$\Pi_1$}
\doubleLine
\UnaryInfC{$p^{\phi \wedge \psi}$}
\AxiomC{$\Gamma_2; \Delta_{2}, \llbracket p^{\phi}\rrbracket$}
\noLine
\UnaryInfC{$\Pi_2$}
\doubleLine
\UnaryInfC{$q$}
\AxiomC{$\Gamma_3; \Delta_{3}, \llbracket p^{\psi}\rrbracket$}
\noLine
\UnaryInfC{$\Pi_3$}
\doubleLine
\UnaryInfC{$q$}
\doubleLine
\RightLabel{\tiny{$ E_{1} (-), p^{\phi \wedge \psi}, q$}}
\TrinaryInfC{$q$}
\DisplayProof

\bigskip

\AxiomC{$\Gamma_1; \Delta_{1}$}
\noLine
\UnaryInfC{$\Pi_1$}
\doubleLine
\UnaryInfC{$p^{\phi \wedge \psi}$}
\AxiomC{$\Gamma_2; \Delta_{2}, \llbracket p^{\phi}\rrbracket$}
\noLine
\UnaryInfC{$\Pi_2$}
\UnaryInfC{$q$}
\AxiomC{$\Gamma_3; \Delta_{3}, \llbracket p^{\psi}\rrbracket$}
\noLine
\UnaryInfC{$\Pi_3$}
\UnaryInfC{$q$}
\RightLabel{\tiny{$ E_{2} (-), p^{\phi \wedge \psi}, q$}}
\TrinaryInfC{$q$}
\end{prooftree}

\begin{prooftree}
\AxiomC{$\Gamma_1; \Delta_{1}$}
\noLine
\UnaryInfC{$\Pi_1$}
\doubleLine
\UnaryInfC{$p^{\phi}$}
\AxiomC{$\Gamma_2; \Delta_{2}$}
\noLine
\UnaryInfC{$\Pi_2$}
\doubleLine
\UnaryInfC{$p^{\psi}$}
\doubleLine
\RightLabel{\tiny{$ I (-),p^{\phi \vee \psi}$}}
\BinaryInfC{$p^{\phi \vee \psi}$}
\DisplayProof
\qquad
\AxiomC{$\Gamma; \Delta$}
\noLine
\UnaryInfC{$\Pi$}
\doubleLine
\UnaryInfC{$p^{\phi \vee \psi}$}
\doubleLine
\RightLabel{\tiny{$ E_{1} (-),p^{\phi \vee \psi}$}}
\UnaryInfC{$p^{\phi}$}
\DisplayProof
\qquad
\AxiomC{$\Gamma; \Delta$}
\noLine
\UnaryInfC{$\Pi$}
\doubleLine
\UnaryInfC{$p^{\phi \vee \psi}$}
\doubleLine
\RightLabel{\tiny{$ E_{2} (-), p^{\phi \vee \psi}$}}
\UnaryInfC{$p^{\psi}$}
\end{prooftree}

\begin{prooftree}
\AxiomC{$\Gamma_{1}; \Delta_{1}$}
\noLine
\UnaryInfC{$\Pi_{1}$}
\UnaryInfC{$p^{\phi}$}
\AxiomC{$\Gamma_{1}; \Delta_{1}$}
\noLine
\UnaryInfC{$\Pi_{1}$}
\doubleLine
\UnaryInfC{$p^{ \psi}$}
\doubleLine
\RightLabel{\tiny{$I (-),p^{\phi \to \psi}$}}
\BinaryInfC{$p^{\phi \to \psi}$}
\DisplayProof
\quad
\AxiomC{$\Gamma; \Delta$}
\noLine
\UnaryInfC{$\Pi$}
\doubleLine
\UnaryInfC{$p^{\phi \to \psi}$}
\RightLabel{\tiny{$ E_{1} (-),p^{\phi \to \psi}$}}
\UnaryInfC{$p^{\phi}$}
\DisplayProof
\qquad
\AxiomC{$\Gamma; \Delta$}
\noLine
\UnaryInfC{$\Pi$}
\doubleLine
\UnaryInfC{$p^{\phi \to \psi}$}
\RightLabel{\tiny{$ E_{2} (-), p^{\phi \to \psi}$}}
\doubleLine
\UnaryInfC{$p^{\psi}$}
\end{prooftree}

\bigskip
\bigskip
\bigskip

\begin{lemma}\label{lemma:simulationbase}
Let $\phi \in \Gamma^{Sub}$, $\alpha$ be an atomic mapping for $\Gamma^{Sub}$ and $\mathcal{U}$ a simulation base defined with $\alpha$. Then, $\forall \mathcal{B} (\mathcal{B} \supseteq \mathcal{U})$, ($\Vdash^{*}_{\mathcal{B}} \phi$ iff $\vdash^{*}_{\mathcal{B}} p^{\phi}$).
\end{lemma}

\begin{proof}
We prove the result by induction on the complexity of formulas, understood as the number of logical operators occurring on them.

\begin{enumerate}
    \item (Base case) The result follows trivially from the clauses \ref{c:at+} and \ref{c:at-} and the fact that $\alpha(p^{\chi}) = \chi$ whenever $\chi \in \At$.

    \bigskip
    \item $(\chi=\top, * = +)$. 
    \medskip

    \begin{enumerate}
        \item[$(\Rightarrow)$] Observe that  $\Vdash^{+}_{\mathcal{B}} \top$ always holds by \ref{c:top+} for any base $\mathcal{B}$. We must show that $\vdash^{+}_{\mathcal{B}} p^{\top}$ is always the case too. Note that the rule $ p^{\top} (+)$ is in $\mathcal{U}$, hence it is also in any base $\mathcal{B}$ extending $\mathcal{U}$. Therefore, by applying this rule, we obtain a proof of $\vdash^{+}_{\mathcal{B}} p^{\top}$ for any $\mathcal{B} \supseteq \mathcal{U}$. 
        \medskip

        \item[$(\Leftarrow)$] Assume $\vdash^{+}_{\mathcal{B}} p^{\top}$. Then $\Vdash^{+}_{\mathcal{B}} \top$, as this hold for any base $\mathcal{B}$ by \ref{c:top+}.
        \medskip
        
     %   \item Assume $\Vdash^{-}_{\mathcal{B}} \top$. Then, we have that $\Vdash^{-}_{\mathcal{B}} q$, for any $q\in\mathsf{At}$. In particular, this must hold for $p^{\top}$, hence  $\Vdash^{-}_{\mathcal{B}} p^{\top}$, and therefore $\vdash^{-}_{\mathcal{B}} p^{\top}$.
      %  \medskip 
        
  %      \item Assume that $\vdash^{-}_{\mathcal{B}} p^{\top}$. We have that the rule $p^{\top}(-),q$ is in $\mathcal{U}$, for any $q\in\mathsf{At}$. Since $\mathcal{U}\subseteq\mathcal{B}$, we can use this rule to conclude that there is a refutation of $q$ from empty premises, that is $\vdash^{-}_{\mathcal{B}} q$. Hence, we conclude that $\Vdash^{-}_{\mathcal{B}} \top$.  
    \end{enumerate}
    \medskip
    
    \item $(\chi=\bot, * = +)$. 
    \medskip 
    
    \begin{enumerate}
        \item[$(\Rightarrow)$] Assume $\Vdash^{+}_{\mathcal{B}} \bot$. Then by \ref{c:bot+} it holds that $\Vdash^{+}_{\mathcal{B}} q$ and $\Vdash^{-}_{\mathcal{B}} q$ for any $q\in\mathsf{At}$. In particular, $\Vdash^{+}_{\mathcal{B}}p^{\bot}$, thus $\vdash^{+}_{\mathcal{B}}p^{\bot}$ by \ref{c:at+}.
        \medskip
        
        \item[$(\Leftarrow)$] Assume $\vdash^{+}_{\mathcal{B}}p^{\bot}$. By applying the rule $p^{\bot}(+)_{1},q$ at the end of the proof of $p^{\bot}$ we get a proof showing $\vdash^{+}_{\mathcal{B}}q$ for each $q\in\mathsf{At}$, and by applying the rule $p^{\bot}(+)_{2},q$ we also get a refutation showing $\vdash^{-}_{\mathcal{B}}q$ for each $q\in\mathsf{At}$. By \ref{c:at+} and \ref{c:at-} we conclude $\Vdash^{+}_{\mathcal{B}}q$ and $\Vdash^{-}_{\mathcal{B}}q$ for every $q \in \At$, so we conclude $\Vdash^{+}_{\mathcal{B}} \bot$ by \ref{c:bot+}.
        
        \medskip
        
     %   \item Similarly to case $2(a)$, we have that $\Vdash^{-}_{\mathcal{B}}\bot$ holds for all $\mathcal{B}$. On the other hand, a simulation base $\mathcal{U}$ contains the rule $p^{\bot}(-)$, so any base $\mathcal{B}$ such that $\mathcal{U}\subseteq\mathcal{B}$ also contains this rule. Hence, we can apply it, obtaining a proof of $\vdash^{-}_{\mathcal{B}}p^{\bot}$ for any such $\mathcal{B}$.
    %    \medskip 

   %     \item Assume that $\vdash^{-}_{\mathcal{B}}p^{\bot}$. Then $\Vdash^{-}_{\mathcal{B}} \bot$, as this hold for any base $\mathcal{B}$. 
    \end{enumerate}
    \medskip
    
    \item $(\chi = \phi \lor \psi, * = +)$.

    \medskip

    \begin{enumerate}
      \item[$(\Rightarrow)$] Assume $\Vdash^{+}_{\mathcal{B}} \phi\lor \psi$. By \ref{c:or+} we have that, for all extensions $\mathcal{C}$ of $\mathcal{B}$ and for every $p \in \At$, $\phi ; \emptyset \Vdash^{+}_{\mathcal{C}} p$ and $\psi ; \emptyset \Vdash^{+}_{\mathcal{C}} p$ implies $\Vdash^{+}_{\mathcal{C}} p$. Let $\mathcal{C}$ be an arbitrary extension of $\mathcal{B}$ with $\Vdash^{+}_{\mathcal{C}} \phi$. The induction hypothesis yields $\vdash^{+}_{\mathcal{C}} p^{\phi}$, so by applying $I_{1} (+), p^{\phi \lor \psi}$ at the end of the proof showing $p^{\phi}$ we get a proof showing $\vdash^{+}_{\mathcal{C}} p^{\phi \lor \psi}$, hence we conclude $\Vdash^{+}_{\mathcal{C}} p^{\phi \lor \psi}$ by \ref{c:at+}. Since $\mathcal{C}$ was an arbitrary extension of $\mathcal{B}$ with $\Vdash^{+}_{\mathcal{C}} \phi$ we conclude $\phi ; \emptyset \Vdash^{+}_{\mathcal{B}} p^{\phi \lor \psi}$ by \ref{c:inf+}. A similar argument yields $\psi ; \emptyset \Vdash^{+}_{\mathcal{B}} p^{\phi \lor \psi}$. Since $\Vdash^{+}_{\mathcal{B}} \phi\lor \psi$,
       holds as well as $\phi ; \emptyset \Vdash^{+}_{\mathcal{B}} p^{\phi \lor \psi}$ and $\psi ; \emptyset \Vdash^{+}_{\mathcal{B}} p^{\phi \lor \psi}$ we conclude $\Vdash^{+}_{\mathcal{B}} p^{\phi\lor \psi}$ by \ref{c:or+}, so $\vdash^{+}_{\mathcal{B}} p^{\phi\lor \psi}$ by \ref{c:at+}.

      \medskip

      \item[$(\Leftarrow)$] Assume $\vdash^{+}_{\mathcal{B}} p^{\phi \lor \psi}$. Let $\mathcal{C}$ be any extension of $\mathcal{B}$ with $\phi ; \emptyset \Vdash^{+}_{\mathcal{C}} q$ and $\psi ; \emptyset \Vdash^{+}_{\mathcal{C}} q$ for some atom $q$. Let $\mathcal{D}$ be any extension of $\mathcal{C}$ with $\Vdash^{+}_{\mathcal{D}} p^{\phi}$. Then $\vdash^{+}_{\mathcal{D}} p^\phi$ by \ref{c:at+}, so the induction hypothesis yields $\Vdash^{+}_{\mathcal{D}} \phi$. Since $\phi ; \emptyset \Vdash_{\mathcal{C}} q$ and $\mathcal{D} \supseteq \mathcal{C}$ we conclude $\Vdash^{+}_{\mathcal{D}} q$ by \ref{c:inf+}. But $\mathcal{D}$ was an arbitrary extension of $\mathcal{C}$ with $\Vdash^{+}_{\mathcal{D}} p^{\phi}$, so we conclude $p^{\phi};\emptyset\Vdash^{+}_{\mathcal{C}} q$ by \ref{c:inf+}, whence by Lemma \ref{lemma:atomicsupportiffderivability} we have $p^{\phi} ; \emptyset \vdash^{+}_{\mathcal{C}} q$.  A similar argument yields $p^{\psi} ; \emptyset \vdash^{+}_{\mathcal{C}} q$. Since $\mathcal{C} \supseteq \mathcal{B}$ and $\vdash^{+}_{\mathcal{B}} p^{\phi \lor \psi}$ we also have $\vdash^{+}_{\mathcal{C}} p^{\phi \lor \psi}$, hence we can apply $E_{1} (+), p^{\phi \lor \psi}, q$ using as premises the deductions showing $\vdash^{+}_{\mathcal{C}} p^{\phi \lor \psi}$, $p^{\phi} ; \emptyset \vdash^{+}_{\mathcal{C}} q$ and $p^{\psi} ; \emptyset \vdash^{+}_{\mathcal{C}} p$ to obtain a deduction showing $\vdash^{+}_{\mathcal{C}} q$, whence $\Vdash^{+}_{\mathcal{C}} q$ by \ref{c:at+}. A similar argument shows that, for any $\mathcal{C} \supseteq \mathcal{B}$ with $\phi ; \emptyset \Vdash^{-}_{\mathcal{C}} q$ and $\psi ; \emptyset \Vdash^{-}_{\mathcal{C}} q$ for some atom $q$, we can get a refutation showing $\vdash^{-}_{\mathcal{C}} q$ through an application of $E_{2} (+), p^{\phi \lor \psi}, q$, hence $\Vdash^{-}_{\mathcal{C}} q$ by \ref{c:at-}. Since $\mathcal{C}$ and $q$ are arbitrary we have shown that, for all $\mathcal{C} \supseteq \mathcal{B}$ and all $q \in \At$, $(\phi; \emptyset \Vdash_{\mathcal{C}}^+ q$ and $\psi; \emptyset \Vdash_{\mathcal{C}}^+ q$ implies $\Vdash_{\mathcal{C}}^+ q)$ and $(\phi; \emptyset \Vdash_{\mathcal{C}}^- q$ and $\psi; \emptyset \Vdash_{\mathcal{C}}^- q$ implies $\Vdash_{\mathcal{C}}^- q)$, hence $\Vdash^{+}_{\mathcal{B}} \phi \lor \psi$ by \ref{c:or+}.

      %\medskip

      %\item   Assume $\Vdash^{-}_{\mathcal{B}} \psi\lor \chi$. Then, by the clause for refutation of disjunction, $\Vdash^{-}_{\mathcal{B}} \psi$ and $\Vdash^{-}_{\mathcal{B}} \chi$. Induction hypothesis: $\vdash^{-}_{\mathcal{B}} p^{\psi}$ and $\vdash^{-}_{\mathcal{B}} p^{\chi}$. Then we simply apply $I(-), p^{\psi \lor \chi}$ to obtain a deduction showing $\vdash^{-}_{\mathcal{B}}p^{\psi \lor \chi}$.

     % \medskip

    %  \item Assume $\vdash^{-}_{\mathcal{B}}p^{\psi \lor \chi}$, and let $\Pi$ be a refutation of $p^{\psi \lor \chi}$ in $\mathcal{B}$. By applying $E_{1}(-),p^{\psi \lor \chi}$ to the end of $\Pi$ we get a refutation of $p^{\psi}$, and by applying $E_{2}(-),p^{\psi \lor \chi}$ we get one of $p^{\chi}$, so we conclude $\vdash^{-}_{\mathcal{B}} p^{\psi}$ and $\vdash^{-}_{\mathcal{B}} p^{\chi}$. Induction hypothesis: $\Vdash^{-}_{\mathcal{B}} \psi$ and $\Vdash^{-}_{\mathcal{B}} \chi$, which by the clause for refutation of disjunction yields $\Vdash^{-}_{\mathcal{B}} \psi\lor \chi$.

  \end{enumerate}
\medskip
  
  \item $(\chi=\phi\wedge \psi, * = +).$
  \medskip
  \begin{enumerate}
      \item[$(\Rightarrow)$] Assume $\Vdash^{+}_{\mathcal{B}}\phi\wedge \psi$. Then $\Vdash^{+}_{\mathcal{B}}\phi$ and $\Vdash^{+}_{\mathcal{B}}\psi$ by \ref{c:and+}. The induction hypothesis yields $\vdash^{+}_{\mathcal{B}}p^\phi$ and $\vdash^{+}_{\mathcal{B}}p^\psi$, so by applying $I(+),p^{\phi\wedge \psi}$ we obtain a proof of $\vdash^{+}_{\mathcal{B}}p^{\phi\wedge \psi}$.
      
      \medskip
      
      \item[$(\Leftarrow)$] Assume $\vdash^{+}_{\mathcal{B}}p^{\phi\wedge \psi}$. By applying $E_1(+),p^{\phi\wedge \psi}$ at the end of this proof we show $\vdash^{+}_{\mathcal{B}} p^\phi$. Similarly, we obtain a proof showing $\vdash^{+}_{\mathcal{B}} p^\psi$ by applying $E_2(+),p^{\phi\wedge \psi}$. The induction hypothesis yields $\Vdash^{+}_{\mathcal{B}}\phi$ and $\Vdash^{+}_{\mathcal{B}}\psi$, hence $\Vdash^{+}_{\mathcal{B}}\phi\wedge \psi$ by \ref{c:and+}.
  \end{enumerate}
  \medskip
  
  \item $(\chi=\phi\rightarrow \psi, * = +).$
  \medskip
  \begin{enumerate}
      \item[$(\Rightarrow)$] Assume $\Vdash^{+}_{\mathcal{B}}\phi\rightarrow \psi$. By \ref{c:to+} we have $\phi; \emptyset \Vdash^{+}_{\mathcal{B}} \psi$. Let $\mathcal{C}$ be a base extending $\mathcal{B}$ such that $\Vdash^{+}_{\mathcal{C}}p^\phi$. Then  $\vdash^{+}_{\mathcal{C}}p^\phi$ by \ref{c:at+}, so the induction hypothesis yields $\Vdash^{+}_{\mathcal{C}}\phi$, hence we conclude $\Vdash^{+}_{\mathcal{C}}\psi$ by \ref{c:inf+}. The induction hypothesis yields $\vdash^{+}_{\mathcal{C}}p^\psi$ and thus $\Vdash^{+}_{\mathcal{C}}p^\psi$ by \ref{c:at+}, whence by \ref{c:inf+} and arbitrariness of $\mathcal{C}$ we conclude $p^{\phi}; \emptyset \Vdash^{+}_{\mathcal{B}} p^{\psi}$. By Lemma \ref{lemma:atomicsupportiffderivability} we have $p^{\phi}; \emptyset \vdash^{+}_{\mathcal{B}} p^{\psi}$, so we can apply $I(+),p^{\phi\rightarrow \psi}$ at the end of the deduction showing this to obtain a proof showing $\vdash^{+}_{\mathcal{B}} p^{\phi\rightarrow \psi}$.
      \medskip
      
      \item[$(\Leftarrow)$] Assume  $\vdash^{+}_{\mathcal{B}}p^{\phi\rightarrow \psi}$. Take a base $\mathcal{C}$ extending $\mathcal{B}$ such that $\Vdash^{+}_{\mathcal{C}}\phi$. The induction hypothesis yields $\vdash^{+}_{\mathcal{C}}p^\phi$. Using the proofs of $p^{\phi\rightarrow \psi}$ and $p^\phi$ as premises, we can apply rule $E(+),p^{\phi\rightarrow \psi}$ to obtain a proof of $p^{\psi}$, showing $\vdash^{+}_{\mathcal{C}}p^\psi$. The induction hypothesis then yields $\Vdash^{+}_{\mathcal{C}}\psi$, so we can conclude $\phi,\emptyset\Vdash^{+}_{\mathcal{B}}\psi$ by \ref{c:inf+} and arbitrariness of $\mathcal{C}$, whence $\Vdash^{+}_{\mathcal{C}} \phi \to \psi$ by \ref{c:to+}.
%      \medskip
      
%      \item Assume that $\Vdash^{-}_{\mathcal{B}}\psi\rightarrow \chi$. Then, by the semantic clause for refutation of implication $\Vdash^{+}_{\mathcal{B}}\psi$ and $\Vdash^{-}_{\mathcal{B}}\chi$. By the induction hypothesis, $\vdash^{+}_{\mathcal{B}}p^\psi$ and $\vdash^{-}_{\mathcal{B}}p^\chi$. Then, we can apply $I(-),p^{\psi\rightarrow \chi}$, using as premises the deductions showing $\vdash^{+}_{\mathcal{B}}p^\psi$ and $\vdash^{-}_{\mathcal{B}}p^\chi$ to obtain a deduction showing $\vdash^{-}_{\mathcal{B}}p^{\psi\rightarrow \chi}$.
  %    \medskip
      
    %  \item Assume $\vdash^{-}_{\mathcal{B}}p^{\psi\rightarrow \chi}$. Applying $E_1(-),p^{\psi\rightarrow \chi}$ using the deduction showing $\vdash^{-}_{\mathcal{B}}p^{\psi\rightarrow \chi}$ as premiss, we obtain a deduction showing $\vdash^{+}_{\mathcal{B}}p^\psi$. Similarly, we can obtain a deduction showing $\vdash^{-}_{\mathcal{B}}p^\chi$. By the induction hypothesis, we have that $\Vdash^{+}_{\mathcal{B}}\psi$ and $\Vdash^{-}_{\mathcal{B}}\chi$, hence $\Vdash^{-}_{\mathcal{B}}\psi\rightarrow \chi$.
  \end{enumerate}
  \medskip 
  
  \item $(\chi=\phi\mapsfrom \psi, * = +)$
  \medskip
  \begin{enumerate}
      \item[$(\Rightarrow)$] Assume $\Vdash^{+}_{\mathcal{B}}\phi \mapsfrom \psi$. By \ref{c:mapsfrom+} we have $\Vdash^{+}_{\mathcal{B}}\phi$ and $\Vdash^{-}_{\mathcal{B}}\psi$, so the induction hypothesis yields $\vdash^{+}_{\mathcal{B}}p^\phi$ and $\vdash^{-}_{\mathcal{B}}p^\psi$. Using these deductions as premises of an application of $I(+),p^{\phi\mapsfrom \psi}$ we obtain a proof of $p^{\phi\mapsfrom \psi}$, showing $\vdash^{+}_{\mathcal{B}}p^{\phi\mapsfrom \psi}$.
      \medskip
      
      \item[$(\Leftarrow)$] Assume $\vdash^{+}_{\mathcal{B}}p^{\phi \mapsfrom \psi}$. By applying $E_1(+),p^{\phi\mapsfrom \psi}$ at the end of this proof we obtain $\vdash^{+}_{\mathcal{B}} p^{\phi\mapsfrom \psi}$, and by applying $E_2(+),p^{\phi\mapsfrom \psi}$ we obtain $\vdash^{-}_{\mathcal{B}}p^\psi$. The induction hypothesis yields $\Vdash^{+}_{\mathcal{B}} \phi$ and $\Vdash^{-}_{\mathcal{B}} \psi$, whence $\Vdash^{+}_{\mathcal{B}} \phi\mapsfrom \psi$ by \ref{c:mapsfrom+}.

  \end{enumerate}

  \medskip

  \item The inductive steps for $( \chi = \top, *= -)$, $( \chi = \bot, *= -)$, $( \chi = \phi \lor \psi, *= -)$, $( \chi = \phi \land \psi, *= -)$, $( \chi = \phi \to \psi, *= -)$, $( \chi = \phi \mapsfrom \psi, *= -)$ are respectively symmetrical with $( \chi = \bot, *= +)$, $( \chi = \top, *= +)$, $( \chi = \phi \land \psi, *= +)$, $( \chi = \phi \lor \psi, *= +)$,  $( \chi = \phi \mapsfrom \psi, *= +)$, $( \chi = \phi \to \psi, *= +)$, respectively, so we omit their proofs.
\end{enumerate}

\end{proof}

\begin{theorem}[Completeness]
    If $\Gamma;\Delta\Vdash^* \phi$ then $\Gamma;\Delta\vdash^*_{\BintNA} \phi$
\end{theorem}

\begin{proof}
    Assume $\Gamma;\Delta\Vdash^* \phi$. Let $\Theta=(\Gamma\cup\Delta\cup\{\phi\})^{Sub}$. Furthermore, let $\alpha$ be an atomic mapping for $\Theta$ and $\mathcal{U}$ a simulation base based on $\alpha$. Then, $\Gamma;\Delta\Vdash_{\mathcal{U}}^* \phi$ holds by either \ref{c:2int+} or \ref{c:2int-}. Let $\mathcal{B}$ be a base extending $\mathcal{U}$ such that $\Vdash_{\B}^+ p^{\gamma}$ for every $\gamma\in\Gamma$, and $\Vdash_{\B}^- p^{\delta}$, for every $\delta\in\Delta$. Then, by \ref{c:at+} and \ref{c:at-} together with Lemma \ref{lemma:simulationbase}, we get that $\Vdash_{\B}^+ {\gamma}$ for every $\gamma\in\Gamma$, and $\Vdash_{\B}^- {\delta}$, for every $\delta\in\Delta$. Since $\Gamma;\Delta\Vdash_{\mathcal{U}}^* \phi$, by either \ref{c:inf+} or \ref{c:inf-} we conclude $\Vdash_{\B}^* \phi$, whence $\vdash_{\B}^* p^\phi$ again by Lemma \ref{lemma:simulationbase}. By arbitrariness of $\mathcal{B}$ and either \ref{c:inf+} or \ref{c:inf-} we conclude $\{p^{\gamma}\}_{\gamma\in\Gamma};\{p^{\Delta}\}_{\delta\in\Delta}\Vdash_{\mathcal{U}}^* p^\phi$. By Lemma \ref{lemma:atomicsupportiffderivability} this is equivalent to $\{p^{\gamma}\}_{\gamma\in\Gamma};\{p^{\Delta}\}_{\delta\in\Delta}\vdash_{\mathcal{U}}^* p^{\phi}$. To conclude this proof, note that every rule in $\mathcal{U}$ corresponds to an instance of a rule in $\BintNA$, hence by replacing every formula occurrence with shape $p^{\chi}$ by $\chi$ in the derivation showing $\{p^{\gamma}\}_{\gamma\in\Gamma};\{p^{\Delta}\}_{\delta\in\Delta}\vdash_{\mathcal{U}}^* p^{\phi}$ we obtain a deduction showing $\Gamma;\Delta\vdash^*_{\BintNA} \phi$.

    %Consider the set $\Gamma\cup\Delta\cup\{A\}$, an atomic mapping $\alpha$ for this set, and a simulation base $\mathcal{U}$ based on $\alpha$.
\end{proof}
\subsection{Semantic harmony}

From the fact that proof and refutation rules establish assertion and rejection conditions for the \textit{same} logical connectives it follows that their choice cannot be arbitrary. Assertion and rejection are independent but also \textit{opposite} acts, so the proof conditions for a connective should also be meaningfully opposite to its refutation conditions.

A syntactic proof-theoretic characterization of this opposition is presented in \cite{francez2014bilateralism}, in which a notion of duality between formulas is used to justify the formulation of a horizontal inversion principle (``there must be a harmonious relation between the proof and refutation conditions for a connective"), introduced as a bilateralist version of Prawitz's traditional (vertical) inversion principle (``there must be a harmonious relation between the introduction and elimination rules of a connective") \cite{prawitz1965}. The paper shows that this opposition essentially consists in the requirement that there must exist an isomorphism between the proof conditions for a connective and the refutation conditions for its dual, which leads to bilaterally harmonious proof-theoretic behavior.

In this section we show that bilateral base-extension semantics allows us to provide \textit{purely semantic} (\textit{but still proof-theoretical}) characterizations of bilateral harmony if the notion of duality is extended from logical connectives to atomic rules, derivations and  bases. Turning back to Schroeder-Heister's distinction between proof-theoretic semantics as the study of \textit{semantics of proofs} and of \textit{semantics in terms of proofs}, this means that bilateral base-extension semantics, a semantics in terms of proofs and refutations, can also be used to study the semantics of proofs and refutations themselves, evidentiating that both aspects are intertwined also in a bilateralist setting.

Just like duality for a formula is defined by taking the formula with opposite proof and refutation conditions, the dual of any rule is given by the rule that opposes it in every possible way. The dual of a base is obtained by taking all rules dual to its own rules, and the dual of a deduction by replacing every proof assumption by a refutation assumption (and vice-versa), as well as every rule by its dual. Once this is done, it becomes possible to show that every deduction in a base has an isomorphic dual deduction with the opposite assertoric force in its dual base, so every provable atoms becomes refutable and every refutable atoms becomes provable. The opposition that exists between a formula and its dual then allows us to extend the results from atoms to all non-atomic formulas, meaning that the proof-theoretic harmony between a formula and its dual results in a  \textit{semantic harmony} between theirs proof and refutation clauses.

\begin{definition} \label{def:dualrules}
    Given any atomic proof rule $R{+}$ and any atomic refutation rule $R{-}$, their duals $(R+)^{\mathbb{D}}$ and  $(R-)^{\mathbb{D}}$ are defined as follows:

\begin{prooftree}
\AxiomC{$[\Gamma^{1}_{\At}]; \llbracket \Delta^{1}_{\At} \rrbracket$}
\noLine
\UnaryInfC{.}
\noLine
\UnaryInfC{.}
\noLine
\UnaryInfC{.}
\UnaryInfC{$p_1$}
\AxiomC{$\ldots$}
\AxiomC{$[\Gamma^{n}_{\At}]; \llbracket \Delta^{n}_{\At} \rrbracket$}
\noLine
\UnaryInfC{.}
\noLine
\UnaryInfC{.}
\noLine
\UnaryInfC{.}
\doubleLine
\UnaryInfC{$p_n$}
\RightLabel{\tiny{$R+$}}
\TrinaryInfC{$p$}
\DisplayProof \quad 
$\Longrightarrow$
\AxiomC{$[\Delta^{1}_{\At}];  \llbracket \Gamma^{1}_{\At} \rrbracket$}
\noLine
\UnaryInfC{.}
\noLine
\UnaryInfC{.}
\noLine
\UnaryInfC{.}
\doubleLine
\UnaryInfC{$p_1$}
\AxiomC{$\ldots$}
\AxiomC{$[\Delta^{n}_{\At}] ; \llbracket \Gamma^{n}_{\At} \rrbracket$}
\noLine
\UnaryInfC{.}
\noLine
\UnaryInfC{.}
\noLine
\UnaryInfC{.}
\UnaryInfC{$p_n$}
\doubleLine
\RightLabel{\tiny{$(R+)^{\mathbb{D}}$}}
\TrinaryInfC{$p$}
\end{prooftree}

\begin{prooftree}
\AxiomC{$[\Gamma^{1}_{\At}]; \llbracket \Delta^{1}_{\At} \rrbracket$}
\noLine
\UnaryInfC{.}
\noLine
\UnaryInfC{.}
\noLine
\UnaryInfC{.}
\UnaryInfC{$p_1$}
\AxiomC{$\ldots$}
\AxiomC{$[\Gamma^{n}_{\At}]; \llbracket \Delta^{n}_{\At} \rrbracket$}
\noLine
\UnaryInfC{.}
\noLine
\UnaryInfC{.}
\noLine
\UnaryInfC{.}
\doubleLine
\UnaryInfC{$p_n$}
\RightLabel{\tiny{$R-$}}
\doubleLine
\TrinaryInfC{$p$}
\DisplayProof \quad 
$\Longrightarrow$
\AxiomC{$[\Delta^{1}_{\At}];  \llbracket \Gamma^{1}_{\At} \rrbracket$}
\noLine
\UnaryInfC{.}
\noLine
\UnaryInfC{.}
\noLine
\UnaryInfC{.}
\doubleLine
\UnaryInfC{$p_1$}
\AxiomC{$\ldots$}
\AxiomC{$[\Delta^{n}_{\At}] ; \llbracket \Gamma^{n}_{\At} \rrbracket$}
\noLine
\UnaryInfC{.}
\noLine
\UnaryInfC{.}
\noLine
\UnaryInfC{.}
\UnaryInfC{$p_n$}
\RightLabel{\tiny{$(R-)^{\mathbb{D}}$}}
\TrinaryInfC{$p$}
\end{prooftree}

\end{definition}

\begin{definition}\label{def:dualdeductions}
    Let $\Pi$ be an atomic deduction. Its dual $(\Pi)^{\mathbb{D}}$ is defined as follows:

    \begin{enumerate}
       \item ($\Pi$ has $0$ rule applications). If $\Pi$ consists in a single proof assumption with shape $p$ then $(\Pi)^{\mathbb{D}}$ consists in a single refutation assumption with shape $p$, and if $\Pi$ consists in a single refutation assumption with shape $p$ then $(\Pi)^{\mathbb{D}}$ consists in a single proof assumption with shape $p$.

       \medskip

       \item ($\Pi$ has $n > 0$ rule application). If $R+$ is the last rule applied in $\Pi$ and has the deductions $\Pi^{1}, \ldots, \Pi^n$ above its premises then $(R+)^{\mathbb{D}}$ is the last rule applied in $(\Pi)^{\mathbb{D}}$, which has the deductions $(\Pi^{1})^{\mathbb{D}}, \ldots, (\Pi^n)^{\mathbb{D}}$ above its premises (in the same order). If $R-$ is the last rule applied in $\Pi$ and has the deductions $\Pi^{1}, \ldots, \Pi^n$ above its premises then $(R-)^{\mathbb{D}}$ is the last rule applied in $(\Pi)^{\mathbb{D}}$, which has the deductions $(\Pi^{1})^{\mathbb{D}}, \ldots, (\Pi^n)^{\mathbb{D}}$ above its premises (in the same order).

        %\item ($\Pi$ has no rule applications). If $\Pi$ consists in a single proof assumption with shape $p$ then    Every proof assumption with shape $p$ is replaced by a refutation assumption with shape $p$ for every $p \in \At$, and every refutation assumption with shape $q$ is replaced by a proof assumption with shape $q$ for every $q \in \At$.

     %   \medskip

    %    \item Every application of a atomic rule is replaced by an application of its dual.
    \end{enumerate}
\end{definition}

\begin{definition}\label{def:dualbases}
    Let $\mathcal{B}$ be any atomic base. Then $(\mathcal{B})^{\mathbb{D}} = \{(R)^{\mathbb{D}} | R$ is an atomic rule of $\mathcal{B} \}$.
\end{definition}

From Definitions \ref{def:bilateralbase}, \ref{def:dualrules} and \ref{def:dualbases} it follows that the duals of rules are rules and the duals of bases are bases. We quickly show that the same holds for duals of deductions:

\begin{proposition} \label{prop:dualsdeductionsaredeductions}
If $\Pi$ is proof deduction with conclusion $p$ then $(\Pi)^{\mathbb{D}}$ is a refutation deduction with conclusion $p$, and if $\Pi$ is a refutation deduction with conclusion $p$ then $(\Pi)^{\mathbb{D}}$ is a proof deduction with conclusion $p$.
\end{proposition}

\begin{proof}
 The result is shown by induction on the length of $\Pi$. If $\Pi$ is of length $0$ the result is immediate by Definitions \ref{def:dualdeductions}, \ref{def:deducatomicproofs} and \ref{def:deducatomicrefs}. If $\Pi$ has length $n >0$, let $R*$ be the last rule applied in it and $\Pi^{1}, \ldots , \Pi^{n}$ be the deductions occurring above its premises, which have shape $p_{1}, \ldots, p_{n}$. By the induction hypothesis we have that $(\Pi^{1})^{\mathbb{D}}, \ldots , (\Pi^{n})^{\mathbb{D}}$ are deductions with conclusions $p_{1}, \ldots, p_{n}$.  Since by Definition \ref{def:dualrules} the rule $(R*)^{\mathbb{D}}$ has premises with the same shape $p_{1}, \ldots, p_{n}$ of the rule $R*$ but requires a proof deduction of the premise whenever $R*$ required a refutation deduction of the premise (and vice-versa), the application of $(R*)^{\mathbb{D}}$ with premises $(\Pi^{1})^{\mathbb{D}}, \ldots , (\Pi^{n})^{\mathbb{D}}$ yields a valid deduction, hence from the definition of $(R*)^{\mathbb{D}}$ and Definitions \ref{def:deducatomicproofs} and \ref{def:deducatomicrefs} if follows that if $\Pi$ was a proof deduction then $(\Pi)^{\mathbb{D}}$ is a refutation deduction and if $\Pi$ was a refutation deduction then $(\Pi)^{\mathbb{D}}$ is a proof deduction.

\end{proof}

  %Intuitively, the dual of a deduction is obtained by replacing all its assumptions by assumptions of opposite assertoric value and all rules by their duals.
 
Before proceeding to the results, we briefly fix a new placeholder sign $(*)^{\mathbb{D}}$ by defining that if $(* = +)$ then $((*)^{\mathcal{D}} = -)$, and if $(* = -)$ then $((*)^{\mathcal{D}} = +)$.

 %From Definitions \ref{def:dualdeductions} and \ref{def:dualrules} it follows that if $\Pi'$ ends with a proof rule concluding $p$ then $(\Pi')^{\mathbb{D}}$ ends with a refutation rule concluding $p$, and if $\Pi'$ ends with a refutation rule concluding $p$ then $(\Pi')^{\mathbb{D}}$ ends with a proof rule concluding $p$. 

\begin{lemma} \label{lemma:fundamentalatomicdualitylemma}
    If $\Pi$ is a deduction showing $\Gamma_{\At}; \Delta_{\At} \vdash^{*}_{\mathcal{B}} p$ then $(\Pi)^{\mathbb{D}}$ is a deduction showing $\Delta_{\At} ; \Gamma_{\At} \vdash^{(*)^{\mathbb{D}}}_{\mathcal{(B)^{\mathbb{D}}}} p$.
\end{lemma}

\begin{proof}
    We show the result by induction on the length of $\Pi$.

    \medskip

\begin{enumerate}
    \item $(\Pi$ has length $0$). Then either it consists in a single proof assumption and is a deduction showing $p ; \emptyset \vdash^{\mathcal{+}}_{\mathcal{B}} p$ or in a single refutation assumption is a deduction showing $\emptyset ; p  \vdash^{\mathcal{-}}_{\mathcal{B}} p$ for some $p \in \At$. By Definition \ref{def:dualdeductions}, in the first case we have that $(\Pi)^{\mathbb{D}}$ consists in a single refutation assumption with shape $p$ and shows $\emptyset ; p  \vdash^{\mathcal{-}}_{(\mathcal{B})^{\mathbb{D}}} p$, and in the second that it consists in a single proof assumption with shape $p$ and showns $p ; \emptyset  \vdash^{\mathcal{+}}_{(\mathcal{B})^{\mathbb{D}}} p$, so the proof is already finished.

    \medskip

    \item ($\Pi$ has lenght greater than $0$). Without loss of generality, assume that the last rule applied in it has the following shape:

\begin{prooftree}
    \AxiomC{$[\Gamma^{1}_{\At}]; \llbracket \Delta^{1}_{\At} \rrbracket$}
\noLine
\UnaryInfC{.}
\noLine
\UnaryInfC{.}
\noLine
\UnaryInfC{.}
\UnaryInfC{$p_1$}
\AxiomC{$\ldots$}
\AxiomC{$[\Gamma^{n}_{\At}]; \llbracket \Delta^{n}_{\At} \rrbracket$}
\noLine
\UnaryInfC{.}
\noLine
\UnaryInfC{.}
\noLine
\UnaryInfC{.}
\doubleLine
\UnaryInfC{$p_n$}
\RightLabel{\tiny{$R+$}}
\TrinaryInfC{$p$}
\end{prooftree}
    
\end{enumerate}

Then above its premises there are deductions $\Pi^{1}, \ldots, \Pi^n$ showing $\Theta^{1}_{\At} ; \Sigma^{1}_{\At} \vdash^{+}_{\mathcal{B}}p_{1}$, $\ldots$, $\Theta^{n}_{\At} ; \Sigma^{n}_{\At} \vdash^{-}_{\mathcal{B}}p_{n}$, and $\Pi$ itself shows $\{\Theta^{1}_{\At} - \Gamma^{1}_{\At}\} \cup \ldots \cup \{\Theta^{n}_{\At} - \Gamma^{n}_{\At}\}; \{\Sigma^{1}_{\At} - \Delta^{1}_{\At}\} \cup \ldots \cup \{\Sigma^{n}_{\At} - \Delta^{n}_{\At}\} \vdash^{+}_{\mathcal{B}} p$. The induction hypothesis yields that the deductions $(\Pi^{1})^{\mathbb{D}}, \ldots, (\Pi^n)^\mathbb{D}$ show $\Sigma^{1}_{\At} ; \Theta^{1}_{\At} \vdash^{-}_{(\mathcal{B})^{\mathbb{D}}}p_{1}$, $\ldots$, $\Sigma^{n}_{\At} ; \Theta^{n}_{\At} \vdash^{+}_{(\mathcal{B})^{\mathbb{D}}}p_{n}$. By Definition \ref{def:dualdeductions} we have that $(\Pi)^{\mathbb{D}}$ has the deductions $(\Pi^{1})^{\mathbb{D}}, \ldots, (\Pi^n)^\mathbb{D}$ above its premises and ends with and application of $(R+)^{\mathbb{D}}$, which by Definition \ref{def:dualrules} has the following shape:

\begin{prooftree}
    \AxiomC{$[\Delta^{1}_{\At}];  \llbracket \Gamma^{1}_{\At} \rrbracket$}
\noLine
\UnaryInfC{.}
\noLine
\UnaryInfC{.}
\noLine
\UnaryInfC{.}
\doubleLine
\UnaryInfC{$p_1$}
\AxiomC{$\ldots$}
\AxiomC{$[\Delta^{n}_{\At}] ; \llbracket \Gamma^{n}_{\At} \rrbracket$}
\noLine
\UnaryInfC{.}
\noLine
\UnaryInfC{.}
\noLine
\UnaryInfC{.}
\UnaryInfC{$p_n$}
\doubleLine
\RightLabel{\tiny{$(R+)^{\mathbb{D}}$}}
\TrinaryInfC{$p$}
\end{prooftree}

Since $R+$ was used in $\Pi$ we have $R+ \in \mathcal{B}$, so by Definition \ref{def:dualbases} we conclude $(R+)^{\mathbb{D}} \in (\mathcal{B})^{\mathbb{D}}$, whence since $(\Pi^{1})^{\mathbb{D}}, \ldots, (\Pi^n)^\mathbb{D}$ are deductions in $(\mathcal{B})^{\mathbb{D}}$ we conclude that $(\Pi)^{\mathbb{D}}$ is indeed a deduction showing $\{\Sigma^{1}_{\At} - \Delta^{1}_{\At}\} \cup \ldots \cup \{\Sigma^{n}_{\At} - \Delta^{n}_{\At}\}; \{\Theta^{1}_{\At} - \Gamma^{1}_{\At}\} \cup \ldots \cup \{\Theta^{n}_{\At} - \Gamma^{n}_{\At}\}  \vdash^{-}_{(\mathcal{B})^{\mathbb{D}}} p$. If $\Pi$ ends with a refutation rule $R-$ instead the result can be shown through a similar line of reasoning.

\end{proof}

\begin{proposition} \label{remark}
    $((R*)^{\mathbb{D}} ) ^{\mathbb{D}} = R*$, $((\Pi)^{\mathbb{D}} ) ^{\mathbb{D}} = \Pi$ and $((\mathcal{B})^{\mathbb{D}} ) ^{\mathbb{D}} = \mathcal{B}$.
\end{proposition}

The involutive character of duality is immediate for rules an bases, and in the case of deductions (which is not required for the proofs that follow) it can be proved inductively though a straightforward adaptation of the proof of Proposition \ref{prop:dualsdeductionsaredeductions}.

\begin{corollary}\label{cor:maincorollary}
   $\Vdash^{*}_{\mathcal{B}} p$ iff $\Vdash^{(*)^{\mathbb{D}}}_{(\mathcal{B})^{\mathbb{D}}} p$ for any $p \in \At$.
\end{corollary}

\begin{proof}
    Assume $\Vdash^{+}_{\mathcal{B}} p$. By \ref{c:at+} we conclude that there must be an atomic proof $\Pi$ showing $\vdash^{+}_{\mathcal{B}} p$, hence by Lemma \ref{lemma:fundamentalatomicdualitylemma} we conclude that $(\Pi)^{\mathbb{D}}$ shows  $\vdash^{-}_{(\mathcal{B})^{\mathbb{D}}} p$, whence $\Vdash^{-}_{(\mathcal{B})^{\mathbb{D}}} p$ by \ref{c:at-}. A similar argument shows that if $\Vdash^{-}_{\mathcal{B}} p$ then $\Vdash^{+}_{(\mathcal{B})^{\mathbb{D}}} p$. For the converse, assume $\Vdash^{+}_{(\mathcal{B})^{\mathbb{D}}} p$. By \ref{c:at+} we conclude that there must be an atomic proof $\Pi$ showing $\vdash^{+}_{(\mathcal{B})^{\mathbb{D}}} p$, hence by Lemma \ref{lemma:fundamentalatomicdualitylemma} we conclude that $(\Pi)^{\mathbb{D}}$ shows $\vdash^{-}_{((\mathcal{B})^{\mathbb{D}})^{\mathbb{D}}} p$. But  $((\mathcal{B})^{\mathbb{D}} ) ^{\mathbb{D}} = \mathcal{B}$ by Proposition \ref{remark}, hence by \ref{c:at-} we conclude $\Vdash^{-}_{\mathcal{B}} p$. A similar argument shows that if $\Vdash^{-}_{(\mathcal{B})^{\mathbb{D}}} p$ then $\Vdash^{+}_{\mathcal{B}} p$.
\end{proof}

\begin{definition} \label{def:dualityformulas}
    Given any formula $\chi$, its dual $(\chi)^{\mathbb{D}}$ is defined inductively as follows:

    \begin{enumerate}
        \item $(p)^{\mathbb{D}} = p$, for $p \in \At$;
        \item $(\bot)^{\mathbb{D}} = \top$;
        \item $(\top)^{\mathbb{D}} = \bot$;
        \item $(\phi \lor \psi)^{\mathbb{D}} = (\phi)^{\mathbb{D}} \land (\psi)^{\mathbb{D}}$;
         \item $(\phi \land \psi)^{\mathbb{D}} = (\phi)^{\mathbb{D}} \lor (\psi)^{\mathbb{D}}$;
          \item $(\phi \to \psi)^{\mathbb{D}} = (\psi)^{\mathbb{D}} \mapsfrom (\phi)^{\mathbb{D}}$;
          \item $(\phi \mapsfrom \psi)^{\mathbb{D}} = (\psi)^{\mathbb{D}} \to (\phi)^{\mathbb{D}}$;
    \end{enumerate}
\end{definition}

\begin{lemma}\label{lemma:involutivenessformulas}
    $((\chi)^{\mathbb{D}})^{\mathbb{D}} = \chi$.

    \begin{proof}
    We prove the result by induction on the complexity of $\chi$.

    \begin{enumerate}
        \item $(\chi = p$ for some $p \in \At)$. Since $(p)^{\mathbb{D}} =p$ we have $((p)^{\mathbb{D}})^{\mathbb{D}} = (p)^{\mathbb{D}}$ and $((p)^{\mathbb{D}})^{\mathbb{D}} = p$.

        \medskip

        \item $(\chi = \bot)$. Since $(\bot)^{\mathbb{D}} = \top$ we have $((\bot)^{\mathbb{D}})^{\mathbb{D}} = (\top)^{\mathbb{D}}$ and since $(\top)^{\mathbb{D}} = \bot$ we have $((\bot)^{\mathbb{D}})^{\mathbb{D}} = \bot$.

\medskip
        
 \item $(\chi = \top)$. Since $(\top)^{\mathbb{D}} = \bot$ we have $((\top)^{\mathbb{D}})^{\mathbb{D}} = (\bot)^{\mathbb{D}}$ and since $(\bot)^{\mathbb{D}} = \top$ we have $((\top)^{\mathbb{D}})^{\mathbb{D}} = \top$.

 \medskip

 \item $(\chi = \phi \lor \psi)$. Since $(\phi \lor \psi)^{\mathbb{D}} = (\phi)^{\mathbb{D}} \land (\psi)^{\mathbb{D}}$ we have $((\phi \lor \psi)^{\mathbb{D}})^{\mathbb{D}} = ((\phi)^{\mathbb{D}} \land (\psi)^{\mathbb{D}})^{\mathbb{D}}$. Since $((\phi)^{\mathbb{D}} \land (\psi)^{\mathbb{D}})^{\mathbb{D}} = ((\phi)^{\mathbb{D}})^{\mathbb{D}} \lor ((\psi)^{\mathbb{D}})^{\mathbb{D}}$ and the induction hypothesis yields $((\phi)^{\mathbb{D}})^{\mathbb{D}} \lor ((\psi)^{\mathbb{D}})^{\mathbb{D}} = \phi \lor \psi$ we conclude $((\phi \lor \psi)^{\mathbb{D}})^{\mathbb{D}} = \phi \lor \psi$.

 \medskip

 \item $(\chi = \phi \land \psi)$. Since $(\phi \land \psi)^{\mathbb{D}} = (\phi)^{\mathbb{D}} \lor (\psi)^{\mathbb{D}}$ we have $((\phi \land \psi)^{\mathbb{D}})^{\mathbb{D}} = ((\phi)^{\mathbb{D}} \lor (\psi)^{\mathbb{D}})^{\mathbb{D}}$. Since $((\phi)^{\mathbb{D}} \lor (\psi)^{\mathbb{D}})^{\mathbb{D}} = ((\phi)^{\mathbb{D}})^{\mathbb{D}} \land ((\psi)^{\mathbb{D}})^{\mathbb{D}}$ and the induction hypothesis yields $((\phi)^{\mathbb{D}})^{\mathbb{D}} \land ((\psi)^{\mathbb{D}})^{\mathbb{D}} = \phi \land \psi$ we conclude $((\phi \land \psi)^{\mathbb{D}})^{\mathbb{D}} = \phi \land \psi$.

 \medskip

  \item $(\chi = \phi \to \psi)$. Since $(\phi \to \psi)^{\mathbb{D}} = (\psi)^{\mathbb{D}} \mapsfrom (\phi)^{\mathbb{D}}$ we have $((\phi \to \psi)^{\mathbb{D}})^{\mathbb{D}} = ((\psi)^{\mathbb{D}} \mapsfrom (\phi)^{\mathbb{D}})^{\mathbb{D}}$. Since $((\psi)^{\mathbb{D}} \mapsfrom (\phi)^{\mathbb{D}})^{\mathbb{D}} = ((\phi)^{\mathbb{D}})^{\mathbb{D}} \to ((\psi)^{\mathbb{D}})^{\mathbb{D}}$ and the induction hypothesis yields $((\phi)^{\mathbb{D}})^{\mathbb{D}} \to ((\psi)^{\mathbb{D}})^{\mathbb{D}} = \phi \to \psi$ we conclude $((\phi \to \psi)^{\mathbb{D}})^{\mathbb{D}} = \phi \to \psi$.

   \medskip

  \item $(\chi = \phi \mapsfrom \psi)$. Since $(\phi \mapsfrom \psi)^{\mathbb{D}} = (\psi)^{\mathbb{D}} \to (\phi)^{\mathbb{D}}$ we have $((\phi \mapsfrom \psi)^{\mathbb{D}})^{\mathbb{D}} = ((\psi)^{\mathbb{D}} \to (\phi)^{\mathbb{D}})^{\mathbb{D}}$. Since $((\psi)^{\mathbb{D}} \to (\phi)^{\mathbb{D}})^{\mathbb{D}} = ((\phi)^{\mathbb{D}})^{\mathbb{D}} \mapsfrom ((\psi)^{\mathbb{D}})^{\mathbb{D}}$ and the induction hypothesis yields $((\phi)^{\mathbb{D}})^{\mathbb{D}} \mapsfrom ((\psi)^{\mathbb{D}})^{\mathbb{D}} = \phi \mapsfrom \psi$ we conclude $((\phi \mapsfrom \psi)^{\mathbb{D}})^{\mathbb{D}} = \phi \mapsfrom \psi$.

    \end{enumerate}
    \end{proof}
\end{lemma}

\begin{lemma}\label{lemma:baseinversion}
For any $\mathcal{B}$, if $\mathcal{C} \supseteq (\mathcal{B})^{\mathbb{D}}$ then $(\mathcal{C})^{\mathbb{D}} \supseteq \mathcal{B}$.
\end{lemma}

\begin{proof}
    Pick any $\mathcal{B}$ and any $\mathcal{C} \supseteq (\mathcal{B})^{\mathbb{D}}$. By Definition \ref{def:dualbases} we conclude that, for any rule $R$, $R \in \mathcal{B}$ implies $(R)^{\mathbb{D}} \in (\mathcal{B})^{\mathbb{D}}$, and since  $\mathcal{C} \supseteq (\mathcal{B})^{\mathbb{D}}$ by the definition of extensions also $(R)^{\mathbb{D}} \in \mathcal{C}$. By Definition \ref{def:dualbases} we also conclude that, for any $R$, $R \in \mathcal{C}$ implies $(R)^{\mathbb{D}} \in (\mathcal{C})^{\mathbb{D}}$. Now let $R$ be any rule with $R \in \mathcal{B}$. Then $(R)^{\mathbb{D}} \in \mathcal{C}$ and  $((R)^{\mathbb{D}})^{\mathbb{D}} \in (\mathcal{C})^{\mathbb{D}}$, so since $((R)^{\mathbb{D}})^{\mathbb{D}} = R$ we conclude $R \in (\mathcal{C})^{\mathbb{D}}$, whence by the definition of extensions and arbitrariness of $R$ we obtain  $(\mathcal{C})^{\mathbb{D}} \supseteq \mathcal{B}$.
    
\end{proof}

\begin{theorem}[Weak bilateral semantic harmony]\label{the:semanticharmony}
    $\Vdash^{*}_{\mathcal{B}} \chi$ iff $\Vdash^{(*)^{\mathbb{D}}}_{(\mathcal{B})^{\mathbb{D}}} (\chi)^{\mathbb{D}}$
\end{theorem}

\begin{proof}
    We prove the result by induction on the complexity of $\chi$.

    \begin{enumerate}
        \item $(\chi = p$ for $p \in \At)$.

        \medskip

    \noindent    The result follows immediately from Definition \ref{def:dualityformulas} and Corollary \ref{cor:maincorollary}.

 \medskip
 
   \item $(\chi = \top, * = +)$.

        \medskip

    \noindent    Immediate from \ref{c:top+}, \ref{c:bot-} and the fact that $(\top)^{\mathbb{D}} = \bot$.

     \medskip
 
   \item $(\chi = \bot, * = +)$.

        \medskip

\begin{enumerate}
    \item[$(\Rightarrow)$] Assume $\Vdash^{+}_{\mathcal{B}} \bot$. By \ref{c:bot+} we conclude $\Vdash^{+}_{\mathcal{B}} p$ and $\Vdash^{-}_{\mathcal{B}} p$ for all $p \in \At$. The induction hypothesis yields $\Vdash^{-}_{(\mathcal{B})^{\mathbb{D}}} (p)^{\mathbb{D}}$ and $\Vdash^{+}_{(\mathcal{B})^{\mathbb{D}}} (p)^{\mathbb{D}}$ for all $p \in \At$. By Definition \ref{def:dualityformulas} we have $(p)^{\mathbb{D}} = p$ for every $p \in \At$, so $\Vdash^{-}_{(\mathcal{B})^{\mathbb{D}}} p$ and $\Vdash^{+}_{(\mathcal{B})^{\mathbb{D}}} p$ hold  for all $p \in \At$, whence $\Vdash^{-}_{(\mathcal{B})^{\mathbb{D}}} \top$ by \ref{c:top-}, thence since $(\bot)^{\mathbb{D}} = \top$ we conclude $\Vdash^{-}_{(\mathcal{B})^{\mathbb{D}}} (\bot)^{\mathbb{D}}$.

    \medskip

    \item [$(\Leftarrow)$] Assume $\Vdash^{-}_{(\mathcal{B})^{\mathbb{D}}} (\bot)^{\mathbb{D}}$, which is equivalent to $\Vdash^{-}_{(\mathcal{B})^{\mathbb{D}}} \top$ by Definition \ref{def:dualityformulas}. By \ref{c:top-} we conclude $\Vdash^{+}_{(\mathcal{B})^{\mathbb{D}}} p$ and $\Vdash^{-}_{(\mathcal{B})^{\mathbb{D}}} p$ for all $p \in \At$, so also $\Vdash^{-}_{(\mathcal{B})^{\mathbb{D}}} (p)^{\mathbb{D}}$ and $\Vdash^{+}_{(\mathcal{B})^{\mathbb{D}}} (p)^{\mathbb{D}}$ for all $p \in \At$ by Definition \ref{def:dualityformulas}. The induction hypothesis yields $\Vdash^{+}_{\mathcal{B}} p$ and $\Vdash^{-}_{\mathcal{B}} p$ for all $p \in \At$, whence $\Vdash^{+}_{\mathcal{B}} \bot$ by \ref{c:top+}.
\end{enumerate}

\medskip

    \item $(\chi = \phi \lor \psi, * = +)$.

    \medskip

\begin{enumerate}
    \item[$(\Rightarrow)$]  Assume $\Vdash^{+}_{\mathcal{B}} \phi \lor \psi$. We have to show $\Vdash^{-}_{(\mathcal{B})^{\mathbb{D}}} (\phi)^{\mathbb{D}} \land (\psi)^{\mathbb{D}}$, which by Definition \ref{def:dualityformulas} is equivalent to $\Vdash^{-}_{(\mathcal{B})^{\mathbb{D}}} (\phi \lor \psi)^{\mathbb{D}}$.  Pick any $\mathcal{C} \supseteq (\mathcal{B})^{\mathbb{D}}$ such that $\emptyset ; (\phi)^{\mathbb{D}} \Vdash^{+}_{\mathcal{C}} p$ and $\emptyset ; (\psi)^{\mathbb{D}} \Vdash^{+}_{\mathcal{C}} p$ for some $p \in \At$.  First we show that $ \phi; \emptyset \Vdash^{-}_{(\mathcal{C})^{\mathbb{D}}} p$ and $\psi ; \emptyset \Vdash^{-}_{(\mathcal{C})^{\mathbb{D}}} p$ hold. Pick any $\mathcal{D} \supseteq (\mathcal{C})^{\mathbb{D}}$ such that $\Vdash^{+}_{\mathcal{D}} \phi$. The induction hypothesis yields $\Vdash^{-}_{(\mathcal{D})^{\mathbb{D}}} (\phi)^{\mathbb{D}}$. Since $\mathcal{D} \supseteq (\mathcal{C})^{\mathbb{D}}$ by Lemma \ref{lemma:baseinversion} we have $(\mathcal{D})^{\mathbb{D}} \supseteq \mathcal{C}$, hence since $\emptyset ; (\phi)^{\mathbb{D}} \Vdash^{+}_{\mathcal{C}} p$ and $\Vdash^{-}_{(\mathcal{D})^{\mathbb{D}}} (\phi)^{\mathbb{D}}$ we conclude  $ \Vdash^{+}_{(\mathcal{D})^{\mathbb{D}}} p$ by \ref{c:inf+}. The induction hypothesis yields  $ \Vdash^{-}_{((\mathcal{D})^{\mathbb{D}})^{\mathbb{D}}} (p)^{\mathbb{D}}$, which by Definition \ref{def:dualityformulas} and Proposition \ref{remark} is equivalent to $\Vdash^{-}_{\mathcal{D}} p$. Since $\mathcal{D}$ is an arbitrary extension of $(\mathcal{C})^{\mathbb{D}}$ with $\Vdash^{+}_{\mathcal{D}} \phi$ we conclude $ \phi; \emptyset \Vdash^{-}_{(\mathcal{C})^{\mathbb{D}}} p$ by \ref{c:inf-}. A similar argument yields $ \psi; \emptyset \Vdash^{-}_{(\mathcal{C})^{\mathbb{D}}} p$. Since $\mathcal{C} \supseteq (\mathcal{B})^{\mathbb{D}}$ by Lemma \ref{lemma:baseinversion} we have $(\mathcal{C})^{\mathbb{D}} \supseteq \mathcal{B}$. Since $ \phi; \emptyset \Vdash^{-}_{(\mathcal{C})^{\mathbb{D}}} p$ and $ \psi; \emptyset \Vdash^{-}_{(\mathcal{C})^{\mathbb{D}}} p$ both hold and $\Vdash^{+}_{\mathcal{B}} \phi \lor \psi$ also holds we conclude $\Vdash^{-}_{(\mathcal{C})^{\mathbb{D}}} p$ by \ref{c:or+}, hence $\Vdash^{-}_{(\mathcal{C})^{\mathbb{D}}} (p)^{\mathbb{D}}$ by Definition \ref{def:dualityformulas} and $\Vdash^{+}_{\mathcal{C}} p$ by the induction hypothesis. This means that, for any $\mathcal{C} \supseteq (\mathcal{B})^{\mathbb{D}}$ and any $p \in \At$, if $\emptyset ; (\phi)^{\mathbb{D}} \Vdash^{+}_{\mathcal{C}} p$ and $\emptyset ; (\psi)^{\mathbb{D}} \Vdash^{+}_{\mathcal{C}} p$ then $\Vdash^{+}_{C} p$. A similar argument shows that, for any $\mathcal{C} \supseteq (\mathcal{B})^{\mathbb{D}}$ and any $p \in \At$, if $\emptyset ; (\phi)^{\mathbb{D}} \Vdash^{-}_{\mathcal{C}} p$ and $\emptyset ; (\psi)^{\mathbb{D}} \Vdash^{-}_{\mathcal{C}} p$ then $\Vdash^{-}_{C} p$, thence by \ref{c:and-} and arbitrariness of $\mathcal{C}$ we finally conclude $\Vdash^{-}_{(\mathcal{B})^{\mathbb{D}}} (\phi)^{\mathbb{D}} \land (\psi)^{\mathbb{D}}$.

 \medskip

   \item[$(\Rightarrow)$]  Assume $\Vdash^{-}_{(\mathcal{B})^{\mathbb{D}}} (\phi \lor \psi)^{\mathbb{D}}$. By Definition \ref{def:dualityformulas} we conclude $\Vdash^{-}_{(\mathcal{B})^{\mathbb{D}}} (\phi)^{\mathbb{D}} \land (\psi)^{\mathbb{D}}$, and we have to show $\Vdash^{+}_{\mathcal{B}} \phi \lor \psi$. Pick any $\mathcal{C} \supseteq \mathcal{B}$ such that $\phi ; \emptyset \Vdash^{+}_{\mathcal{C}} p$ and $\psi; \emptyset \Vdash^{+}_{\mathcal{C}} p$ for some $p \in \At$. First we show that $\emptyset ; (\phi)^{\mathbb{D}} \Vdash^{-}_{(\mathcal{C})^{\mathbb{D}}} p$ and $\emptyset ; (\psi)^{\mathbb{D}} \Vdash^{-}_{(\mathcal{C})^{\mathbb{D}}} p$ hold. Pick any $\mathcal{D} \supseteq (\mathcal{C})^{\mathbb{D}}$ such that $\Vdash^{-}_{\mathcal{D}} (\phi)^{\mathbb{D}}$. The induction hypothesis yields $\Vdash^{+}_{(\mathcal{D})^{\mathbb{D}}} ((\phi)^{\mathbb{D}})^{\mathbb{D}}$, which by Lemma \ref{lemma:involutivenessformulas} is equivalent to $\Vdash^{+}_{(\mathcal{D})^{\mathbb{D}}} \phi$.  Since $\mathcal{D} \supseteq (\mathcal{C})^{\mathbb{D}}$ by Lemma \ref{lemma:baseinversion} we have $(\mathcal{D})^{\mathbb{D}} \supseteq \mathcal{C}$, hence since $\phi ; \emptyset \Vdash^{+}_{\mathcal{C}} p$ and $ \Vdash^{+}_{(\mathcal{D})^{\mathbb{D}}} \phi$  we conclude  $ \Vdash^{+}_{(\mathcal{D})^{\mathbb{D}}} p$ by \ref{c:inf+}.  The induction hypothesis yields  $ \Vdash^{-}_{((\mathcal{D})^{\mathbb{D}})^{\mathbb{D}}} (p)^{\mathbb{D}}$, which by Definition \ref{def:dualityformulas} and Proposition \ref{remark} is equivalent to $\Vdash^{-}_{\mathcal{D}} p$.   Since $\mathcal{D}$ is an arbitrary extension of $(\mathcal{C})^{\mathbb{D}}$ with $\Vdash^{-}_{\mathcal{D}} (\phi)^{\mathbb{D}}$ we conclude $ \emptyset; (\phi)^{\mathbb{D}} \Vdash^{-}_{(\mathcal{C})^{\mathbb{D}}} p$ by \ref{c:inf-}. A similar argument yields $ \emptyset; (\psi)^{\mathbb{D}} \Vdash^{-}_{(\mathcal{C})^{\mathbb{D}}} p$. Since $\mathcal{C \supseteq \mathcal{B}}$ and $\mathcal{B} = ((\mathcal{B})^{\mathbb{D}})^{\mathbb{D}}$ we conclude $\mathcal{C} \supseteq ((\mathcal{B})^{\mathbb{D}})^{\mathbb{D}}$, hence by Lemma \ref{lemma:baseinversion} also $(\mathcal{C})^{\mathbb{D}} \supseteq (\mathcal{B})^{\mathbb{D}}$.  Since $\emptyset ; (\phi)^{\mathbb{D}} \Vdash^{-}_{(\mathcal{C})^{\mathbb{D}}} p$ and $\emptyset ; (\psi)^{\mathbb{D}} \Vdash^{-}_{(\mathcal{C})^{\mathbb{D}}} p$ both hold and $\Vdash^{-}_{(\mathcal{B})^{\mathbb{D}}} (\phi)^{\mathbb{D}} \land (\psi)^{\mathbb{D}}$ holds we conclude $\Vdash^{-}_{(\mathcal{C})^{\mathbb{D}}} p$ by \ref{c:and-}, hence $\Vdash^{-}_{(\mathcal{C})^{\mathbb{D}}} (p)^{\mathbb{D}}$ by Definition \ref{def:dualityformulas} and $\Vdash^{+}_{\mathcal{C}} p$ by the induction hypothesis.  This means that, for any $\mathcal{C} \supseteq \mathcal{B}$ and any $p \in \At$, if $\phi ; \emptyset \Vdash^{+}_{\mathcal{C}} p$ and $\psi; \emptyset \Vdash^{+}_{\mathcal{C}} p$ then $\Vdash^{+}_{C} p$. A similar argument shows that, for any $\mathcal{C} \supseteq \mathcal{B}$ and any $p \in \At$, if $\phi ; \emptyset \Vdash^{-}_{\mathcal{C}} p$ and $\psi; \emptyset \Vdash^{-}_{\mathcal{C}} p$ then $\Vdash^{-}_{\mathcal{C}} p$, thence by \ref{c:or+} and arbitrariness of $\mathcal{C}$ we finally conclude $\Vdash^{+}_{\mathcal{B}} \phi \lor \psi$.
\end{enumerate}

\medskip

  \item $(\chi = \phi \land \psi, * = +)$.
  \medskip

\begin{enumerate}
 \item[$(\Rightarrow)$] Assume $\Vdash^{+}_{\mathcal{B}} \phi \land \psi$. Then by \ref{c:and+} we have $\Vdash^{+}_{\mathcal{B}} \phi$ and $\Vdash^{+}_{\mathcal{B}} \psi$. The induction hypothesis yields $\Vdash^{-}_{(\mathcal{B})^{\mathbb{D}}} (\phi)^{\mathbb{D}}$ and $\Vdash^{-}_{(\mathcal{B})^{\mathbb{D}}} (\psi)^{\mathbb{D}}$, hence $\Vdash^{-}_{(\mathcal{B})^{\mathbb{D}}} (\phi)^{\mathbb{D}} \lor (\psi)^{\mathbb{D}}$ by \ref{c:or-}, whence $\Vdash^{-}_{(\mathcal{B})^{\mathbb{D}}} (\phi \land \psi)^{\mathbb{D}}$ by Definition \ref{def:dualityformulas}.

 \medskip

  \item[$(\Leftarrow)$] Assume $\Vdash^{-}_{(\mathcal{B})^{\mathbb{D}}} (\phi \land \psi)^{\mathbb{D}}$. Then $\Vdash^{-}_{(\mathcal{B})^{\mathbb{D}}} (\phi)^{\mathbb{D}} \lor (\psi)^{\mathbb{D}}$  by Definition \ref{def:dualityformulas}, hence $\Vdash^{-}_{(\mathcal{B})^{\mathbb{D}}} (\phi)^{\mathbb{D}}$ and $\Vdash^{-}_{(\mathcal{B})^{\mathbb{D}}} (\psi)^{\mathbb{D}}$ by \ref{c:or-}. The induction hypothesis yields $\Vdash^{+}_{\mathcal{B}} \phi$ and $\Vdash^{+}_{\mathcal{B}} \psi$, thence by \ref{c:and+} we conclude $\Vdash^{-}_{(\mathcal{B})^{\mathbb{D}}} (\phi \land \psi)^{\mathbb{D}}$.
  
\end{enumerate}

\medskip

 \item $(\chi = \phi \to \psi, * = +)$.

    \medskip

\begin{enumerate}
    \item[$(\Rightarrow)$] Assume $\Vdash^{+}_{\mathcal{B}} \phi \to \psi$. We have to show  $\Vdash^{-}_{(\mathcal{B})^{\mathbb{D}}} (\phi \to \psi)^{\mathbb{D}}$, which is equivalent to $\Vdash^{-}_{(\mathcal{B})^{\mathbb{D}}} (\psi)^{\mathbb{D}} \mapsfrom (\phi)^{\mathbb{D}}$ by Definition \ref{def:dualityformulas}. Pick any $\mathcal{C} \supseteq (\mathcal{B})^{\mathbb{D}}$ such that $\Vdash^{-}_{\mathcal{C}} (\phi)^{\mathbb{D}}$. The induction hypothesis yields $\Vdash^{+}_{(\mathcal{C})^{\mathbb{D}}} ((\phi)^{\mathbb{D}})^{\mathbb{D}}$, which by Lemma \ref{lemma:involutivenessformulas} is equivalent to $\Vdash^{+}_{(\mathcal{C})^{\mathbb{D}}} \phi$. Since $\mathcal{C} \supseteq (\mathcal{B})^{\mathbb{D}}$ by Lemma \ref{lemma:baseinversion} we conclude $(\mathcal{C})^{\mathbb{D}} \supseteq \mathcal{B}$. Since $\Vdash^{+}_{\mathcal{B}} \phi \to \psi$ we conclude $\phi; \emptyset \Vdash^{+}_{\mathcal{B}} \psi$ by \ref{c:to+}. Since $\phi; \emptyset \Vdash^{+}_{\mathcal{B}} \psi$, $(\mathcal{C})^{\mathbb{D}} \supseteq \mathcal{B}$ and $\Vdash^{+}_{(\mathcal{C})^{\mathbb{D}}} \phi$  we conclude $\Vdash^{+}_{(\mathcal{C})^{\mathbb{D}}} \psi$ by \ref{c:inf+}. The inductive hypothesis yields $\Vdash^{-}_{((\mathcal{C})^{\mathbb{D}})^{\mathbb{D}}} (\psi)^{\mathcal{D}}$, which by Proposition \ref{remark} is equivalent to $\Vdash^{-}_{\mathcal{C}} (\psi)^{\mathcal{D}}$. But then for every $\mathcal{C} \supseteq (\mathcal{B})^{\mathbb{D}}$ with $\Vdash^{-}_{\mathcal{C}} (\phi)^{\mathbb{D}}$ we have $\Vdash^{-}_{\mathcal{C}} (\psi)^{\mathbb{D}}$,  so we conclude $\emptyset ; (\phi)^{\mathbb{D}} \Vdash^{-}_{(\mathcal{B})^{\mathbb{D}}} (\psi)^{\mathbb{D}}$ by \ref{c:inf-}, thence $\Vdash^{-}_{(\mathcal{B})^{\mathbb{D}}} (\psi)^{\mathcal{D}} \mapsfrom (\phi)^{\mathbb{D}}$ by \ref{c:mapsfrom-}.

    \medskip

    \item[$(\Leftarrow)$] Assume $\Vdash^{-}_{(\mathcal{B})^{\mathbb{D}}} (\phi \to \psi)^{\mathbb{D}}$, which is equivalent to $\Vdash^{-}_{(\mathcal{B})^{\mathbb{D}}} (\psi)^{\mathbb{D}} \mapsfrom (\phi)^{\mathbb{D}}$ by Definition \ref{def:dualityformulas} . We have to show $\Vdash^{+}_{\mathcal{B}} \phi \to \psi$. Pick any $\mathcal{C} \supseteq \mathcal{B}$ such that $\Vdash^{+}_{\mathcal{C}} \phi$. The induction hypothesis yields $\Vdash^{-}_{(\mathcal{C})^{\mathbb{D}}} (\phi)^{\mathbb{D}}$. Since $\mathcal{C} \supseteq \mathcal{B}$ and $\mathcal{B} =((\mathcal{B})^{\mathbb{D}})^{\mathbb{D}}$ by Proposition \ref{remark}, we have $\mathcal{C} \supseteq ((\mathcal{B})^{\mathbb{D}})^{\mathbb{D}}$, hence $(\mathcal{C})^{\mathbb{D}} \supseteq (\mathcal{B})^{\mathbb{D}}$ by Lemma \ref{lemma:baseinversion}. Since $\Vdash^{-}_{(\mathcal{B})^{\mathbb{D}}} (\psi)^{\mathbb{D}} \mapsfrom (\phi)^{\mathbb{D}}$ we conclude $\emptyset ;(\phi)^{\mathbb{D}} \Vdash^{-}_{(\mathcal{B})^{\mathbb{D}}} (\psi)^{\mathbb{D}}$ by \ref{c:mapsfrom-}. Since $\emptyset ;(\phi)^{\mathbb{D}} \Vdash^{-}_{(\mathcal{B})^{\mathbb{D}}} (\psi)^{\mathbb{D}}$, $(\mathcal{C})^{\mathbb{D}} \supseteq (\mathcal{B})^{\mathbb{D}}$ and $\Vdash^{-}_{(\mathcal{C})^{\mathbb{D}}} (\phi)^{\mathbb{D}}$ we conclude $\Vdash^{-}_{(\mathcal{C})^{\mathbb{D}}} (\psi)^{\mathbb{D}}$ by \ref{c:inf-}. The inductive hypothesis yields $\Vdash^{+}_{((\mathcal{C})^{\mathbb{D}})^{\mathbb{D}}} ((\psi)^{\mathbb{D}})^{\mathbb{D}}$, which by Proposition \ref{remark} and Lemma \ref{lemma:involutivenessformulas} is equivalent to $\Vdash^{+}_{\mathcal{C}} \psi$. But then for every $\mathcal{C} \supseteq \mathcal{B}$ with $\Vdash^{+}_{\mathcal{C}} \phi$ we have $\Vdash^{+}_{\mathcal{C}} \psi$, hence we conclude $\phi ; \emptyset \Vdash^{+}_{\mathcal{B}} \psi$ by \ref{c:inf-}, thence $\Vdash^{+}_{\mathcal{B}} \phi \to \psi$ by \ref{c:to+}.
\end{enumerate}

\medskip

  \item $(\chi = \phi \mapsfrom \psi, * = +)$.
  \medskip

\begin{enumerate}
 \item[$(\Rightarrow)$] Assume $\Vdash^{+}_{\mathcal{B}} \phi \mapsfrom \psi$. Then by \ref{c:mapsfrom+} we have $\Vdash^{+}_{\mathcal{B}} \phi$ and $\Vdash^{-}_{\mathcal{B}} \psi$. The induction hypothesis yields $\Vdash^{-}_{(\mathcal{B})^{\mathbb{D}}} (\phi)^{\mathbb{D}}$ and $\Vdash^{+}_{(\mathcal{B})^{\mathbb{D}}} (\psi)^{\mathbb{D}}$, hence $\Vdash^{-}_{(\mathcal{B})^{\mathbb{D}}} (\psi)^{\mathbb{D}} \to (\phi)^{\mathbb{D}}$ by \ref{c:to-}, whence $\Vdash^{-}_{(\mathcal{B})^{\mathbb{D}}} (\phi \to \psi)^{\mathbb{D}}$ by Definition \ref{def:dualityformulas}.

 \medskip

  \item[$(\Leftarrow)$] Assume $\Vdash^{-}_{(\mathcal{B})^{\mathbb{D}}} (\phi \mapsfrom \psi)^{\mathbb{D}}$. Then $\Vdash^{-}_{(\mathcal{B})^{\mathbb{D}}} (\psi)^{\mathbb{D}} \to (\phi)^{\mathbb{D}}$  by Definition \ref{def:dualityformulas}, hence $\Vdash^{+}_{(\mathcal{B})^{\mathbb{D}}} (\psi)^{\mathbb{D}}$ and $\Vdash^{-}_{(\mathcal{B})^{\mathbb{D}}} (\phi)^{\mathbb{D}}$ by \ref{c:to-}. The induction hypothesis yields $\Vdash^{+}_{\mathcal{B}} \phi$ and $\Vdash^{-}_{\mathcal{B}} \psi$, thence by \ref{c:mapsfrom+} we conclude $\Vdash^{+}_{\mathcal{B}} \phi \mapsfrom \psi$.
  
\end{enumerate}

\medskip

\item The inductive steps for $(\chi = \bot, *= -)$, $(\chi = \top, *= -)$, $(\chi = \phi \lor \psi, * = -)$, $(\chi = \phi \land \psi, * = -)$, $(\chi = \phi \to \psi, * = -)$ and $(\chi = \phi \mapsfrom \psi, * = -)$  are symmetrical with the ones for  $(\chi = \top, *= +)$, $(\chi = \bot, *= +)$, $(\chi = \phi \land \psi, * = +)$, $(\chi = \phi \lor \psi, * = +)$, $(\chi = \phi \mapsfrom \psi, * = +)$ and $(\chi = \phi \to \psi, * = +)$, respectively, so their proofs are omited. 
   \end{enumerate}

\end{proof}

\begin{definition}\label{def:dualitysets} 
  For any set $\Gamma$ of formulas we define  $(\Gamma)^{\mathbb{D}} = \{(\chi)^{\mathbb{D}} | \chi \in \Gamma \}$.

  %Notice that by Lemma \ref{lemma:involutivenessformulas}, $((\Gamma)^\mathbb{D})^{\mathbb{D}} = \Gamma$.
\end{definition}

\begin{corollary}[Strong bilateral semantic harmony]\label{cor:strongharmony}
    $\Gamma ; \Delta \Vdash^{*}_{\mathcal{B}} \chi$ iff $(\Delta)^{\mathbb{D}}; (\Gamma)^{\mathbb{D}} \Vdash^{(*)^{\mathbb{D}}}_{(\mathcal{B})^{\mathbb{D}}} (\chi)^{\mathbb{D}}$.
\end{corollary}

\begin{proof}
    For $\Gamma \cup \Delta$ empty the result follows from Theorem \ref{the:semanticharmony}. For non-empty $\Gamma \cup \Delta$, assume  $\Gamma ; \Delta \Vdash^{*}_{\mathcal{B}} \chi$. Pick any extension $\mathcal{C}$ of $(\mathcal{B})^{\mathbb{D}}$ with $\Vdash^{+}_{C} (\psi)^{\mathbb{D}}$ for all $(\psi)^{\mathbb{D}} \in (\Delta)^{\mathbb{D}}$ and $\Vdash^{-}_{C} (\phi)^{\mathbb{D}}$ for all $(\phi)^{\mathbb{D}} \in (\Gamma)^{\mathbb{D}}$. Theorem \ref{the:semanticharmony} yields $\Vdash^{+}_{(C)^{\mathbb{D}}} ((\phi)^{\mathbb{D}})^{\mathbb{D}}$ and $\Vdash^{-}_{(C)^{\mathbb{D}}} ((\psi)^{\mathbb{D}})^{\mathbb{D}}$ for all $(\psi)^{\mathbb{D}} \in (\Delta)^{\mathbb{D}}$ and all $(\phi)^{\mathbb{D}} \in (\Gamma)^{\mathbb{D}}$, which by Lemma \ref{lemma:involutivenessformulas} and Definition \ref{def:dualitysets} means that $\Vdash^{+}_{(C)^{\mathbb{D}}} \phi$ holds for all $\phi \in \Gamma$ and $\Vdash^{-}_{(C)^{\mathbb{D}}} \psi$ holds for all $\psi \in \Delta$. Since $\mathcal{C} \supseteq (\mathcal{B})^{\mathbb{D}}$ we conclude $(\mathcal{C})^{\mathbb{D}} \supseteq \mathcal{B}$ by Lemma \ref{lemma:baseinversion}, hence since $\Gamma ; \Delta \Vdash^{*}_{\mathcal{B}} \chi$ and both $\Vdash^{+}_{(C)^{\mathbb{D}}} \phi$ holds for all $\phi \in \Gamma$ and $\Vdash^{-}_{(C)^{\mathbb{D}}} \psi$ hold for all $\psi \in \Delta$, by either \ref{c:inf+} or \ref{c:inf-} we conclude $\Vdash^{*}_{(\mathcal{C})^{\mathbb{D}}} \chi$. Theorem \ref{the:semanticharmony} thus yields $\Vdash^{(*)^{\mathbb{D}}}_{((\mathcal{C})^{\mathbb{D}})^{\mathbb{D}}} (\chi)^{\mathbb{D}}$, so by Proposition \ref{remark} we have $\Vdash^{(*)^{\mathbb{D}}}_{\mathcal{C}} (\chi)^{\mathbb{D}}$. Since $\mathcal{C}$ was an arbitrary extension of $(\mathcal{B})^{\mathbb{D}}$ with  $\Vdash^{+}_{C} (\psi)^{\mathbb{D}}$ for all $(\psi)^{\mathbb{D}} \in (\Delta)^{\mathbb{D}}$ and $\Vdash^{-}_{C} (\phi)^{\mathbb{D}}$ for all $(\phi)^{\mathbb{D}} \in (\Gamma)^{\mathbb{D}}$ (which are the totality of formulas in $(\Gamma)^{\mathbb{D}}$ and $(\Delta)^{\mathbb{D}}$ due to Definition \ref{def:dualitysets}) and we have shown $\Vdash^{(*)^{\mathbb{D}}}_{\mathcal{C}} (\chi)^{\mathbb{D}}$ we conclude $(\Delta)^{\mathbb{D}}; (\Gamma)^{\mathbb{D}} \Vdash^{(*)^{\mathbb{D}}}_{(\mathcal{B})^{\mathbb{D}}} (\chi)^{\mathbb{D}}$ by either \ref{c:inf+} or \ref{c:inf-}. For the  converse, assume $(\Delta)^{\mathbb{D}}; (\Gamma)^{\mathbb{D}} \Vdash^{(*)^{\mathbb{D}}}_{(\mathcal{B})^{\mathbb{D}}} (\chi)^{\mathbb{D}}$. Pick any $\mathcal{C} \supseteq \mathcal{B}$ with $\Vdash^{+}_{C} \phi$ for all $\phi \in \Gamma$ and $\Vdash^{-}_{\mathcal{C}} \psi$ for all $\psi \in \Delta$. Theorem \ref{the:semanticharmony} yields $\Vdash^{-}_{(C)^{\mathbb{D}}} (\phi)^{\mathbb{D}}$ for all $\phi \in \Gamma$ and $\Vdash^{+}_{(\mathcal{C})^{\mathbb{D}}} (\psi)^{\mathbb{D}}$ for all $\psi \in \Delta$. Since $\mathcal{C} \supseteq \mathcal{B}$ and $\mathcal{B} = ((\mathcal{B)^{\mathbb{D}}})^{\mathbb{D}}$ we have $\mathcal{C} \supseteq ((\mathcal{B)^{\mathbb{D}}})^{\mathbb{D}}$, so also $(\mathcal{C})^\mathbb{D} \supseteq (\mathcal{B)^{\mathbb{D}}}$ by Lemma \ref{lemma:baseinversion}. Since $(\Delta)^{\mathbb{D}}; (\Gamma)^{\mathbb{D}} \Vdash^{(*)^{\mathbb{D}}}_{(\mathcal{B})^{\mathbb{D}}} (\chi)^{\mathbb{D}}$, $(\mathcal{C})^\mathbb{D} \supseteq (\mathcal{B)^{\mathbb{D}}}$ and both $\Vdash^{-}_{(C)^{\mathbb{D}}} (\phi)^{\mathbb{D}}$ for all $\phi \in \Gamma$ and $\Vdash^{+}_{(\mathcal{C})^{\mathbb{D}}} (\psi)^{\mathbb{D}}$ for all $\psi \in \Delta$ (thus for the totality of formulas  $(\phi)^{\mathbb{D}}$ in $ (\Gamma)^{\mathbb{D}}$ and $(\psi)^{\mathbb{D}} $ in $ (\Delta)^{\mathbb{D}}$ by Definition \ref{def:dualitysets}) we conclude $\Vdash^{(*)^{\mathbb{D}}}_{\mathcal{(C)^{\mathbb{D}}}} (\chi)^{\mathbb{D}}$ by either \ref{c:inf+} or \ref{c:inf-}, whence Theorem \ref{the:semanticharmony} yields $\Vdash^{*}_{\mathcal{C}} \chi$. Since $\mathcal{C}$ is an arbitrary extension of $\mathcal{B}$ with $\Vdash^{+}_{C} \phi$ for all $\phi \in \Gamma$ and $\Vdash^{-}_{\mathcal{C}} \psi$ for all $\psi \in \Delta$ we finally conclude  $\Gamma ; \Delta \Vdash^{*}_{\mathcal{B}} \chi$ by either \ref{c:inf+} or \ref{c:inf-}.

    %Pick any extension $\mathcal{C}$ of $(\mathcal{B})^{\mathbb{D}}$ with $\Vdash^{+}_{C} \phi$ for all $\phi \in (\Delta)^{\mathbb{D}}$ and $\Vdash^{-}_{C} \psi$ for all $\psi \in (\Gamma)^{\mathbb{D}}$. Theorem \ref{the:semanticharmony} yields 

    %Pick any extension $\mathcal{C}$ of $(\mathcal{B})^{\mathbb{D}}$ with $\Vdash^{+}_{C} (\phi)^{\mathbb{D}}$ for all $(\phi)^{\mathbb{D}} \in (\Delta)^{\mathbb{D}}$ and $\Vdash^{-}_{C} (\psi)^{\mathbb{D}}$ for all $(\psi)^{\mathbb{D}} \in (\Gamma)$
\end{proof}

%(\mathcal{C})^{\mathbb{D}}

Bilateral semantic harmony comes both in a \textit{weak} version, which pertains only to proofs and refutations, and a \textit{strong} version, which also takes into account derivability and entailment. Theorem \ref{the:semanticharmony} and Corollary \ref{cor:strongharmony} shows that $\Bint$ satisfy both versions. 

From a technical perspective, satisfaction of bilateral semantic harmony is desirable because it enables simple semantic embeddings between fragments of the language. Refutation conditions of a formula $\phi$ in $\mathcal{B}$ can be restated in terms of proof conditions for its dual $(\phi)^{\mathbb{D}}$ in $(\mathcal{B})^{\mathbb{D}}$, the same holding for the proof conditions of $\phi$ and refutation conditions of $(\phi)^{\mathbb{D}}$, so the proof fragment of the language can always be embedded into the refutation fragment (and vice-versa). Weak bilateral harmony enables embeddings preserving demonstrability, whereas strong bilateral harmony enables the preservation of derivability. This can in turn be combined with other embeddings to show even stronger results. For instance, it could be used to semantically prove the embedding results in \cite{Wan110.1093/logcom/ext035}, which show that $\Bint$ can be embedded in its entirety in either the $\mapsfrom - $free proof fragment or the $\to - $ free refutation fragment of $\BintN$ (hence also into both intuitionistic and dual-intuitionistic language). The result would immediately follow after proofs of co-implications were embedded into proofs of other connectives and refutations of implications were embedded into refutations of other connectives.

Still on the technical side, the results allows us to infer the refutation conditions of a logical connective from its proof conditions (and vice-versa), provided we are looking for the strongest opposite conditions harmonic w.r.t. the original conditions\footnote{The inversion principle cannot be applied, for instance, to modal defitinions, in which there is a fundamental imbalance between the strength of introduction and elimination rules.}. Notice also that the semantic principles are not sensitive to the language of the logic, as $\Bint$ does not really contains the dual of implication (strictly speaking, the dual of a formula $\phi \to \psi$ should be a formula $\phi \to^{\circ} \psi$ absent in the language \cite{francez2014bilateralism}, but we are able to circumvent this because $\phi \to^{\circ}\psi$ is equivalent to $\psi \mapsfrom \phi$). Once either the proof or refutation clauses have been defined for a language, the remaining ones can be obtained via bilateral semantic harmony, showing that it is indeed merely a semantic restatement of the syntactically defined horizontal inversion principle.

%From a philosophical standpoint, semantic harmony is important both due to the existence of arguments to the effect that the horizontal inversion principle is a requisite of meaning-conferring bilateral definitions \cite{francez2014bilateralism} and to what it tells us about bilateral proofs and refutations in general. Base-extension semantics shows that logical validity is reducible to base-validity, which in turn is reducible do derivability in particular bases. Bilateral base-extension semantics also shows that logical proofs and refutations are ultimately reducible to proofs and refutations in particular bases, and since proofs and refutations are objects with similar structure but opposite assertoric forces we are now allowed to use the inner structure of bases to define opposite bases. Since the proofs of a base are isomorphic to the refutations in its dual base and vice-versa, the very concept of opposition of assertoric forces becomes characterizable in terms of this structural isomorphism at the atomic level, whence if a similar isomorphism are demanded of semantic clauses and the rules they represent the isomorphism naturally carries over to logical connectives and their opposites.

From a philosophical standpoint, bilateral semantic harmony is important both due to the existence of arguments to the effect that the horizontal inversion principle is a requisite of meaning-conferring bilateral definitions \cite{francez2014bilateralism} and to what it tells us about bilateral semantics in general. The isomorphisms observed between proofs of formulas and refutations of their duals (and vice-versa) in well-behaved bilateral logics is also observed between atomic deductions and their duals, with a clear impact on how proof and refutation support operates in bases and dual bases. The very possibility of dualizing all of the basic elements of the semantics through such proof-theoretic isomorphisms reveals their semantic importance. Moreover, since such an isomorphism is now a characteristic property of the basic semantics, the meaning-conferring character of the horizontal inversion principle may now be interpreted as the fact that the requirements it imposes on proof and refutation rules are harmony-preserving.

\section{Conclusion}

Bilateral base-extension semantics extends traditional base-extension semantics by allowing atomic bases to contain rules capable of producing independent proofs and refutations of atomic formulas. Atomic proofs and refutations can then be extended to non-atomic formulas through semantic clauses as usual. From the fact that rules present in the natural deduction system $\BintNA$ but absent in the deductively equivalent $\BintN$ must be explicitly incorporated in the semantics for it to function properly we infer that, due to differences in how syntactic and semantic definitions are recursively defined, a proof system does not always explicitly represent its underlying semantic proof notion, a point not yet raised in the literature. Furthermore, the added structure of bilateral base-extension semantics allows the definition of duality notions for rules, deduction and bases, which then allows the proof-theoretic horizontal inversion principle to be restated as the principle of bilateral semantic harmony.

 Traditional base-extension semantics shows that logical validity is reducible to base-validity, which in turn is reducible do derivability in particular bases. Bilateral base-extension semantics also shows that logical proofs and refutations are ultimately reducible to proofs and refutations in bases. Since proofs and refutations are objects with similar structure but opposite assertoric forces, we are now allowed to use the proof-theoretic structure of bases to define their opposite base. Since the proofs of a base are isomorphic to the refutations in its dual base and vice-versa, the very concept of opposition of assertoric forces becomes characterizable in terms of this structural isomorphism at the atomic level, whence if a similar isomorphism is demanded of the semantic clauses (or the rules they represent) the isomorphism naturally carries over to logical connectives and their opposites. As such, bilateral proof-theoretic harmony may now be interpreted as the requirement that proof and refutation rules for connectives preserve the semantic harmony already present at the core level of the semantics.

 \bigskip

\noindent \textbf{Acknowledgements.} Barroso-Nascimento is supported by the Leverhulme grant RPG-2024-196. Osório is supported by the Portuguese Foundation for Science and Technology grant 2024.00633.BD. Both authors would like to thank Prof. Elaine Pimentel for her invaluable support and feedback during the writing of this paper.

\newpage


\bibliography{references}% common bib file

\end{document}